\documentclass[11pt]{article}

\usepackage[linesnumbered,ruled]{algorithm2e}

\usepackage{jhtitle}
\usepackage{geometry}
\usepackage{amsbsy,amsmath,amsthm,amssymb,graphicx}
\usepackage{natbib}
\usepackage{hyperref}
\usepackage{enumitem}
\usepackage{pstricks}
\usepackage{subcaption}
\usepackage{tcolorbox}

\newcommand{\subscript}[2]{$#1 _ #2$}

\newcommand{\df}{\mathrm{d}}
\def\baro{\vskip  .2truecm\hfill \hrule height.5pt \vskip  .2truecm}
\def\barba{\vskip -.1truecm\hfill \hrule height.5pt \vskip .4truecm}

\newtheorem{theorem}{Theorem}
\newtheorem{lemma}[theorem]{Lemma}

\newtheorem{remark}[theorem]{Remark}
\newtheorem{proposition}[theorem]{Proposition}

\newcommand{\rb}{{\bar{r}}}
\newcommand{\X}{{\mathsf{X}}}

\newcommand{\norm}[1]{\lVert#1\rVert}

\title{\bf Dimension free convergence rates for Gibbs samplers for
  Bayesian linear mixed models}

\author{Zhumengmeng Jin and James P. Hobert \\ Department of Statistics
  \\ University of Florida}

\date{August 2021}

\begin{document}
\maketitle

\begin{abstract}
The emergence of big data has led to a growing interest in so-called
convergence complexity analysis, which is the study of how the
convergence rate of a Monte Carlo Markov chain (for an intractable
Bayesian posterior distribution) scales as the underlying data set
grows in size.  Convergence complexity analysis of \textit{practical}
Monte Carlo Markov chains on continuous state spaces is quite
challenging, and there have been very few successful analyses of such
chains.  One fruitful analysis was recently presented by
\citet{qin:hobe:2021a}, who studied a Gibbs sampler for a simple
Bayesian random effects model.  These authors showed that, under
regularity conditions, the geometric convergence rate of this Gibbs
sampler converges to zero as the data set grows in size.
It is shown herein that similar behavior is
exhibited by Gibbs samplers for more general Bayesian models that
possess both random effects and traditional continuous covariates, the
so-called mixed models.  The analysis employs the Wasserstein-based
techniques introduced by \citet{qin:hobe:2021a}.
\end{abstract}

\section{Introduction}
\label{sec:intro}

Markov chain Monte Carlo (MCMC) methods have revolutionized the
exploration of intractable probability distributions (see,
  e.g., \citep{diac:2009}).  This has affected many areas of science, but
none more than Bayesian statistics.  Indeed, MCMC allows for the
exploration of intractable high dimensional posterior distributions
that would have remained virtually impenetrable without these tools.
The revolution has not come without a cost, however.  Employing an
MCMC algorithm in a principled manner requires information about the
rate at which the underlying Markov chain converges to its target
distribution \citep{fleg:hara:jone:2008}, and such information is
notoriously difficult to ascertain.  Substantial progress has been
made over the last several decades, both on the development of general
techniques to get bounds on convergence rates (mostly based on drift
and minorization conditions), and on the application of said
techniques to specific Monte Carlo Markov chains.  Unfortunately, most
of these results are not sharp enough to be useful in the emerging
area of \textit{convergence complexity analysis}, which is the (big
data driven) study of how the convergence rate of a Monte Carlo Markov
chain for a Bayesian problem scales with the size of the underlying
data set \citep{raja:spar:2015,qin:hobe:2021b}.  In fact, it has
become clear that convergence complexity analysis requires new theory
(see, e.g.,\citep{durm:moul:2015,hair:matt:sche:2011,qin:hobe:2021a,yang:rose:2019}).
 \citet{qin:hobe:2021a} provide a new technique that entails bounding
the \textit{Wasserstein} distance to stationarity, and then converting
the Wasserstein bounds into total variation bounds.  These authors
apply their method to a Gibbs sampler for a simple Bayesian random
effects model and show that, under regularity conditions, the
geometric convergence rate converges to zero
as the data set grows in size.  This is, of course, even
better than dimension free.  In this paper, we show that similar
behavior is exhibited by Gibbs samplers for more general Bayesian
models that possess both random effects and traditional continuous
covariates, the so-called mixed models.  We now provide a detailed
description of the Gibbs samplers that are the focus of our study.

Consider the linear mixed effects model given by
\[
y_{ij} = x_{ij}^\intercal \beta + \eta_i + e_{ij} \,,
\]
where $i=1,\dots,q$, $j=1,\dots,r_i$, $x_{ij}$ is a $p \times 1$
vector of known covariates associated with $y_{ij}$, $\beta$ is an
unknown $p$-dimensional regression parameter, the $\eta_i$ are iid
$\mbox{N}(\mu,\lambda^{-1})$, and the $e_{ij}$ are iid
$\mbox{N}(0,\tau^{-1})$ and independent of the $\eta_i$.  A Bayesian
version of this model is formed by placing prior distributions on the
unknown parameters $\beta$, $\mu$, $\lambda$, and $\tau$.  In the
model we consider, the four parameters are assumed \textit{a priori}
independent, and we put flat priors on $\beta$ and $\mu$, and proper
gamma priors on $\lambda$ and $\tau$.  In particular, our prior
density is proportional to
\[
\lambda^{a_1-1} \exp(-b_1 \lambda) I_{\mathbb{R}_+}(\lambda)
\tau^{a_2-1} \exp(-b_2 \tau) I_{\mathbb{R}_+}(\tau) \;,
\]
where $a_1$, $a_2$, $b_1$, and $b_2$ are strictly positive
hyperparameters and $\mathbb{R}_+ := (0,\infty)$.  (Assume for the
time being that the resulting posterior distribution is proper.)
Denote the posterior density by
$\pi^*(\beta,\mu,\eta_1,\dots,\eta_q,\lambda,\tau \, | \, Y)$, where
$Y = (y_{11},\dots,y_{qr_q})^\intercal$ is the observed data, and, as
usual, the random effects are included because they are unobserved.
This posterior density, which has dimension $q + p + 3$, is highly
intractable.  However, because the prior density is conditionally
conjugate, there is a simple (two-block) Gibbs sampler that can be
used to explore $\pi^*$.  It turns out to be more convenient to work
with a very simple transformation of $\pi^*$.  Indeed, let $\eta_{00}
= \sqrt{q} \, \beta$ and $\eta_0 = \sqrt{q} \, \mu$, and define $\eta
= (\eta^\intercal_{00}, \eta_0, \eta_1, \dots, \eta_q)^\intercal$.
Denote the new posterior density by $\pi(\eta,\lambda,\tau \, | \,
Y)$.  In this paper, we analyze the Markov chain $\Gamma =
\{\eta^{(n)}\}_{n=0}^\infty$ whose Markov transition density $k:
\mathbb{R}^{q+p+1} \times \mathbb{R}^{q+p+1} \rightarrow (0,\infty)$
is given by
\begin{equation}
  \label{eq:mtd}
  k(\eta,\tilde{\eta}) = \int_0^\infty \int_0^\infty
  \pi_1(\tilde{\eta} \, | \, \lambda, \tau, Y) \pi_2(\lambda,\tau \, |
  \, \eta, Y) \, d\lambda \, d\tau \;,
\end{equation}
where $\pi_1(\eta \, | \, \lambda, \tau, Y)$ and $\pi_2(\lambda,\tau
\, | \, \eta, Y)$ denote conditional densities associated with $\pi$.
Of course, $\Gamma$ is the $\eta$-marginal chain of the two-block
Gibbs sampler that alternates between $(\lambda,\tau)$ and $\eta$.
It's easy to see that $\Gamma$ is irreducible, aperiodic, and positive
Harris recurrent with invariant density $\pi(\eta \, | \, Y) :=
\int_0^\infty \int_0^\infty \pi(\eta,\lambda,\tau \, | \, Y) \,
d\lambda \, d\tau$.  Furthermore, it is well known that, for two-block
Gibbs samplers, the marginal chains have the same convergence rate as
the underlying two-block Gibbs sampler (see,
  e.g., \citep{robe:rose:2001,diac:khar:salo:2008}).  Thus, in order to
learn about the convergence complexity of the two-block Gibbs sampler,
it suffices to study the marginal chain $\Gamma$.  It would clearly
also suffice to study the other marginal chain whose Markov transition
density $k': \mathbb{R}_+^2 \times \mathbb{R}_+^2 \rightarrow
(0,\infty)$ is given by
\[
k'\big((\lambda,\tau),(\tilde{\lambda},\tilde{\tau}) \big) =
\int_{\mathbb{R}^{q+p+1}} \pi_2(\tilde{\lambda},\tilde{\tau} \, | \,
\eta, Y) \pi_1(\eta \, | \, \lambda, \tau, Y) \, d\eta \;.
\]
However, despite the fact that this chain's dimension remains constant
(at 2) as $p$ and $q$ increase, $\Gamma$ is more amenable to analysis.

Drawing from $\lambda,\tau \,| \, \eta,Y$ is simple since, conditional
on $(\eta,Y)$, $\lambda$ and $\tau$ are independent, and each has a
gamma distribution.  Making draws from $\eta \,| \, \lambda,\tau,Y$ is
more complicated, but still straightforward.  Our strategy is to
sample sequentially.  We first draw from $\eta_{00} \,|\,
\lambda,\tau,Y$, and then from $\eta_0 \,|\, \eta_{00},
\lambda,\tau,Y$, and finally from $\eta_1,\dots,\eta_q \,|\, \eta_0,
\eta_{00}, \lambda,\tau,Y$.  All three of these conditional
distributions are normal distributions, and the last of the three
steps is easy since $\{\eta_i\}_{i=1}^q$ are independent given
$(\eta_0, \eta_{00}, \lambda,\tau,Y)$.  The distributions of
$\eta_{00} \,|\, \lambda,\tau,Y$ and $\eta_0 \,|\, \eta_{00},
\lambda,\tau,Y$ are derived in Appendix~\ref{app:A}.

In order to provide a precise statement of the algorithm, we must
introduce a bit of notation.  Let $N = \sum_{i=1}^q r_i$ (total sample
size), and let $X$ denote what is traditionally called the design
matrix, i.e., $X$ is the $N \times p$ matrix given by $X = (x_{11},
\dots, x_{q r_q})^\intercal$.  Let $\bar{X} = (\bar{x}_1
1_{r_1}^\intercal, \dots, \bar{x}_q 1_{r_q}^\intercal)^\intercal$,
where $\bar{x}_i = r_i^{-1} \sum_{j=1}^{r_i} x_{ij}$, and $1_{r_i}$ is
an $r_i \times 1$ column vector of 1s.  Note that $\bar{X}$ is also $N
\times p$.  Let $\bar{Y} = (\bar{y}_1 1_{r_1}^\intercal, \dots,
\bar{y}_q 1_{r_q}^\intercal)^\intercal$, where $\bar{y}_i = r_i^{-1}
\sum_{j=1}^{r_i} y_{ij}$.  Thus, we have
\[
X = \begin{bmatrix}
x_{11}^\intercal \\
\vdots \\
x_{1 r_1}^\intercal \\
\vdots \\
x_{q1}^\intercal \\
\vdots \\
x_{q r_q}^\intercal
\end{bmatrix},
\qquad
\bar{X} = \begin{bmatrix}
\bar{x}_1^\intercal \\
\vdots \\
\bar{x}_1^\intercal \\
\vdots \\
\bar{x}_q^\intercal \\
\vdots \\
\bar{x}_q^\intercal
\end{bmatrix},  \qquad \mbox{and}
\qquad
\bar{Y} = \begin{bmatrix}
\bar{y}_1 \\
\vdots \\
\bar{y}_1 \\
\vdots \\
\bar{y}_q \\
\vdots \\
\bar{y}_q
\end{bmatrix}.
\]
Let $c_i = r_i \tau/(\lambda + r_i \tau)$, $D_c = \oplus_{i=1}^q c_i
I_{r_i}$, and
\[
M = D_c + \frac{(I-D_c)1 1^\intercal (I-D_c)}{1^\intercal (I-D_c) 1} \;.
\]
(When we use $1$ with no subscript, it is understood to mean $1_N$.)
Finally, let $Q = (X^\intercal X - \bar{X}^\intercal M \bar{X})^{-1}$,
$v = \sqrt{q} \, Q (X^\intercal Y - \bar{X}^\intercal M \bar{Y})$,
$t_i = \lambda + r_i \tau$, and $z_i = t_i/(r_i \lambda \tau)$.  We
now state the algorithm for simulating $\Gamma$.  If the current state
of the chain is $\eta^{(n)} = \eta = (\eta^\intercal_{00}, \eta_0,
\eta_1, \dots, \eta_q)^\intercal$, then we simulate the next state,
$\eta^{(n+1)}$ using the following procedure.

\baro \vspace*{2mm}
\noindent {\rm Iteration $n+1$ of the $\eta$-marginal Gibbs algorithm:}
\begin{enumerate}
\item Draw $\lambda \sim \mbox{Gamma} \Big( q/2 + a_1, b_1 +
  \frac{1}{2} \sum_{i=1}^q (\eta_i - \eta_0/\sqrt{q})^2 \Big)$.
\item Draw $\tau \sim \mbox{Gamma} \Big( N/2 + a_2, b_2 +
  \frac{1}{2} \sum_{i=1}^q \sum_{j=1}^{r_i} (y_{ij} - x^\intercal_{ij}
  \eta_{00}/\sqrt{q} - \eta_i)^2 \Big)$.
\item Draw $\eta^{(n+1)}_{00} \sim \mbox{N}_p \Big( v, \frac{q}{\tau}
  Q \Big)$.
\item Draw $\eta^{(n+1)}_0 \sim \mbox{N} \Big( \frac{\sqrt{q}
  \sum_{i=1}^q (\bar{y}_i - \bar{x}^\intercal_i
  \eta^{(n+1)}_{00}/\sqrt{q})/z_i}{\sum_{i=1}^q 1/z_i},
  \frac{q}{\sum_{i=1}^q 1/z_i} \Big)$.
\item For $i=1,2,\dots,q$, draw $\eta^{(n+1)}_i \sim \mbox{N} \Big(
  \frac{\lambda \eta^{(n+1)}_0}{\sqrt{q} \, t_i} + \frac{r_i
    \tau}{t_i} \big(\bar{y}_i - \bar{x}^\intercal_i
  \eta^{(n+1)}_{00}/\sqrt{q} \big), 1/t_i \Big)$.
\end{enumerate}
\barba

\noindent Steps 1 and 2 simulate from $\pi_2(\lambda,\tau \, | \,
\eta, Y)$, while steps 3, 4, and 5 simulate from $\pi_1(\eta \, | \,
\lambda, \tau, Y)$.

Let $k^{(n)}(\eta,\cdot)$ denote the $n$-step Markov transition
density (Mtd); i.e., the density of $\eta^{(n)}$ given that
$\eta^{(0)} = \eta$.  The chain $\Gamma$ is said to be geometrically
ergodic if there exist a function $M: \mathbb{R}^{q+p+1} \rightarrow
[0,\infty)$ and a constant $\rho \in [0,1)$ such that for all $n$ and
    all $\eta$, we have
\begin{equation}
\label{eq:ge_den}
  \frac{1}{2} \int_{\mathbb{R}^{q+p+1}}
  \big|k^{(n)}(\eta,\tilde{\eta}) - \pi(\tilde{\eta} \, | \, Y) \big|
  \, d\tilde{\eta} \le M(\eta) \, \rho^n \;.
\end{equation}
The left-hand side of \eqref{eq:ge_den} is the total variation
distance between the distribution of $\eta^{(n)}$ (given that
$\eta^{(0)} = \eta$) and the invariant distribution.  (A more general
development of these concepts is provided in Section~\ref{sec:back}.)
The \textit{geometric rate of convergence}, $\rho_*$, is defined to be
the smallest $\rho \in [0,1]$ that satisfies \eqref{eq:ge_den} with
some $M(\cdot)$.  Results in \citet{roma:hobe:2012} imply that, under
mild regularity conditions, $\Gamma$ is indeed geometrically ergodic,
but their result is \textit{qualitative} in the sense that they do not
provide an explicit upper bound for $\rho_*$.  Our interest here
centers on the deeper issue of convergence complexity, i.e., on an
understanding of how $\rho_*$ scales as the underlying data set
increases in size.  More specifically, it is clear that $\rho_*$
depends on the underlying data, which is comprised of $Y$, $X$, $q$,
$p$, and the $r_i$s.  We consider a sequence of data sets of
increasing size, and show that, under regularity conditions concerning
the relative rates at which $q$, $p$, and the $r_i$ grow, not only is
$\rho_*$ bounded away from unity as the data set grows, but it
actually converges to 0.  Keep in mind that a convergence rate of 0
corresponds to immediate convergence, as if making iid draws from the
target.
Our main result (Theorem~\ref{thm:main} in
Section~\ref{sec:main}) is a generalization of Proposition 25 in
\citep{qin:hobe:2021a},
who consider the Gibbs sampler for a Bayesian
random effects model that is a simplified version of our mixed model
in which there are no covariates, and the data set is balanced, i.e.,
$r_i \equiv r$.  Our proofs are similar in structure to those of
\citet{qin:hobe:2021a}, but there are substantial differences due to
the fact that our mixed model is markedly more complex than the model
that they considered.

The remainder of the paper is organized as follows.
Section~\ref{sec:back} contains general background on Markov chain
convergence, as well as statements of results from
\citet{qin:hobe:2021a} and \citet{madr:seze:2010}, which we use to
analyze $\Gamma$.  Section~\ref{sec:main} contains a statement and
proof of the main result.  The proof relies on three preparatory
results, also stated in Section~\ref{sec:main}, that are proven in
subsequent sections.  In particular, Section~\ref{sec:conversion}
contains a proof of Proposition~\ref{prop:conversion}, which allows us
to convert Wasserstein bounds into total variation bounds.
Sections~\ref{sec:drift} and \ref{sec:contraction} contain proofs of
Proposition~\ref{prop:drift} (a drift condition) and
Proposition~\ref{prop:contraction} (a contraction condition),
respectively.  Proofs of some technical lemmas are relegated to the
Appendix.

\section{Markov Chain Background}
\label{sec:back}

Let $\X \subset \mathbb{R}^d$ and let $\mathcal{B}$ denote its Borel
$\sigma$-algebra.  Suppose that $K:\X \times \mathcal{B} \to [0,1]$ is
a Markov transition kernel (Mtk).  For any $n \in \mathbb{N} :=
\{1,2,3,\dots\}$, let $K^n$ be the $n$-step transition kernel, so that
$K^1 = K$.  For any probability measure $\nu: \mathcal{B} \to [0,1]$,
denote $\int_{\X} \nu(\df x) K^n(x, \cdot)$ by $\nu K^n (\cdot)$.  If
$\delta_x$ denotes a point mass at $x$, then $\delta_x K^n (\cdot) =
K^n(x,\cdot)$, and we will abbreviate this to $K_x^n(\cdot)$.  Assume
that the Markov chain corresponding to $K$ is irreducible, aperiodic,
and positive Harris recurrent (see, e.g., \citep{meyn:twee:2009}),
and let $\Pi$ denote its unique invariant probability distribution.
The goal of convergence analysis is to understand how quickly $\nu
K^n$ converges to~$\Pi$ as~$n \to \infty$ for a large class of $\nu$s.
The difference between $\nu K^n$ and $\Pi$ is usually measured using
the total variation (TV) distance, which is now defined.  For two
probability measures on $(\X, \mathcal{B})$, $\nu$ and $\phi$, the TV
distance between them is
\[
d_{\mbox{\scriptsize{TV}}} (\nu,\phi) = \sup_{A \in \mathcal{B}} \,
[\nu(A) - \phi(A)] \,.
\]
(If $\nu$ and $\phi$ each have a density with respect to a common
measure, then TV distance can be computed by integrating the absolute
difference between the densities - see, e.g., \eqref{eq:ge_den}.)  The
Markov chain defined by $K$ is \textit{geometrically ergodic} if there
exist $\rho < 1$ and $M: \X \to [0,\infty)$ such that, for each $x \in
  \X$ and $n \in \mathbb{N}$,
\begin{equation}
  \label{eq:ge}
  d_{\mbox{\scriptsize{TV}}}(K_x^n, \Pi) \leq M(x) \, \rho^n \,.
\end{equation}
Define the \textit{geometric convergence rate} of the chain as
\[
\rho_* = \inf \big\{ \rho \in [0,1] : \text{\eqref{eq:ge} holds for
  some } M: \X \to [0,\infty) \big\} \;.
\]
The chain is geometrically ergodic if and only if $\rho_* < 1$.

The standard technique for developing upper bounds on $\rho_*$
requires the construction of drift and minorization (d\&m) conditions
for the Markov chain under study
\citep{rose:1995,robe:rose:2004,baxe:2005}.  Unfortunately, the
d\&m-based methods are often overly conservative, especially in
high-dimensional situations (see, e.g., \citep{raja:spar:2015,qin:hobe:2021b}), and there is mounting
evidence suggesting that convergence complexity analysis becomes more
tractable when TV distance is replaced with Wasserstein distance
(see, e.g., \citep{hair:matt:sche:2011,durm:moul:2015,qin:hobe:2021a}).  In the
remainder of this section, we describe a method of bounding $\rho_*$
indirectly using Wasserstein distance.

Assume that $\X$ together with the usual Euclidean distance (based on the usual Euclidean norm $\|\cdot\|$) constitutes a Polish metric space.
For two probability measures on
$(\X, \mathcal{B})$, $\nu$ and $\phi$, their Wasserstein distance is
defined as
\[
d_{\mbox{\scriptsize{W}}}(\nu, \phi) = \underset{\xi \in \tau(\nu,
  \phi)}{\inf} \int_{\X \times \X} \norm{x-y} \, \xi(\df x, \df y) \;,
\]
where $\tau(\nu, \phi)$ is the set of all couplings of $\nu$ and
$\phi$; i.e., the set of all probability measures $\xi(\cdot, \cdot)$
on $(\X \times \X, \mathcal{B} \times \mathcal{B})$ having marginals
$\nu$ and $\phi$.  Here is a result that provides a connection between
Wasserstein distance and TV distance.

\begin{theorem}[\citet{madr:seze:2010}]
  \label{thm:ms}
  Assume that $K_x(\cdot)$ has a density $k(x,\cdot)$ with respect to
  some dominating measure $\mu$ for all $x \in \X$.  If there exists a
  constant $C < \infty$ such that, for all $x, y \in \X$,
\[
\int_\X \big \lvert k(x,z) - k(y,z) \big \rvert \, \mu(\df z) \le C \,
\norm{x-y} \;,
\]
then, for all $n \in \{2,3,4,\dots\}$, we have
\[
d_{\mbox{\scriptsize{TV}}}(K_x^n, \Pi) \le \frac{C}{2} \,
d_{\mbox{\scriptsize{W}}}(K_x^{n-1}, \Pi) \;.
\]
\end{theorem}

Suppose that the hypothesis of Theorem~\ref{thm:ms} holds, and that
the Markov chain driven by $K$ is geometrically ergodic with respect
to Wasserstein distance, i.e., we have $\gamma \in [0,1)$ and $M: \X
  \rightarrow [0,\infty)$ such that
    $d_{\mbox{\scriptsize{W}}}(K^m_x,\Pi) \le M(x) \, \gamma^m$ for
    all $x \in \X$ and all $m \in \mathbb{N}$.  Then it follows
    immediately from Theorem~\ref{thm:ms} that $\rho_* \le \gamma$.

One way to bound Wasserstein distance is through coupling, and
coupling is often achieved via random mappings, which we now describe.
On a probability space $(\Omega, \mathcal{F}, P)$, let $\theta: \Omega
\to \Theta$ be a random element, and let $\tilde{f}: \X \times \Theta
\to \X$ be a Borel measurable function.  Define $f(x) = \tilde{f}(x,
\theta)$ for all $x \in \X$.  Then $f$ is called a random mapping on
$\X$.  The evolution of a Markov chain can often be viewed as being
driven by a random mapping.  If $f(x) \sim K_x(\cdot)$ for all $x \in
\X$, then we say that $f$ \textit{induces} $K$.  For example, suppose
that $\X = \mathbb{R}$ and $K(x,dy) = (2 \pi)^{-1/2} \exp \{ - (y -
x/2)^2/2 \} \, dy$.  Let $Z$ be standard normal, and define
$\tilde{f}(x,Z) = x/2 + Z$.  Then the random mapping $f(x) = x/2 + Z$
induces $K$.  A random mapping $f$ is called \textit{differentiable}
if, with probability 1, for each $x,y \in \X$, $\frac{\df}{\df t} f(x
+ t(y-x))$, as a function of $t \in [0,1]$, exists and is integrable
\citep{qin:hobe:2021a}.  The following result provides a constructive
method of forming a bound on the Wasserstein distance to stationarity.

\begin{theorem}[Qin \& Hobert, 2021b]
  \label{thm:qh_2021a}

Suppose that $\X$ is a convex subset of $\mathbb{R}^d$, and assume
that the following three conditions hold.

\begin{enumerate}[label=(\subscript{A}{{\arabic*}})]
    \item
    There exist $c \in (0, \infty)$, $\zeta \in [0,1)$, $L \in
      [0,\infty)$, and a function $V: \X \to [0, \infty)$ such that
    \begin{equation}
    \label{eq:drift_drift}
      \int_{\mathsf{X}}V(x')K(x, dx') \leq \zeta V(x) + L
    \end{equation}
    for each $x \in \mathsf{X}$, and
    \begin{equation}
    \label{eq:norm_drift}
      c^{-1} \norm{x-y} \leq V(x) + V(y) + 1
    \end{equation}
    for each $(x,y) \in \mathsf{X} \times \mathsf{X}$.
    \item
    There exist a differentiable random mapping $f$ that induces $K$,
    some $\xi > 2L/(1-\zeta)$, $\gamma < 1$ and $\gamma_0 < \infty$
    such that
    \[
    \sup_{s \in [0,1]} \mathrm{E}
    \bigg \| \frac{\df}{\df s} f(x + s(y - x)) \bigg \| \leq
    \begin{cases}
          \gamma   \| x - y \| & \text{if $(x,y) \in C$} \\
          \gamma_0 \| x - y \| & \text{otherwise} \;,
    \end{cases}
    \]
    where ${\cal C} = \{ (x,y) \in \X \times \X : V(x) + V(y) \leq \xi
    \}$.
    \item
    Either $\gamma_0 \leq 1$ or
    \[
    \frac{\log(2L + 1)}{\log(2L + 1) - \log(\gamma)} < \frac{-\log
      \big[ (\zeta \xi + 2L + 1) / (\xi+1) \big]}{\log(\gamma_0) -
      \log \big[ (\zeta \xi + 2L + 1) / (\xi+1) \big]} \,.
    \]
\end{enumerate}

\noindent Then for each $x \in \mathsf{X}$, $n \in \mathbb{N}$, and
any real number $a$ such that
\[
 \frac{\log(2L + 1)}{\log(2L + 1) - \log(\gamma)} < a < \frac{-\log \big[
     (\zeta \xi + 2L + 1) / (\xi+1) \big]}{\log(\gamma_0 \vee 1) -
   \log \big[ (\zeta \xi + 2L + 1) / (\xi+1) \big]} \;,
\]
we have
\[
d_W(K^n_x, \Pi) \leq c \Big( \frac{(\zeta + 1)V(x) + L + 1}{1 -
  \rho_a} \Big) \rho_a^n \;,
\]
where
\[
\rho_a = \big[ \gamma^a (2L + 1)^{1-a} \big] \vee \Big[ \gamma_0^a
  \Big( \frac{\zeta \xi + 2L + 1}{\xi + 1} \Big)^{1-a} \Big] < 1 \;.
\]
\end{theorem}

We refer to conditions $(A_1)$ and $(A_2)$ as the \textit{drift} and
\textit{contraction} conditions, respectively.  In the next section,
we apply Theorems~\ref{thm:ms} and \ref{thm:qh_2021a} to $\Gamma$.

\section{Main Result}
\label{sec:main}

In order to state our main result, we must introduce a few definitions
and some notation.  Let $K(\eta,\cdot)$ denote the Mtk corresponding
to the Mtd $k(\eta,\cdot)$ defined at \eqref{eq:mtd}, and let
$K^n_\eta(\cdot)$ denote the probability measure associated with the
$n$-step Mtd $k^{(n)}(\eta,\cdot)$.  So, for a Borel set $B \in
\mathbb{R}^{p+q+1}$, $K(\eta,B) = \int_B k(\eta,\tilde{\eta}) \,
d\tilde{\eta}$ and $K^n_\eta(B) = \int_B k^{(n)}(\eta,\tilde{\eta}) \,
d\tilde{\eta}$.  Also let $\Pi(\cdot)$ denote the invariant
probability measure of $\Gamma$, i.e., $\Pi(\cdot)$ is the probability
measure corresponding to the marginal posterior density $\pi(\eta \, |
\, Y)$.  (The conditions that we will impose in our main result imply
posterior propriety.)  Define $V: \mathbb{R}^{p+q+1} \rightarrow [0,
  \infty)$ as follows:
\begin{equation}
  \label{eq:df}
  V(\eta) = \frac{1}{\rb q} \sum_{i=1}^{q} \sum_{j=1}^{r_i} \big(
  \bar{y} - x_{ij}^\intercal \eta_{00}/\sqrt{q} \big)^2 +
  \frac{\eta_0^2}{q} + \frac{1}{\rb q} \sum_{i=1}^{q} r_i (\eta_i +
  \bar{y} - \bar{y}_i)^2 \;,
\end{equation}
where $\bar{y} = q^{-1} \sum_{i=1}^q \bar{y}_i$ and
$\rb = q^{-1} \sum_{i=1}^q r_i$.
This will serve as
our drift function.

For two non-negative functions $g(z)$ and $h(z)$, we write $g(z) =
\mathcal{O}(h(z))$ if there exist positive numbers $c$ and $z_0$ such
that $g(z) \le c h(z)$ for all $z \ge z_0$.  We write $g(z) =
\Theta(h(z))$ if $g(z) = \mathcal{O}(h(z))$ and $h(z) =
\mathcal{O}(g(z))$.  Let $r_{\max} = \max\{r_1, \dots, r_q\}$ and
$r_{\min} = \min\{r_1, \dots, r_q\}$.
Let $\lambda_{\min}(\cdot)$ denote the
smallest eigenvalue of the (square symmetric matrix) argument, and
define $\lambda_{\max}(\cdot)$ analogously.

We envision a sequence of growing data sets with $q$, $p = p(q)$, and
$r_{\min} = r_{\min}(q)$ all diverging.  Our basic assumptions about
the manner in which the data set grows are as follows.
\begin{enumerate}[label=(\subscript{B}{{\arabic*}})]
  \item There exists a positive constant $m$, not depending on $q$,
    such that for all $q$,
\[
\frac{r_{\max}}{r_{\min}} \leq m \;.
\]
  \item  There exist positive constants $k_1$ and $k_2$, not depending
    on $q$, such that for all large $q$,
\[
  k_1 \leq \lambda_{\min} \Big[ \frac{1}{\bar{r}q} \big( X^\intercal X -
    \bar{X}^\intercal \bar{X} \big) \Big] \leq \lambda_{\max} \Big(
  \frac{1}{\bar{r}q} X^\intercal X \Big) \leq k_2 \,.
\]
  \item There exists a positive constant $\ell$, not depending on $q$,
    such that for all $q$,
\[
\frac{1}{\bar{r} q} \sum_{i=1}^q \sum_{j=1}^{r_i} y_{ij}^2 \leq \ell
\;.
\]
\item $p = \mathcal{O}(q)$.
\item There exists a positive constant $\delta$, not depending on $q$,
  such that
\[
  \frac{\bar{r}}{q^{2+\delta}} \rightarrow \infty
\]
as $q \rightarrow \infty$.
\end{enumerate}
Here is our main result:

\begin{theorem}
  \label{thm:main}
  Under conditions $(B_1)$-$(B_5)$, there exist positive constants $C$,
  $C'$ and $L_0$ (all independent of $q$), a rate $\rho = \rho(q)$
  satisfying $\rho(q) \rightarrow 0$ as $q \rightarrow \infty$, and a
  positive integer $q_0$ such that, for all $q \ge q_0$, all $\eta \in
  \mathbb{R}^{p+q+1}$, and all $n \in \{2,3,4,\dots\}$,
\[
d_{\mbox{\scriptsize{TV}}}(K^n_\eta,\Pi) \le C q \bar{r}^{3/2}
d_{\mbox{\scriptsize{W}}}(K^{n-1}_\eta, \Pi) \le C' q^2 \bar{r}^{3/2}
\big( V(\eta) + L_0 \big) \rho^{n-1} \;.
\]
\end{theorem}

\begin{remark}
  \label{rem:0}
  It follows from \eqref{eq:ge} that if $\rho = 0$, the Markov chain becomes stationary after one iteration (no matter where it is started).
  Given such a chain, one could produce an exact draw from $\Pi$ by running a single iteration (from any starting point).
  According to Theorem~\ref{thm:main}, the geometric convergence rates of our sequence of Gibbs samplers converge to 0 as $q \rightarrow \infty$.
  Unfortunately, this does not imply that a single iteration (or even several iterations) of the Gibbs sampler will provide an exact (or even a near exact) draw from the posterior distribution, even when the underlying data set is extremely large.
  Indeed, Theorem~\ref{thm:main} states that the total variation distance to stationarity after $n$ iterations of the Gibbs sampler is bounded above by $C' q^2 \bar{r}^{3/2} (V(\eta) + L_0) \rho^{n-1}$, where $\rho \in (0,1)$.
  Even if $\rho$ is very small, the term $C' q^2 \bar{r}^{3/2} (V(\eta) + L_0)$ may be very large - so large that $n$ needs to be large before the entire upper bound on the total variation distance to stationarity is small.
  The point is this: While we are able to say that $\rho_* \rightarrow 0$ as the sample size increases, we are not able to parlay this into a statement about fixed Gibbs samplers approaching stationarity in one (or even a few) iterations.
\end{remark}

We will prove Theorem~\ref{thm:main} using Theorems~\ref{thm:ms}
and \ref{thm:qh_2021a}.  We begin by stating three standalone results
(each proved in a subsequent section) that we will use in the proof.
The first allows for conversion (Theorem~\ref{thm:ms}), and the other
two concern drift and contraction (Theorem~\ref{thm:qh_2021a}).  Here
is the conversion condition, which is established using
Theorem~\ref{thm:ms} in Section~\ref{sec:conversion}.

\begin{proposition}
  \label{prop:conversion}
  Under conditions $(B_1)$ and $(B_2)$, there exist a constant $C > 0$
  and a positive integer $q_0$ such that, for all $q \ge q_0$,
\[
d_{\mbox{\scriptsize{TV}}}(K_\eta^n, \Pi) \leq C q \bar{r}^{3/2}
d_{\mbox{\scriptsize{W}}}(K_\eta^{n-1}, \Pi)
\]
for all $\eta \in \mathbb{R}^{p+q+1}$ and all $n \in \{2,3,4,\dots\}$.
\end{proposition}

Here is the drift condition, which is established in
Section~\ref{sec:drift}.

\begin{proposition}
  \label{prop:drift}
  Under conditions $(B_1)$-$(B_5)$, there exist $\zeta = \zeta(q)
  = \mathcal{O}(q^{-1})$, a constant $L$ independent of $q$, and a positive integer $q_0$ such that for all $q \ge q_0$ and all $\eta \in
  \mathbb{R}^{p+q+1}$, we have
\begin{equation*}
  \int_{\mathbb{R}^{p+q+1}} V(\tilde{\eta}) K(\eta, d\tilde{\eta})
  \leq \zeta V(\eta) + L \;.
\end{equation*}
\end{proposition}

The contraction condition involves a random mapping that induces
$\Gamma$, and we now describe this mapping.  (We note that this
mapping is also exploited in the proof of
Proposition~\ref{prop:drift}.)  Let $N_{00} \sim \mathrm{N}_p(0, I)$,
$N_i \sim \mathrm{N}(0,1)$, $i = 0,\dots, q$, $J_1 \sim
\mathrm{Gamma}(q/2+a_1, 1)$, and $J_2 \sim \mathrm{Gamma}(N/2+a_2,
1)$, and assume these are all pairwise independent.  Denote the
current state of $\Gamma$ as $\eta = (\eta_{00}^\intercal, \eta_0,
\eta_1, \dots, \eta_q)^{\intercal}$.  Then the next state
$\tilde{\eta} = (\tilde{\eta}_{00}^\intercal, \tilde{\eta}_0,
\tilde{\eta}_1,\dots, \tilde{\eta}_p)^{\intercal}$ can be expressed as
the following random mapping:
\[
f(\eta) = \tilde{\eta}(\eta) =
\begin{bmatrix}
\tilde{\eta}_{00}^{(\eta)} \\
\tilde{\eta}_0^{(\eta)} \\
\tilde{\eta}_1^{(\eta)} \\
\vdots \\
\tilde{\eta}_q^{(\eta)} \\
\end{bmatrix}
\]
where
\begin{align*}
\tilde{\eta}_{00}^{(\eta)} & = v^{(\eta)} +
\sqrt{\frac{q}{\tau^{(\eta)}}} \big(Q^{(\eta)}\big)^{\frac{1}{2}}
N_{00} \\ \tilde{\eta}_0^{(\eta)} & = \sqrt{q} \frac{\sum_{i=1}^{q}
  (\bar{y}_i - \bar{x}_i^\intercal \tilde{\eta}_{00}^{(\eta)} /
  \sqrt{q}) / z_i^{(\eta)}}{\sum_{i=1}^{q} 1/z_i^{(\eta)}} +
\sqrt{\frac{q}{\sum_{i=1}^{q} 1/z_i^{(\eta)}}} N_0
\\ \tilde{\eta}_i^{(\eta)} & = \frac{\lambda^{(\eta)}}{t_i^{(\eta)}}
\tilde{\eta}_0^{(\eta)} / \sqrt{q} + \frac{r_i
  \tau^{(\eta)}}{t_i^{(\eta)}} (\bar{y}_i - \bar{x}_i^\intercal
\tilde{\eta}_{00}^{(\eta)}/\sqrt{q}) + \sqrt{\frac{1}{t_i^{(\eta)}}}
N_i \;, \quad i = 1,2,\dots,q \;,
\end{align*}
and
\begin{align*}
\lambda^{(\eta)} & = \frac{J_1}{b_1+ \frac{1}{2}\sum_{i=1}^{q} (\eta_i
  - \eta_0/\sqrt{q})^2} \\ \tau^{(\eta)} & = \frac{J_2}{b_2 +
  \frac{1}{2}\sum_{i=1}^{q}\sum_{j=1}^{r_i} (y_{ij} - x_{ij}^\intercal
  \eta_{00} /\sqrt{q} - \eta_i)^2} \\ t_i^{(\eta)} & = r_i
\tau^{(\eta)} + \lambda^{(\eta)} \;, \quad i = 1,2,\dots,q
\\ z_i^{(\eta)} & = \frac{t_i^{(\eta)}}{r_i \tau^{(\eta)}
  \lambda^{(\eta)}} \;, \quad i = 1,2,\dots,q \\ D_c^{(\eta)} & =
\bigoplus_{i=1}^{q}\frac{r_i \tau^{(\eta)}}{t_i^{(\eta)}} I_{r_i}
\\ M^{(\eta)} & = D_c^{(\eta)} + \frac{(I - D_c^{(\eta)})1 1^\intercal
  (I - D_c^{(\eta)})}{1^\intercal (I - D_c^{(\eta)}) 1} \\ Q^{(\eta)}
& = (X^\intercal X - \bar{X}^\intercal M^{(\eta)} \bar{X})^{-1}
\\ v^{(\eta)} & = \sqrt{q} Q^{(\eta)} (X^\intercal Y -
\bar{X}^\intercal M^{(\eta)} \bar{Y}) \;.
\end{align*}

Here is the contraction condition, which is established in
Section~\ref{sec:contraction}.

\begin{proposition}
  \label{prop:contraction}
  Assume that $(B_1)$-$(B_5)$ hold, and define
\[
{\cal C} = \big \{ (\eta,\eta') \in \mathbb{R}^{p+q+1} \times
\mathbb{R}^{p+q+1} : V(\eta) + V(\eta') \le q^{\delta/3} \big \} \;,
\]
where $V(\cdot)$ is the drift function defined in \eqref{eq:df}, and
$\delta$ is given in $(B_5)$.  Let $f$ be the random mapping defined
above.  There exist
\[
\gamma = \gamma(q) = \mathcal{O} \bigg( \sqrt{\frac{q^{2 +
      \delta}}{\bar{r}}} \vee \frac{1}{\sqrt{q}} \bigg) \;\;\;\;
\mbox{and} \;\;\;\; \gamma_0 = \gamma_0(q) = \mathcal{O}(\sqrt{q})
\]
and a positive integer $q_0$ such that for all $q \ge q_0$, we have
\[
\sup_{s\in [0,1]} \mathrm{E} \, \bigg \| \frac{\df}{\df s} f(\eta +
s(\eta' - \eta)) \bigg\| \leq
      \begin{cases}
          \gamma \|\eta'-\eta\| & (\eta, \eta') \in
            C \\ \gamma_0 \|\eta'-\eta\| &\text{otherwise} \,.
      \end{cases}
\]
\end{proposition}

\begin{remark}
  \label{rem:2}
  Under $(B_5)$, $\lim_{q \to \infty} \gamma(q) = 0$.
\end{remark}

\begin{proof}[Proof of Theorem~\ref{thm:main}.]

We will show that there exist positive constants $C''$ and $L_0$ (both
independent of $q$), a rate $\rho = \rho(q)$ satisfying $\rho(q)
\rightarrow 0$ as $q \rightarrow \infty$, and a positive integer $q_0$
such that, for all $q \ge q_0$, all $\eta \in \mathbb{R}^{p+q+1}$, and
all $n \in \{2,3,4,\dots\}$,
\begin{equation}
  \label{eq:all_we_need}
  d_{\mbox{\scriptsize{W}}}(K^n_\eta, \Pi) \le C'' q \big( V(\eta) +
  L_0 \big) \rho^n \;.
\end{equation}
The result then follows immediately from
Proposition~\ref{prop:conversion}.  The Wasserstein bound in
\eqref{eq:all_we_need} will be established using
Theorem~\ref{thm:qh_2021a} in conjunction with
Propositions~\ref{prop:drift} and \ref{prop:contraction}.  We begin by
showing that assumption $(A_1)$ holds.  Proposition~\ref{prop:drift}
provides the inequality \eqref{eq:drift_drift}, and shows that we can
choose $\zeta = \zeta(q) = \Theta(q^{-1})$ and $L = L(q) =
\Theta(1)$.  However, we still need to establish the second
inequality, \eqref{eq:norm_drift}, which relates the drift function to
the Euclidean norm.  Fix $(\eta,\eta') \in \mathbb{R}^{q+p+1} \times
\mathbb{R}^{q+p+1}$.  We have
\begin{align*}
  \norm{\eta - \eta'}
  \le \norm{\eta - \eta'}^2 + 1
  =
  (\eta_{00} - \eta'_{00})^\intercal (\eta_{00} - \eta'_{00}) +
  (\eta_0 - \eta'_0)^2 + \sum_{i=1}^q (\eta_i - \eta'_i)^2 + 1
  \;.
\end{align*}
For two non-negative definite matrices, $M_1$ and $M_2$, we write $M_1 \preccurlyeq M_2$ if $M_2 - M_1$ is non-negative definite.
Now, $X^\intercal X - \bar{X}^\intercal \bar{X} \preccurlyeq
X^\intercal X$, so $\lambda_{\min}(X^\intercal X - \bar{X}^\intercal
\bar{X}) \le \lambda_{\min}(X^\intercal X)$.
Hence,
\begin{align*}
  & (\eta_{00} - \eta'_{00})^\intercal (\eta_{00} - \eta'_{00})
  \\ \le &
  \frac{1}{\lambda_{\min} \big( \frac{1}{\rb q} X^\intercal X \big)}
  (\eta_{00} - \eta'_{00})^\intercal \Big( \frac{1}{\rb q} X^\intercal
  X \Big) (\eta_{00} - \eta'_{00})
  \\ \le &  \frac{1}{\lambda_{\min}
    \big( \frac{1}{\rb q} (X^\intercal X - \bar{X}^\intercal \bar{X})
    \big)}  (\eta_{00}/\sqrt{q} - \eta'_{00}/\sqrt{q})^\intercal \Big(
  \frac{1}{\rb} \sum_{i=1}^q \sum_{j=1}^{r_i} x_{ij} x_{ij}^\intercal
  \Big) (\eta_{00}/\sqrt{q} - \eta'_{00}/\sqrt{q}) \\ \le &
  \frac{1}{\rb k_1} \sum_{i=1}^q \sum_{j=1}^{r_i} \big(
  x_{ij}^\intercal \eta_{00}/\sqrt{q} - x_{ij}^\intercal
  \eta'_{00}/\sqrt{q} \big)^2
  \\ \le &  \frac{2}{\rb k_1} \sum_{i=1}^q
  \sum_{j=1}^{r_i} \big(\bar{y} - x_{ij}^\intercal \eta_{00}/\sqrt{q}
  \big)^2 + \frac{2}{\rb k_1} \sum_{i=1}^q \sum_{j=1}^{r_i}
  \big(\bar{y} - x_{ij}^\intercal \eta'_{00}/\sqrt{q} \big)^2 \;,
\end{align*}
where the third inequality follows from $(B_2)$.  Continuing, we have
\begin{align*}
\sum_{i=1}^q (\eta_i - \eta'_i)^2
& \le 2 \sum_{i=1}^q (\eta_i +
\bar{y} - \bar{y}_i)^2 + 2 \sum_{i=1}^q (\eta'_i + \bar{y} -
\bar{y}_i)^2 \\
& \le \frac{2}{r_{\min}} \sum_{i=1}^q r_i (\eta_i +
\bar{y} - \bar{y}_i)^2 + \frac{2}{r_{\min}} \sum_{i=1}^q r_i (\eta'_i
+ \bar{y} - \bar{y}_i)^2 \\
& \le \frac{2m}{\rb} \sum_{i=1}^q r_i
(\eta_i + \bar{y} - \bar{y}_i)^2 + \frac{2m}{\rb } \sum_{i=1}^q r_i
(\eta'_i + \bar{y} - \bar{y}_i)^2 \;.
\end{align*}
So we conclude that
\begin{align*}
  & \norm{\eta - \eta'} \\
  \le &  (\eta_{00} - \eta'_{00})^\intercal
  (\eta_{00} - \eta'_{00}) + (\eta_0 - \eta'_0)^2 + \sum_{i=1}^q
  (\eta_i - \eta'_i)^2 + 1 \\ \le &  \frac{2}{\rb k_1}
  \sum_{i=1}^q \sum_{j=1}^{r_i} \big(\bar{y} - x_{ij}^\intercal
  \eta_{00}/\sqrt{q} \big)^2 + \frac{2}{\rb k_1} \sum_{i=1}^q
  \sum_{j=1}^{r_i} \big(\bar{y} - x_{ij}^\intercal \eta'_{00}/\sqrt{q}
  \big)^2 + 2\eta_0^2 + 2(\eta'_0)^2
  \\ &  + \frac{2m}{\rb}
  \sum_{i=1}^q r_i (\eta_i + \bar{y} - \bar{y}_i)^2 + \frac{2m}{\rb }
  \sum_{i=1}^q r_i (\eta'_i + \bar{y} - \bar{y}_i)^2 + 1 \\ = &
  \mathcal{O}(q) \big[ V(\eta) + V(\eta') + 1 \big] \;.
\end{align*}
Thus, \eqref{eq:norm_drift} is satisfied when $q$ is large, and we can
take $c$ in \eqref{eq:norm_drift} to be $c_1 q$, where $c_1$ is a
constant (not depending on $q$).

Now since $\zeta = \zeta(q) = \Theta(q^{-1})$ and $L = L(q) =
\Theta(1)$, it's clear that $q^{\delta/3} > 2L/(1-\zeta)$ for
all large $q$.  Thus, Proposition~\ref{prop:contraction} implies that
$(A_2)$ holds with $\xi = q^{\delta/3}$ when $q$ is large.  We now
show that $(A_3)$ holds when $q$ is large.  If $\gamma_0 \leq 1$, then
there's nothing to prove, so we assume that $\gamma_0 > 1$.  Define
\[
\mbox{LHS} = \frac{\log(2L + 1)}{\log(2L + 1) - \log(\gamma)} \;,
\]
and
\[\mbox{RHS} = \frac{-\log \big[ (\zeta \xi + 2L +
    1) / (\xi+1) \big]}{\log(\gamma_0) - \log \big[ (\zeta \xi + 2L
    + 1) / (\xi+1) \big]} \,.
\]
We must show that, for large $q$, $\mbox{LHS} < \mbox{RHS}$.  Since $L
\ge 0$ and $L = \Theta(1)$, it follows that $\log(2L+1) \ge 0$ and
$\log(2L+1) = \Theta(1)$.  Remark~\ref{rem:2} implies that
$\log(\gamma) \rightarrow - \infty$ as $q \rightarrow \infty$.  Hence,
$\mbox{LHS} \rightarrow 0$ as $q \rightarrow \infty$.  Next we show
that $\mbox{RHS}$ is bounded from below by a positive constant.  It's
clear from $(B_5)$ that we can assume (without loss of generality)
that $\delta<3$, and it follows that
\[
\frac{\zeta \xi + 2L + 1}{\xi + 1} = \frac{\zeta q^{\delta/3} + 2L
  + 1}{q^{\delta/3} + 1} = \Theta \Big( \frac{1}{q^{\delta/3}} \Big)
\;.
\]
Then since $\gamma_0 = \mathcal{O}(\sqrt{q})$, we have that, for large
$q$, $\mbox{RHS} \ge g(q)$, where $g(q) \rightarrow \frac{2 \delta}{3 + 2 \delta}$ as $q \rightarrow \infty$.  We conclude that, for large
$q$, $\mbox{LHS} < \mbox{RHS}$, and condition $(A_3)$ is satisfied.
According to Theorem~\ref{thm:qh_2021a}, we can take the rate to be
\[
\rho_a = \big[ \gamma^a (2L + 1)^{1-a} \big] \vee \Big[ \gamma_0^a
  \Big( \frac{\zeta \xi + 2L + 1}{\xi + 1} \Big)^{1-a} \Big]
\]
for any $a \in (\mbox{LHS},\mbox{RHS})$.  We now demonstrate that
there exists an $a \in (\mbox{LHS},\mbox{RHS})$ for which $\rho_a
\rightarrow 0$ as $q \rightarrow \infty$.  Clearly, for any fixed $a
\in (0,1)$, $\gamma^a (2L + 1)^{1-a} \rightarrow 0$ as $q \rightarrow
\infty$.  Moreover,
\[
  \gamma_0^a \Big( \frac{\zeta \xi + 2L + 1}{\xi + 1} \Big)^{1-a} =
  \Big[ \mathcal{O}(\sqrt{q}) \Big]^a \bigg[ \Theta \Big(
    \frac{1}{q^{\delta/3}} \Big) \bigg]^{1-a} = \mathcal{O} \big(
  q^{(\frac{\delta}{3} + \frac{1}{2})a - \frac{\delta}{3}} \big) \;.
\]
Thus, as long as $a < \frac{2 \delta}{3 + 2 \delta}$, we have
\[
  \gamma_0^a \Big( \frac{\zeta \xi + 2L + 1}{\xi + 1} \Big)^{1-a}
  \rightarrow 0
\]
as $q \rightarrow \infty$.  Thus, for such a value of $a$, $\rho_a
\rightarrow 0$ as $q \rightarrow \infty$.  Finally, for $q$ large
enough, there exists a constant $L_0$ such that, for all $\eta \in
\mathbb{R}^{p+q+1}$, we have
\[
\Big( \frac{(\zeta + 1)V(\eta) + L + 1}{1 - \rho_a} \Big) \le 2
\big( V(\eta) + L_0 \big) \;.
\]
This analysis, in conjunction with Theorem~\ref{thm:qh_2021a}, shows
that there exists a positive integer $q_0$ such that, for all $q \ge
q_0$, all $\eta \in \mathbb{R}^{p+q+1}$, and all $n \in
\{2,3,4,\dots\}$,
\[
d_W(K^n_\eta, \Pi) \leq c_1 q \Big( \frac{(\zeta + 1)V(\eta) + L + 1}{1
  - \rho} \Big) \rho^n \leq 2 c_1 q \big( V(\eta) + L_0 \big) \rho^n
\;,
\]
and $\rho = \rho(q) \rightarrow 0$ as $q \rightarrow \infty$.
\end{proof}

\section{Conversion Condition: Proof of Proposition~\ref{prop:conversion}}
\label{sec:conversion}

In this section, we use Theorem~\ref{thm:ms} to prove
Proposition~\ref{prop:conversion}, which is restated here for
convenience.

\begin{proposition}
  Under conditions $(B_1)$ and $(B_2)$, there exist a constant $C > 0$
  and a positive integer $q_0$ such that, for all $q \ge q_0$,
\[
d_{\mbox{\scriptsize{TV}}}(K_\eta^n, \Pi) \leq C \bar{r}^{3/2} q \,
d_{\mbox{\scriptsize{W}}}(K_\eta^{n-1}, \Pi)
\]
for all $\eta \in \mathbb{R}^{p+q+1}$ and all $n \in \{2,3,4,\dots\}$.
\end{proposition}
\begin{proof}
According to Theorem~\ref{thm:ms} it suffices to find a $C > 0$ and a
positive integer $q_0$ such that, for all $q \ge q_0$, we have
\[
\int_{\mathbb{R}^{p+q+1}} \big| k(\eta, \tilde{\eta}) -
k(\eta',\tilde{\eta}) \big| \, \df \tilde{\eta} \leq 2 C \bar{r}^{3/2}
q \|\eta - \eta' \|
\]
for all $\eta, \eta' \in \mathbb{R}^{p+q+1}$.  Fix $(\eta, \eta')$ and
assume that $\eta \ne \eta'$, otherwise the result is trivial.
Recall that in the posterior distribution $\lambda$ and $\tau$ are conditionally
independent given $(\eta,Y)$. Hence it follows that the density
$\pi_2(\lambda,\tau|\eta,Y)$ factors, and we will write
$\pi_2(\lambda,\tau|\eta,Y) = \pi_{21}(\lambda|\eta,Y) \pi_{22}(\tau|\eta,Y)$.
Now note that
\begin{align*}
  & \Big| \pi_1(\tilde{\eta} \,  | \, \lambda,\tau,Y) \pi_{21}(\lambda
  \, | \, \eta,Y) \pi_{22}(\tau \, | \, \eta,Y)
    - \pi_1(\tilde{\eta}
  \, | \, \lambda,\tau,Y) \pi_{21}(\lambda \, | \, \eta',Y)
  \pi_{22}(\tau \, | \, \eta',Y) \Big| \\ = & \pi_1(\tilde{\eta} \, |
  \, \lambda,\tau,Y) \Big| \pi_{21}(\lambda \, | \, \eta,Y)
  \pi_{22}(\tau \, | \, \eta,Y) - \pi_{21}(\lambda \, | \, \eta', Y)
  \pi_{22}(\tau \, | \, \eta',Y) \Big| \\ = & \pi_1(\tilde{\eta} \, |
  \, \lambda,\tau,Y) \Big| \pi_{21}(\lambda \, | \, \eta, Y)
  \pi_{22}(\tau \, | \, \eta,Y) - \pi_{21}(\lambda \, | \, \eta', Y)
  \pi_{22}(\tau \, | \, \eta,Y) \\
  & \hspace*{45mm}
  + \pi_{21}(\lambda \, | \, \eta', Y)\pi_{22}(\tau \, | \, \eta, Y)
  - \pi_{21}(\lambda \, | \, \eta', Y)
  \pi_{22}(\tau \, | \, \eta', Y) \Big| \;.
\end{align*}
It follows that
\[
\int_{\mathbb{R}^{p+q+1}} \big| k(\eta, \tilde{\eta}) -
k(\eta',\tilde{\eta}) \big| \, \df \tilde{\eta} \le A + B
\]
where
\[
A = \int_{\mathbb{R}_+} \big| \pi_{21}(\lambda \, | \, \eta, Y) -
\pi_{21}(\lambda \, | \, \eta', Y) \big| \, \df \lambda, \]
and
\[B = \int_{\mathbb{R}_+} \big| \pi_{22}(\tau \, |
\, \eta, Y) - \pi_{22}(\tau \, | \, \eta', Y) \big| \, \df \tau \;.
\]
We begin with the second term, $B$.  Recall that $\tau \, | \, \eta, Y
\sim \mbox{Gamma}(\alpha,\beta)$ where $\alpha = N/2 + a_2$
and
$\beta = b_2 + \frac{1}{2} \sum_{i=1}^q \sum_{j=1}^{r_i}
l^2_{ij}$,
where
\[
l_{ij} = y_{ij} - x_{ij}^\intercal \eta_{00}/\sqrt{q} - \eta_i \;.
\]
So, $\tau \, | \, \eta', Y \sim \mbox{Gamma}(\alpha,\beta')$ where
$\beta' = b_2 + \frac{1}{2}\sum_{i=1}^q
\sum_{j=1}^{r_i}
(l'_{ij})^2$,
where
\[
l'_{ij} = y_{ij} - x_{ij}^\intercal \eta'_{00}/\sqrt{q} - \eta'_i \;.
\]
Without loss of generality, assume that $\beta' > \beta$.  It's easy
to show that there is a (unique) number $u \in (0,\infty)$ such that
$\pi_{22}(u | \eta', Y) = \pi_{22}(u | \eta, Y)$, and that
$\pi_{22}(\tau | \eta', Y) \ge \pi_{22}(\tau | \eta, Y)$ for $\tau
\in (0,u]$.  It follows that
\begin{align*}
  B & = 2 \int_0^u \big[ \pi_{22}(\tau \, | \, \eta', Y) -
    \pi_{22}(\tau \, | \, \eta, Y) \big] \, \df \tau \\ & = 2 \int_0^u
  \frac{\beta^\alpha}{\Gamma(\alpha)} \tau^{\alpha-1} \exp(-\beta \tau)
  \bigg[ \bigg( \frac{\beta'}{\beta} \bigg)^\alpha \exp\big((\beta-\beta')
      \tau\big) - 1 \bigg] \, \df \tau \\ & \le 2 \int_0^u
  \frac{\beta^\alpha}{\Gamma(\alpha)} \tau^{\alpha-1} \exp(-\beta \tau)
  \bigg[ \bigg( \frac{\beta'}{\beta} \bigg)^\alpha - 1 \bigg] \, \df
  \tau \\ & \le 2 \bigg[ \bigg( \frac{\beta'}{\beta} \bigg)^\alpha - 1
    \bigg] \\ & = 2 \bigg[ \bigg( \frac{\beta + \Delta}{\beta}
    \bigg)^\alpha - 1 \bigg] \;,
\end{align*}
where $\Delta = \beta' - \beta$.  Now,
\begin{align*}
  2 \Delta & = 2(\beta' - \beta) \\ & = \sum_{i=1}^q \sum_{j=1}^{r_i}
  \big( (l'_{ij})^2 - l^2_{ij} \big) \\ & = \sum_{i=1}^q
  \sum_{j=1}^{r_i} \bigg[ \Big( y_{ij} - x_{ij}^\intercal
    \eta'_{00}/\sqrt{q} - \eta'_i \Big)^2 - \Big( y_{ij} -
    x_{ij}^\intercal \eta_{00}/\sqrt{q} - \eta_i \Big)^2 \bigg] \\ & =
  \sum_{i=1}^q \sum_{j=1}^{r_i} \bigg[ \Big( x_{ij}^\intercal
    \eta_{00}/\sqrt{q} + \eta_i - x_{ij}^\intercal \eta'_{00}/\sqrt{q}
    - \eta'_i \Big)^2 \\ & \hspace*{10mm} + 2 \Big( y_{ij} -
    x_{ij}^\intercal \eta_{00}/\sqrt{q} - \eta_i \Big)
    \Big(x_{ij}^\intercal \eta_{00}/\sqrt{q} + \eta_i -
    x_{ij}^\intercal \eta'_{00}/\sqrt{q} - \eta'_i \Big) \bigg] \\ &
  \le \sum_{i=1}^q \sum_{j=1}^{r_i} \Big( x_{ij}^\intercal
  \eta_{00}/\sqrt{q} + \eta_i - x_{ij}^\intercal \eta'_{00}/\sqrt{q} -
  \eta'_i \Big)^2 \\ & \hspace*{10mm} + 2 \sqrt{\sum_{i=1}^q
    \sum_{j=1}^{r_i} l^2_{ij}} \sqrt{\sum_{i=1}^q \sum_{j=1}^{r_i}
    \Big(x_{ij}^\intercal \eta_{00}/\sqrt{q} + \eta_i -
    x_{ij}^\intercal \eta'_{00}/\sqrt{q} - \eta'_i \Big)^2} \;,
\end{align*}
where the inequality is Cauchy-Schwarz.  Now, $\sum_{i=1}^q
\sum_{j=1}^{r_i} l^2_{ij} \le 2 \beta$ and
\begin{align*}
& \sum_{i=1}^q \sum_{j=1}^{r_i} \Big( x_{ij}^\intercal
\eta_{00}/\sqrt{q} + \eta_i - x_{ij}^\intercal \eta'_{00}/\sqrt{q} -
\eta'_i \Big)^2 \\
\le & \sum_{i=1}^{q} \sum_{j=1}^{r_i} \bigg[
  \frac{2}{q}(\eta_{00} - \eta'_{00})^\intercal x_{ij}
  x_{ij}^\intercal (\eta_{00} - \eta'_{00}) + 2 (\eta_i - \eta'_i)^2
  \bigg] \\
= &\frac{2}{q} (\eta_{00} - \eta'_{00})^\intercal
\sum_{i=1}^{q} \sum_{j=1}^{r_i} \big( x_{ij} x_{ij}^\intercal \big)
(\eta_{00} - \eta'_{00}) + 2 \sum_{i=1}^{q} \sum_{j=1}^{r_i} (\eta_i -
\eta'_i)^2 \\
= &\frac{2}{q} (\eta_{00} - \eta'_{00})^\intercal
X^\intercal X (\eta_{00} - \eta'_{00}) + 2 \sum_{i=1}^{q}
\sum_{j=1}^{r_i} (\eta_i - \eta'_i)^2 \\
\le & \frac{2 N k_2}{q}
(\eta_{00} - \eta'_{00})^\intercal (\eta_{00} - \eta'_{00}) + 2
r_{\max} \sum_{i=1}^{q} (\eta_i - \eta'_i)^2 \\
\le & 2 \rb k_2
(\eta_{00} - \eta'_{00})^\intercal (\eta_{00} - \eta'_{00}) + 2 m \rb
\sum_{i=1}^{q} (\eta_i - \eta'_i)^2 \\
\le & 2 c_2 \rb \|\eta -
\eta' \|^2 \;,
\end{align*}
where the second inequality follows from assumption $(B_2)$, and $c_2
:= k_2 \vee m$.  Thus,
\[
  \Delta \le c_2 \rb \|\eta - \eta' \|^2 + 2 \sqrt{c_2 \rb \beta}
  \|\eta - \eta' \| \;,
\]
and it follows that
\[
B \le 2 \bigg[ \Big(1 + 2 \sqrt{\frac{c_2 \rb}{b_2}} \|\eta - \eta' \|
  + \frac{c_2 \rb}{b_2} \|\eta - \eta' \|^2 \Big)^\alpha - 1 \bigg]
\;.
\]
So if we define $f: [0,\infty) \rightarrow [0,\infty)$ as
\[
f(x) = \Big(1 + 2 \sqrt{\frac{c_2 \rb}{b_2}} x + \frac{c_2 \rb}{b_2}
x^2 \Big)^\alpha - 1 \;,
\]
then $B \le 2 f(\|\eta - \eta' \|)$.  Note that $f(0)=0$, and
\[
f'(x) = \alpha \Big(1 + 2 \sqrt{\frac{c_2 \rb}{b_2}} x + \frac{c_2
  \rb}{b_2} x^2 \Big)^{\alpha-1} \bigg[ 2 \sqrt{\frac{c_2 \rb}{b_2}} +
  2 \frac{c_2 \rb}{b_2} x \bigg]
\]
is increasing on $(0,\infty)$.  Now assume that $\|\eta - \eta' \| <
\rb^{-3/2} q^{-1}$.  Then by the mean-value theorem, we have
\begin{equation}
  \label{eq:mvt}
  B \le 2 f'\big( \rb^{-3/2} q^{-1} \big) \|\eta - \eta' \| = 2 \Big
  \{ \rb^{-3/2} q^{-1} f'\big( \rb^{-3/2} q^{-1} \big) \Big \}
  \rb^{3/2} q \|\eta - \eta' \| \;.
\end{equation}
We now examine the behavior of the term in curly brackets as $q$ gets
large.
It can be shown using L'H\^opital's rule that,
for any $a \in \mathbb{R}$,
\[
\lim_{x \to 0} (1 + ax + o(x))^{1/x} = e^a  \;,
\]
where $o(x)$ represents any function
$h(x)$ such that $h(x) / x \to 0$ as $x \rightarrow 0$.
Using $x = 2 /(\rb q)$ and $a =  \sqrt{c_2 / b_2}$,
we have,
\[
\lim_{q \rightarrow \infty} \Big(1 + 2 \sqrt{\frac{c_2}{b_2}}
\frac{1}{\rb q} + \frac{c_2}{b_2 \rb^2 q^2} \Big)^{\rb q/2 + a_2 -1} =
\exp \bigg\{ \sqrt{\frac{c_2}{b_2}} \bigg\} \;.
\]
So we have
\begin{align*}
& \lim_{q \rightarrow \infty} \rb^{-3/2}  q^{-1} f'\big( \rb^{-3/2}
q^{-1} \big)
\\ = & \lim_{q \rightarrow \infty} \rb^{-3/2} q^{-1} (\rb
q/2 + a_2) \Big(1 + 2 \sqrt{\frac{c_2}{b_2}} \frac{1}{\rb q} +
\frac{c_2}{b_2 \rb^2 q^2} \Big)^{\rb q/2 + a_2 -1}
 \bigg[ 2
  \sqrt{\frac{c_2 \rb}{b_2}} + \frac{2 c_2}{b_2 \sqrt{\rb} q} \bigg]
\\ = &
\lim_{q \rightarrow \infty} \bigg( \frac{1}{2} + \frac{a_2}{\rb q}
\bigg) \Big(1 + 2 \sqrt{\frac{c_2}{b_2}} \frac{1}{\rb q} +
\frac{c_2}{b_2 \rb^2 q^2} \Big)^{\rb q/2 + a_2 -1} \bigg[ 2
  \sqrt{\frac{c_2}{b_2}} + \frac{2 c_2}{b_2 \rb q} \bigg]
\\ = &
\sqrt{\frac{c_2}{b_2}} \exp \bigg\{ \sqrt{\frac{c_2}{b_2}} \bigg\} \;.
\end{align*}
Thus, there exists $c_3 < \infty$ such that
\[
\max_{q > 2} 2 \Big \{ \rb^{-3/2} q^{-1} f'\big( \rb^{-3/2} q^{-1}
\big) \Big \} = c_3 \;.
\]
It now follows from \eqref{eq:mvt} that
\begin{equation*}
  B \le c_3 \rb^{3/2} q \|\eta - \eta' \|
\end{equation*}
whenever $\|\eta - \eta' \| < \rb^{-3/2} q^{-1}$.

A slightly simpler version of this same argument shows that there
exists a $c_1 \in (0,\infty)$ such that $A \le c_1 q \|\eta - \eta'
\|$ whenever $\|\eta - \eta' \| \le 1/q$.  Putting all of this
together shows that, for all $q > 2$ and all pairs $(\eta,\eta')$ such
that $\|\eta - \eta' \| < \rb^{-3/2} q^{-1}$, we have
\begin{align*}
\int_{\mathbb{R}^{p+q+1}} \big| k(\eta, \tilde{\eta}) -
k(\eta',\tilde{\eta}) \big| \, \df \tilde{\eta} \le A + B \le c_1 q
\|\eta - \eta' \| + c_3 \rb^{3/2} q \|\eta - \eta' \|
\le (c_1 + c_3)
\rb^{3/2} q \|\eta - \eta' \| \;.
\end{align*}
Now all that remains is to extend this so it holds for all pairs
$(\eta,\eta')$.  For an arbitrary pair $(\eta,\eta')$, let $v =
(\eta' - \eta)/m^*$ where the integer $m^*$ is chosen
such that
\[
\| v \| = \bigg\| \frac{1}{m^*}(\eta' - \eta) \bigg\| = \frac{1}{m^*}
\| \eta' - \eta \| < \rb^{-3/2} q^{-1} \;.
\]
For $j=0,1,\dots,m^*$, let $\xi_j = \eta + j v$.  Then $\xi_0 = \eta$,
$\xi_{m^*} = \eta'$, and $\|\xi_{j+1} - \xi_j \| = \|v\| < \rb^{-3/2}
q^{-1}$.  Hence,
\begin{align*}
\int_{\mathbb{R}^{p+q+1}} \big| k(\eta, \tilde{\eta}) -
k(\eta',\tilde{\eta}) \big| \, \df \tilde{\eta} & \le
\sum_{j=0}^{m^*-1} \int_{\mathbb{R}^{p+q+1}} \big| k(\xi_{j+1},
\tilde{\eta}) - k(\xi_j,\tilde{\eta}) \big| \, \df \tilde{\eta} \\ &
\le \sum_{j=0}^{m^*-1} (c_1 + c_3) \rb^{3/2} q \| \xi_{j+1} - \xi_j \|
\\ & = \sum_{j=0}^{m^*-1} (c_1 + c_3) \rb^{3/2} q \| v \| \\ & = (c_1
+ c_3) \rb^{3/2} q \|\eta - \eta' \| \;.
\end{align*}
\end{proof}

\section{Drift Condition: Proof of Proposition~\ref{prop:drift}}
\label{sec:drift}

We begin by stating three results, all proven in Appendix~\ref{app:B},
that will be used to establish the drift and contraction conditions.
The first lists a few simple facts about $X$ \& $Y$, the second is a
simple matrix result that we could not find stated elsewhere, and the
third provides an upper bound on the largest eigenvalue of the square
of the derivative of the square root of a matrix.

\begin{lemma}
  \label{lem:conseq}
  The data, $X$ and $Y$, satisfy the following:
  \begin{enumerate}[label=(\roman*)]
    \item $\bar{X}^\intercal \bar{X} \preccurlyeq X^\intercal X$
    \item $\sum_{i=1}^{q} \bar{x}_i \bar{x}^\intercal_i \preccurlyeq
      \frac{1}{r_{\min}} \bar{X}^\intercal \bar{X} \preccurlyeq
      \frac{1}{r_{\min}} X^\intercal X$
    \item $\sum_{i=1}^{q} \bar{y}_i^2 \le \frac{1}{r_{\min}} Y^\intercal Y$
    \item $\bar{y}^2 \le \frac{1}{q r_{\min}} Y^\intercal Y$
\end{enumerate}
\end{lemma}

\begin{remark}
  \label{rem:1}
  Under $(B_1)$-$(B_3)$, the results in Lemma~\ref{lem:conseq} imply
  that $\sum_{i=1}^{q} \bar{x}_i \bar{x}^\intercal_i \preccurlyeq m q
  k_2 I$, $\sum_{i=1}^{q} \bar{y}_i^2 = \mathcal{O}(q)$, and
  $\bar{y}^2 = \mathcal{O}(1)$.
\end{remark}

\begin{lemma}
  \label{lem:simple_matrix}
  If $C, D$ are conformable matrices, then
\[
(C+D)^\intercal (C+D) \preccurlyeq 2(C^\intercal C + D^\intercal D)
\;.
\]
\end{lemma}

\begin{lemma}
  \label{lem:sr}
  Let $A = A(x)$ be a positive definite matrix that depends on a
  scalar, $x$.  If $A^{1/2} = A^{1/2}(x)$ is the
  unique positive definite square root of $A$, then
\[
\lambda_{\max} \bigg\{ \Big( \frac{\df A^{1/2}}{\df x} \Big)^2
\bigg\} \le \frac{\lambda_{\max} \Big\{ \Big( \frac{\df A}{\df x}
  \Big)^2 \Big\} }{4 \lambda_{\min}(A)} \;.
\]
\end{lemma}

We now restate Proposition~\ref{prop:drift} for convenience.

\begin{proposition}
  Under assumptions $(B_1)$-$(B_5)$, there exist $\zeta = \zeta(q) =
  \mathcal{O}(q^{-1})$, $L = L(q) = \mathcal{O}(1)$, and a positive
  integer $q_0$ such that for all $q \ge q_0$ and all $\eta \in
  \mathbb{R}^{p+q+1}$, we have
\begin{equation*}
  \int_{\mathbb{R}^{p+q+1}} V(\tilde{\eta}) K(\eta, d\tilde{\eta})
  \leq \zeta V(\eta) + L \;.
\end{equation*}
\end{proposition}

\begin{proof}
We have
\begin{align}
  \label{eq:drift_1}
  & V(\tilde{\eta}) \notag \\
  = &\frac{1}{N} \sum_{i=1}^{q} \sum_{j=1}^{r_i}
  \big( \bar{y} - x_{ij}^\intercal \tilde{\eta}_{00}/\sqrt{q} \big)^2
  + \frac{\tilde{\eta}_0^2}{q} + \frac{1}{N} \sum_{i=1}^{q} r_i
  (\tilde{\eta}_i + \bar{y} - \bar{y}_i)^2 \notag \\ \le & \frac{2}{N}
  \sum_{i=1}^{q} \sum_{j=1}^{r_i} \bar{y}^2 + \frac{2}{qN}
  \tilde{\eta}^\intercal_{00} \Big( \sum_{i=1}^{q} \sum_{j=1}^{r_i}
  x_{ij} x_{ij}^\intercal \Big) \tilde{\eta}_{00} +
  \frac{\tilde{\eta}_0^2}{q} + \frac{r_{\max}}{N} \sum_{i=1}^{q}
  (\tilde{\eta}_i + \bar{y} - \bar{y}_i)^2 \notag \\ \le & 2 \bar{y}^2
  + \frac{2}{qN} \tilde{\eta}^\intercal_{00} (X^\intercal X)
  \tilde{\eta}_{00} + \frac{\tilde{\eta}_0^2}{q} + \frac{m}{q}
  \sum_{i=1}^{q} (\tilde{\eta}_i + \bar{y} - \bar{y}_i)^2 \;,
\end{align}
where, in the last line, we used the fact that $r_{\max}/\rb \le m$.
We now examine the last term in \eqref{eq:drift_1}.  Using the
definition of $\tilde{\eta}_i$ from the random mapping, we have
\begin{align*}
  \sum_{i=1}^{q} (\tilde{\eta}_i + \bar{y} - \bar{y}_i)^2 & =
  \sum_{i=1}^{q} \Big( \frac{\lambda}{t_i} \tilde{\eta}_0 / \sqrt{q} +
  \frac{r_i \tau}{t_i} (\bar{y}_i - \bar{x}_i^\intercal
  \tilde{\eta}_{00}/\sqrt{q}) + \sqrt{\frac{1}{t_i}} N_i + \bar{y} -
  \bar{y}_i \Big)^2 \\ & = \sum_{i=1}^{q} \Big( \frac{\lambda}{t_i}
  \tilde{\eta}_0 / \sqrt{q} - \frac{r_i \tau}{t_i} \bar{x}_i^\intercal
  \tilde{\eta}_{00}/\sqrt{q} + \sqrt{\frac{1}{t_i}} N_i + \bar{y} -
  \frac{\lambda}{t_i} \bar{y}_i \Big)^2 \\ & \le 5 \tilde{\eta}_0^2 +
  \frac{5}{q} \tilde{\eta}_{00}^\intercal \Big( \sum_{i=1}^q \bar{x}_i
  \bar{x}_i^\intercal \Big) \tilde{\eta}_{00} + 5 \sum_{i=1}^q
  \frac{1}{t_i} N_i^2 + 5 q \bar{y}^2 + 5 \sum_{i=1}^q \bar{y}_i^2
  \\ & \le 5 \tilde{\eta}_0^2 + 5 m k_2 \tilde{\eta}_{00}^\intercal
  \tilde{\eta}_{00} + 5 \sum_{i=1}^q \frac{1}{r_i \tau} N_i^2 +
  \mathcal{O}(q) + \mathcal{O}(q) \\ & = \mathcal{O}(1)
  \tilde{\eta}_0^2 + \mathcal{O}(1) \tilde{\eta}_{00}^\intercal
  \tilde{\eta}_{00} + \mathcal{O} \Big( \frac{1}{\rb} \Big) \frac{1}{\tau}
  \sum_{i=1}^q N_i^2 + \mathcal{O}(q) \;,
\end{align*}
where the second inequality follows from Remark~\ref{rem:1}.  Hence,
by $(B_2)$ and Remark~\ref{rem:1}, we have
\begin{align*}
  V(\tilde{\eta}) & \le 2 \bar{y}^2 + \frac{2}{qN}
  \tilde{\eta}^\intercal_{00} (X^\intercal X) \tilde{\eta}_{00} +
  \frac{\tilde{\eta}_0^2}{q} + \frac{m}{q} \sum_{i=1}^{q}
  (\tilde{\eta}_i + \bar{y} - \bar{y}_i)^2 \\
& = \mathcal{O}(1) +
  \mathcal{O} \Big( \frac{1}{q} \Big) \tilde{\eta}^\intercal_{00} \tilde{\eta}_{00} +
  \mathcal{O} \Big( \frac{1}{q} \Big) \tilde{\eta}_0^2
    + \frac{m}{q} \Big(
  \mathcal{O}(1) \tilde{\eta}_0^2 + \mathcal{O}(1)
  \tilde{\eta}_{00}^\intercal \tilde{\eta}_{00} +
  \mathcal{O} \Big( \frac{1}{\rb} \Big) \frac{1}{\tau} \sum_{i=1}^q N_i^2 +
  \mathcal{O}(q) \Big) \\
& = \mathcal{O}(1) + \mathcal{O} \Big( \frac{1}{q} \Big)
  \tilde{\eta}_{00}^\intercal \tilde{\eta}_{00} + \mathcal{O} \Big( \frac{1}{q} \Big)
  \tilde{\eta}_0^2 + \mathcal{O}((q \rb)^{-1}) \frac{1}{\tau}
  \sum_{i=1}^q N_i^2 \;.
\end{align*}
Now let $A = \mbox{E} (\tilde{\eta}_{00}^\intercal \tilde{\eta}_{00} |\eta )$
and $B = \mbox{E}(\tilde{\eta}_0^2 | \eta)$, where the expectation is, of
course, with respect to the Mtk $K(\eta,\cdot)$.
It follows that
\begin{equation}
  \label{eq:drift_2}
  \int_{\mathbb{R}^{p+q+1}} V(\tilde{\eta}) K(\eta, d\tilde{\eta}) =
  \mathcal{O}(1) + \mathcal{O} \Big( \frac{1}{q} \Big) A + \mathcal{O}
  \Big( \frac{1}{q} \Big) B + \mathcal{O} \Big( \frac{1}{\rb} \Big)
  \mbox{E}\Big( \frac{1}{\tau} \Big) \;.
\end{equation}
To bound $A$, recall that $ \eta_{00} \,|\, \lambda, \tau, Y \sim
\mathrm{N}_p (v, \frac{q}{\tau}Q)$, where $Q = (X^\intercal X -
\bar{X}^\intercal M \bar{X})^{-1}$ and $v = \sqrt{q} Q (X^\intercal Y
- \bar{X}^\intercal M \bar{Y})$.  Let $\mathrm{tr}(\cdot)$ denote the
trace of the (square matrix) argument.  We have
\begin{equation}
  \label{eq:drift_3}
  A = \mathrm{E} ( \tilde{\eta}_{00}^\intercal \tilde{\eta}_{00} |\eta) =
  \mathrm{E} \Big[ v^\intercal v + \mathrm{tr} \Big( \frac{q}{\tau}Q
    \Big) \Big] = \mathrm{E} (v^\intercal v) + \mathrm{tr} \Big[
    \mathrm{E} \Big( \frac{q}{\tau}Q \Big) \Big] \;.
\end{equation}
In order to handle the trace term in \eqref{eq:drift_3}, we need to
look carefully at $Q$.  First,
\begin{align}
M & = D_c + \frac{(I - D_c)1 1^\intercal (I - D_c)}{1^\intercal (I -
  D_c) 1} \notag \\
  & = D_c + (I - D_c)^{\frac{1}{2}} \Big(\frac{(I -
  D_c)^{\frac{1}{2}} 1 1^\intercal (I -
  D_c)^{\frac{1}{2}}}{1^\intercal (I - D_c) 1} \Big) (I -
D_c)^{\frac{1}{2}}  \notag \\
& \preccurlyeq D_c + (I - D_c)^{\frac{1}{2}}
\Big(\frac{1^\intercal (I - D_c)1 }{1^\intercal (I - D_c) 1} I \Big)
(I - D_c)^{\frac{1}{2}} \notag \\
& = D_c + (I - D_c) \notag \\
& = I \,,   \label{eq:Mbound}
\end{align}
where the ``inequality'' is due to the fact that, for any vector $u$,
$u u^\intercal \preccurlyeq u^\intercal u I$.  It now follows from
$(B_2)$ that
\begin{equation}
  \label{eq:Qbnd}
  Q = (X^\intercal X - \bar{X}^\intercal M \bar{X})^{-1} \preccurlyeq
  (X^\intercal X - \bar{X}^\intercal \bar{X})^{-1} \preccurlyeq \big[
    (q \rb k_1) I \big]^{-1} = \frac{1}{q \rb k_1} I \;.
\end{equation}
Therefore,
\begin{equation}
  \label{eq:drift_4}
  \mathrm{tr} \Big[ \mathrm{E} \Big( \frac{q}{\tau}Q \Big) \Big] \le
  \mathrm{tr} \Big[ \mathrm{E} \Big( \frac{q}{\tau} \frac{1}{q \rb
      k_1} I \Big) \Big] = \mathcal{O}\Big( \frac{p}{\rb} \Big)
  \mathrm{E}\Big( \frac{1}{\tau} \Big) \;.
\end{equation}
We now go to work on the penultimate term in \eqref{eq:drift_3}.  As
in Appendix~\ref{app:B}, let $R = \bigoplus_{i=1}^{q}
r_i^{-1}
J_{r_i}$.  It's clear that $RX = \bar{X}$ and $RY = \bar{Y}$, and that
$R$ is idempotent.  Hence, $R\bar{X} = R^2 X = R X = \bar{X}$.
Moreover, $R$ and $D_c$ commute, which implies that $R$ and $M$ also
commute.  It follows that
\begin{align*}
v = \sqrt{q} Q (X^\intercal Y - \bar{X}^\intercal M \bar{Y}) =
\sqrt{q} Q (X^\intercal Y - \bar{X}^\intercal M R Y) & =
\sqrt{q} Q (X^\intercal  - \bar{X}^\intercal M R) Y \\
& =
\sqrt{q} Q (X  - M \bar{X})^\intercal Y \,.
\end{align*}
Since $Q^2 \preccurlyeq \lambda^2_{\max}(Q) I \preccurlyeq (q \rb
k_1)^{-2} I$, we have
\begin{align}
  \label{eq:drift_5}
  v^\intercal v & = q Y^\intercal (X - M \bar{X}) Q^2 (X - M
  \bar{X})^\intercal Y \notag \\ & \le \frac{q}{(q \rb k_1)^2}
  Y^\intercal (X - M \bar{X}) (X - M \bar{X})^\intercal Y \notag \\ &
  \le \frac{2q}{(q \rb k_1)^2} Y^\intercal (X X^\intercal + M \bar{X}
  \bar{X}^\intercal M) Y \notag \\ & \le \frac{2 q^2 \rb k_2}{(q \rb
    k_1)^2} Y^\intercal (I + M^2) Y \notag \\ & \le \frac{4 k_2}{\rb
    k_1^2} Y^\intercal Y \notag \\ & \le \frac{4 k_2}{\rb k_1^2} (q
  \rb \ell) \notag \\ & = \mathcal{O}(q) \;,
\end{align}
where the second inequality follows from
Lemma~\ref{lem:simple_matrix}, the third inequality follows from
$(B_2)$ and Lemma~\ref{lem:conseq}, the fourth inequality holds because $M \preccurlyeq I$, and
the fifth inequality follows from $(B_3)$.  Combining
\eqref{eq:drift_4} and \eqref{eq:drift_5}, we have
\begin{equation}
  \label{eq:Abnd}
  A = \mathcal{O} \Big( \frac{p}{\rb} \Big) \mathrm{E}\Big( \frac{1}{\tau} \Big) +
  \mathcal{O}(q) \;.
\end{equation}
We now move onto $B$.  Recall that
\[
\tilde{\eta}_0 = \sqrt{q} \frac{\sum_{i=1}^{q} (\bar{y}_i -
  \bar{x}_i^\intercal \tilde{\eta}_{00} / \sqrt{q}) /
  z_i}{\sum_{i=1}^{q} 1/z_i} + \sqrt{\frac{q}{\sum_{i=1}^{q} 1/z_i}}
N_0 \;.
\]
Define $t_{\min} = \lambda + r_{\min} \tau$ and $z_{\min} =
t_{\min}/(r_{\min} \lambda \tau)$, and define $t_{\max}$ and
$z_{\max}$ analogously.  It's important to note that while $t_{\min}$
is, in fact, the smallest $t_i$, $z_{\min}$ is actually the largest
$z_i$, but we find this notation convenient.  We have $1 \le
z_{\min}/z_{\max} \le m$.  Now,
\begin{align}
  \label{eq:drift_6}
  B & = \mbox{E}(\tilde{\eta}_0^2 | \eta) \notag \\ & = \mbox{E} \bigg(
  \sqrt{q} \frac{\sum_{i=1}^{q} (\bar{y}_i - \bar{x}_i^\intercal
    \tilde{\eta}_{00} / \sqrt{q}) / z_i}{\sum_{i=1}^{q} 1/z_i} +
  \sqrt{\frac{q}{\sum_{i=1}^{q} 1/z_i}} N_0 \bigg)^2 \notag \\ & = q
  \mbox{E} \bigg( \frac{\sum_{i=1}^{q} (\bar{y}_i -
    \bar{x}_i^\intercal \tilde{\eta}_{00} / \sqrt{q}) /
    z_i}{\sum_{i=1}^{q} 1/z_i} \bigg)^2 + \mbox{E} \Big(
  \frac{q}{\sum_{i=1}^{q} 1/z_i} \Big) \notag \\ & \le q \mbox{E}
  \bigg( \frac{\sum_{i=1}^{q} (\bar{y}_i - \bar{x}_i^\intercal
    \tilde{\eta}_{00} / \sqrt{q})^2 / z_i}{\sum_{i=1}^{q} 1/z_i}
  \bigg) + \mbox{E} \Big( \frac{q}{\sum_{i=1}^{q} 1/z_i} \Big) \notag
  \\ & \le q \mbox{E} \bigg( \frac{\sum_{i=1}^{q} (\bar{y}_i -
    \bar{x}_i^\intercal \tilde{\eta}_{00} / \sqrt{q})^2 /
    z_{\max}}{\sum_{i=1}^{q} 1/z_{\min}} \bigg) + \mbox{E} (z_{\min})
  \notag \\ & \le m \mbox{E} \Big( \sum_{i=1}^{q} (\bar{y}_i -
  \bar{x}_i^\intercal \tilde{\eta}_{00} / \sqrt{q})^2 \Big) + \mbox{E}
  (z_{\min}) \notag \\ & \le 2 m \sum_{i=1}^{q} \bar{y}_i^2 +
  \frac{2m}{q} \mbox{E} \Big[ \tilde{\eta}_{00}^\intercal \Big(
    \sum_{i=1}^{q} \bar{x}_i \bar{x}_i^\intercal \Big)
    \tilde{\eta}_{00} \Big] + \mbox{E} \Big( \frac{t_{\min}}{r_{\min}
    \lambda \tau} \Big) \notag \\ & = \mathcal{O}(q) +
  \mathcal{O}(1) A + \mbox{E} \Big( \frac{1}{\lambda} \Big) +
  \mathcal{O} \Big( \frac{1}{\rb} \Big) \mbox{E} \Big( \frac{1}{\tau}
  \Big) \;,
\end{align}
where the third equality follows from the fact that $N_0$ has mean 0
and is independent of all of the other random vectors in the random
mapping, the first inequality is Jensen's, and the last line follows
from Remark~\ref{rem:1}.  Combining \eqref{eq:drift_2},
\eqref{eq:Abnd}, and \eqref{eq:drift_6} yields
\begin{align}
  & \label{eq:drift_7}
  \int_{\mathbb{R}^{p+q+1}} V(\tilde{\eta})  K(\eta, d\tilde{\eta})
  \notag
\\ = & \mathcal{O}(1) + \mathcal{O} \Big( \frac{1}{q} \Big) A
  + \mathcal{O} \Big( \frac{1}{q} \Big) B + \mathcal{O} \Big(
  \frac{1}{\rb} \Big) \mbox{E}\Big( \frac{1}{\tau} \Big) \notag
\\ = &
  \mathcal{O}(1) + \mathcal{O} \Big( \frac{1}{q} \Big) A + \mathcal{O}
  \Big( \frac{1}{q} \Big) \bigg[ \mathcal{O}(q) + \mathcal{O}(1) A +
    \mbox{E} \Big( \frac{1}{\lambda} \Big) + \mathcal{O} \Big(
    \frac{1}{\rb} \Big) \mbox{E} \Big( \frac{1}{\tau} \Big) \bigg]
   + \mathcal{O} \Big( \frac{1}{\rb} \Big) \mbox{E} \Big( \frac{1}{\tau}
  \Big) \notag
\\ = & \mathcal{O}(1) + \mathcal{O} \Big( \frac{1}{q}
  \Big) A + \mathcal{O} \Big( \frac{1}{q} \Big) \mbox{E} \Big(
  \frac{1}{\lambda} \Big) + \mathcal{O} \Big( \frac{1}{\rb} \Big)
  \mbox{E} \Big( \frac{1}{\tau} \Big) \notag
\\ = & \mathcal{O}(1) +
  \mathcal{O} \Big( \frac{1}{q} \Big) \bigg[ \mathcal{O} \Big(
    \frac{p}{\rb} \Big) \mathrm{E}\Big( \frac{1}{\tau} \Big) +
    \mathcal{O}(q) \bigg] + \mathcal{O} \Big( \frac{1}{q} \Big)
  \mbox{E} \Big( \frac{1}{\lambda} \Big) + \mathcal{O} \Big(
  \frac{1}{\rb} \Big) \mbox{E} \Big( \frac{1}{\tau} \Big) \notag
\\ = & \mathcal{O}(1) + \mathcal{O} \Big( \frac{1}{q} \Big) \mbox{E}
  \Big( \frac{1}{\lambda} \Big) + \mathcal{O} \Big( \frac{1}{\rb}
  \Big) \mbox{E} \Big( \frac{1}{\tau} \Big) \;,
\end{align}
where the last line follows from $(B_4)$.  Now using the definition of
$\lambda$ from the random mapping, we have
\begin{align}
  \label{eq:drift_8}
  \mbox{E} \Big( \frac{1}{\lambda} \Big) & = \frac{2b_1+
    \sum_{i=1}^{q} (\eta_i - \eta_0/\sqrt{q})^2}{q+2a_1-2} \notag \\ &
  \le \frac{2b_1 + 3 \sum_{i=1}^{q} (\eta_i + \bar{y} - \bar{y}_i)^2 +
    3 \sum_{i=1}^{q} (\bar{y}_i - \bar{y})^2 + \frac{3}{q}
    \sum_{i=1}^{q} \eta_0^2}{q+2a_1-2} \notag \\ & \le \frac{2b_1 + 6
    q \bar{y}^2 + 6 \sum_{i=1}^{q} \bar{y}_i^2}{q+2a_1-2} + \frac{3
    \sum_{i=1}^{q} (\eta_i + \bar{y} - \bar{y}_i)^2 + 3
    \eta_0^2}{q+2a_1-2} \notag \\ & \le \frac{2b_1 + 6 q \bar{y}^2 + 6
    \sum_{i=1}^{q} \bar{y}_i^2}{q+2a_1-2} + \frac{3 \sum_{i=1}^{q} r_i
    (\eta_i + \bar{y} - \bar{y}_i)^2 + 3 r_{\min}
    \eta_0^2}{r_{\min}(q+2a_1-2)} \notag \\ & \le \frac{2b_1 + 6 q
    \bar{y}^2 + 6 \sum_{i=1}^{q} \bar{y}_i^2}{q+2a_1-2} + \frac{3 q
    \rb V(\eta)}{r_{\min}(q+2a_1-2)} \notag \\ & = \mathcal{O}(1) +
  \mathcal{O}(1) V(\eta) \;.
\end{align}
Similarly,
\[
  \mbox{E} \Big( \frac{1}{\tau} \Big) = \frac{2b_2+ \sum_{i=1}^q
    \sum_{j=1}^{r_i} (y_{ij} - x_{ij}^\intercal \eta_{00}/\sqrt{q} -
    \eta_i)^2}{N+2a_2-2} \;.
\]
Now
\begin{align*}
  & \sum_{i=1}^q \sum_{j=1}^{r_i}  (y_{ij} - x_{ij}^\intercal
  \eta_{00}/\sqrt{q} - \eta_i)^2 \\ = & \sum_{i=1}^q \sum_{j=1}^{r_i}
  \big[ (\eta_i + \bar{y} - \bar{y}_i) + (x_{ij}^\intercal
    \eta_{00}/\sqrt{q} - \bar{y}) + (\bar{y}_i - y_{ij}) \big]^2 \\
  \le & 3 \sum_{i=1}^q r_i (\eta_i + \bar{y} - \bar{y}_i)^2 + 3
  \sum_{i=1}^q \sum_{j=1}^{r_i} (x_{ij}^\intercal \eta_{00}/\sqrt{q} -
  \bar{y})^2 + 3 \sum_{i=1}^q \sum_{j=1}^{r_i} (y_{ij} - \bar{y}_i)^2
  \\ \le & 3 q \rb V(\eta) + 3 \sum_{i=1}^q \sum_{j=1}^{r_i} (y_{ij} -
  \bar{y}_i)^2 \;.
\end{align*}
Thus,
\begin{align}
  \label{eq:drift_9}
  \mbox{E} \Big( \frac{1}{\tau} \Big) &\le \frac{2b_2 + 3 q \rb V(\eta)
    + 3 \sum_{i=1}^q \sum_{j=1}^{r_i} (y_{ij} -
    \bar{y}_i)^2}{N+2a_2-2} \notag \\
    & = \mathcal{O}(1) + \mathcal{O}(1) V(\eta)
  \;,
\end{align}
where the last line follows from $(B_3)$ and Remark~\ref{rem:1}.
Combining \eqref{eq:drift_7}, \eqref{eq:drift_8}, and
\eqref{eq:drift_9}, we have
\begin{align*}
  \int_{\mathbb{R}^{p+q+1}} V(\tilde{\eta}) K(\eta, d\tilde{\eta}) & =
  \mathcal{O}(1) + \mathcal{O} \Big( \frac{1}{q} \Big) \mbox{E} \Big(
  \frac{1}{\lambda} \Big) + \mathcal{O} \Big( \frac{1}{\rb} \Big)
  \mbox{E} \Big( \frac{1}{\tau} \Big) \\ & = \mathcal{O}(1) +
  \mathcal{O} \Big( \frac{1}{q} \Big) \Big[ \mbox{E} \Big(
    \frac{1}{\lambda} \Big) + \mbox{E} \Big( \frac{1}{\tau} \Big)
    \Big] \\ & = \mathcal{O}(1) + \mathcal{O} \Big( \frac{1}{q} \Big)
  \Big[ \mathcal{O}(1) + \mathcal{O}(1) V(\eta) \Big] \\ & =
  \mathcal{O}(1) + \mathcal{O} \Big( \frac{1}{q} \Big) V(\eta) \;,
\end{align*}
where the second equality follows from $(B_5)$.  Hence, there exist
$\zeta = \zeta(q) = \mathcal{O}(q^{-1})$ and a constant $L$ such that
that
\[
  \int_{\mathbb{R}^{p+q+1}} V(\tilde{\eta}) K(\eta, d\tilde{\eta})
  \leq \zeta V(\eta) + L \,,
\]
and this completes the proof.
\end{proof}

\section{Contraction Condition: Proof of Proposition~\ref{prop:contraction}}
\label{sec:contraction}

We restate Proposition~\ref{prop:contraction} for convenience.

\begin{proposition}
  Assume that $(B_1)$-$(B_5)$ hold, and define
\[
{\cal C} = \big \{ (\eta,\eta') \in \mathbb{R}^{p+q+1} \times
\mathbb{R}^{p+q+1} : V(\eta) + V(\eta') \le q^{\delta/3} \big \} \;,
\]
where $V(\cdot)$ is the drift function defined in \eqref{eq:df}, and
$\delta$ is given in $(B_5)$.  Let $f$ be the random mapping defined
in Section~\ref{sec:main}.  There exist
\[
\gamma = \gamma(q) = \mathcal{O} \bigg( \sqrt{\frac{q^{2 +
      \delta}}{\bar{r}}} \vee \frac{1}{\sqrt{q}} \bigg) \;\;\;\;
\mbox{and} \;\;\;\; \gamma_0 = \gamma_0(q) = \mathcal{O}(\sqrt{q})
\]
and a positive integer $q_0$ such that for all $q \ge q_0$, we have
\[
\sup_{s\in [0,1]} \mathrm{E} \, \bigg \| \frac{\df f(\eta + s(\eta' -
  \eta))}{\df s} \bigg\| \leq
      \begin{cases}
          \gamma \|\eta'-\eta\| & (\eta, \eta') \in
            C \\ \gamma_0 \|\eta'-\eta\| &\text{otherwise} \,.
      \end{cases}
\]
\end{proposition}

\begin{proof}
We will shorten $(\eta + s(\eta' -\eta))$ to $(\eta + s \alpha)$ where
$\alpha = (\eta' - \eta) = (\alpha_{00}^\intercal, \alpha_0, \alpha_1,
\dots, \alpha_q)^\intercal \in \mathbb{R}^{p+q+1}$.  Now
\[
f(\eta + s  \alpha) =
\begin{bmatrix}
\tilde{\eta}_{00}^{(\eta + s  \alpha)} \\
\tilde{\eta}_0^{(\eta + s  \alpha)} \\
\tilde{\eta}_1^{(\eta + s  \alpha)} \\
\vdots \\
\tilde{\eta}_q^{(\eta + s  \alpha)} \\
\end{bmatrix} \;.
\]
By Jensen's inequality, we have
\begin{align}
  \label{eq:outline}
  \Bigg\{ \mathrm{E} \bigg\| \frac{\df f(\eta + s \alpha)}{\df s}
  \bigg\| \Bigg\}^2 & \leq \mathrm{E} \bigg\| \frac{\df f(\eta + s
    \alpha)}{\df s} \bigg\|^2 \notag \\ & = \mathrm{E} \bigg\|
  \frac{\df \tilde{\eta}_{00}^{(\eta + s \alpha)} }{\df s} \bigg\|^2 +
  \mathrm{E} \bigg( \frac{\df \tilde{\eta}_0^{(\eta + s \alpha)}}{\df
    s} \bigg)^2 + \sum_{i=1}^q \mathrm{E} \bigg( \frac{\df
    \tilde{\eta}_i^{(\eta + s \alpha)}}{\df s} \bigg)^2 \;.
\end{align}
The next part of the proof is an extremely long and tedious
development of upper bounds for the three terms on the right-hand side
of \eqref{eq:outline}.  We relegate these calculations to
Appendix~\ref{app:C}, and simply state the resulting bound here, but
we must first introduce a bit of notation.  (In order to simplify
notation, we will sometimes omit the superscript $\eta + s\alpha$,
which should not cause any confusion.)  Define $\phi = \lambda/(\rb
\tau)$ and
\[
\Phi = \Big( \frac{\lambda^3}{\rb^2 \tau^2 J_1} + \frac{\lambda^2}{\rb
  \tau J_2} \Big) \norm{\alpha}^2 = \Big( \frac{\phi^2 \lambda}{J_1} +
\frac{\phi \lambda}{J_2} \Big) \norm{\alpha}^2 \;.
\]
Let $e_i = t_i/(r_i \tau)$, and define $e_{\min} = t_{\min}/(r_{\min}
\tau)$ and $e_{\max} = t_{\max}/(r_{\max} \tau)$.  So $e_{\min}$ and
$e_{\max}$ are actually the largest and smallest $e_i$, respectively.
Finally, define
\[
U_1 = \mathrm{E} \Big( \frac{\Phi}{e_{\max}^2} \Big) \;, \;\;\;\; U_2
= \mathrm{E} \Big( \frac{\Phi}{\tau e_{\max}^4} \Big) \;, \;\;\;\; U_3
= \mathrm{E} \bigg[ \big( \phi \wedge 1 \big) \frac{1}{J_1} +
  \frac{1}{J_2} \bigg] \|\alpha\|^2 \;,
\]
and
\[  U_4 =
\mathrm{E} \Big( \frac{\Phi}{\lambda e_{\max}^3} \Big) \;.
\]
Here is the key bound, which is derived in Appendix~\ref{app:C}:
\begin{equation}
  \label{eq:outline2}
  \Bigg\{ \mathrm{E} \bigg\| \frac{\df f(\eta + s \alpha)}{\df s}
  \bigg\| \Bigg\}^2 \leq \mathcal{O}(q) U_1 + \mathcal{O} \Big(
  \frac{p}{\rb} \Big) U_2 + \mathcal{O}(q^2) U_3 + \mathcal{O}(1) U_4
  + \mathcal{O} \Big( \frac{1}{q} \Big) \norm{\alpha}^2 \;.
\end{equation}
The next step is to develop upper bounds on $U_i$, $i=1,2,3,4$, in two
different cases: $(\eta,\eta') \in {\cal C}$ and $(\eta,\eta') \notin
{\cal C}$.  Define $g: \mathbb{R}^{q+p+1} \rightarrow [0,\infty)$ as
  follows
\[
g(w) = \frac{1}{2} \sum_{i=1}^q \sum_{j=1}^{r_i} \big[ y_{ij} -
  x_{ij}^\intercal w_{00}/\sqrt{q} - w_i \big]^2 \;,
\]
where $w \in \mathbb{R}^{q+p+1}$, $w_{00}$ denotes the first $p$
elements of $w$, and for $i=0,1,\dots,q$, $w_i$ denotes the
$(p+1+i)$th element.  Recall that our drift function $V:
\mathbb{R}^{p+q+1} \rightarrow [0, \infty)$ is given by
\begin{equation*}
  V(w) = \frac{1}{N} \sum_{i=1}^{q} \sum_{j=1}^{r_i} \big( \bar{y} -
  x_{ij}^\intercal w_{00}/\sqrt{q} \big)^2 + \frac{w_0^2}{q} +
  \frac{1}{N} \sum_{i=1}^{q} r_i (w_i + \bar{y} - \bar{y}_i)^2 \;.
\end{equation*}
We now show that $g$ can be bounded by a linear function of $V$.
Indeed,
\begin{align*}
& g(w) \\
& = \frac{1}{2}
\sum_{i=1}^q \sum_{j=1}^{r_i} \Big[ (w_i + \bar{y} - \bar{y}_i) +
  (x_{ij}^\intercal w_{00}/\sqrt{q} - \bar{y}) + (\bar{y}_i - y_{ij})
  \Big]^2 \\
& \le \sum_{i=1}^q \sum_{j=1}^{r_i} \Big[ (w_i + \bar{y}
       - \bar{y}_i) + (x_{ij}^\intercal w_{00}/\sqrt{q} - \bar{y}) \Big]^2
  + \sum_{i=1}^q \sum_{j=1}^{r_i} (\bar{y}_i - y_{ij})^2 \\
& \le
    2 \sum_{i=1}^q r_i (w_i + \bar{y} - \bar{y}_i)^2
  + 2 \sum_{i=1}^q \sum_{j=1}^{r_i} (x_{ij}^\intercal w_{00}/\sqrt{q} - \bar{y})^2
  + \sum_{i=1}^q \sum_{j=1}^{r_i}  y_{ij}^2 \\
& \le 2 N V(w) + N \ell \;.
\end{align*}
Recall from the random mapping that
\[
\tau = \tau^{(\eta + s \alpha)} = \frac{J_2}{b_2 + g(\eta + s \alpha)} \;,
\]
where $J_2 \sim \mathrm{Gamma}(N/2+a_2, 1)$.  Also recall that
\[
{\cal C} = \big \{ (\eta,\eta') \in \mathbb{R}^{p+q+1} \times
\mathbb{R}^{p+q+1} : V(\eta) + V(\eta') \le q^{\delta/3} \big \} \;.
\]
A straightforward calculation shows that $V$ is a convex function.
\\
\\
Now if $(\eta,\eta') \in {\cal C}$, then $V(\eta) \le q^{\delta/3}$
and $V(\eta') \le q^{\delta/3}$, and convexity implies that $V \big(
\eta + s(\eta' - \eta) \big) \le q^{\delta/3}$ for all $s \in [0,1]$.
Consequently, whenever $(\eta,\eta') \in {\cal C}$, we have
\[
g(\eta + s \alpha) \le 2 N V(\eta + s \alpha) + N \ell \le 2 \rb
q^{1+\delta/3} + \rb q \ell = \mathcal{O}(\rb q^{1+\delta/3}) \;.
\]
Note that this bound is free of $\eta$, $\eta'$ and $s$.  Clearly,
$\lambda = \lambda^{(\eta + s \alpha)} \le J_1/b_1$.  Hence, for
$(\eta,\eta') \in {\cal C}$, we have
\[
\phi = \frac{\lambda}{\rb \tau} \le \frac{J_1}{b_1} \frac{\big[b_2 +
    g(\eta+s\alpha) \big]}{\rb J_2} = \frac{J_1}{b_1}
\frac{\mathcal{O}(\rb q^{1+\delta/3})}{\rb J_2} =
\mathcal{O}(q^{1+\delta/3}) \frac{J_1}{J_2} \;,
\]
and
\begin{align*}
\Phi = \Big( \frac{\phi^2 \lambda}{J_1} + \frac{\phi \lambda}{J_2}
\Big) \norm{\alpha}^2
= \mathcal{O}(q^{2+2\delta/3})
\frac{J_1^2}{J_2^2} \norm{\alpha}^2 + \mathcal{O}(q^{1+\delta/3})
\frac{J_1^2}{J_2^2} \norm{\alpha}^2
= \mathcal{O}(q^{2+2\delta/3})
\frac{J_1^2}{J_2^2} \norm{\alpha}^2 \;.
\end{align*}
Continuing under the assumption that $(\eta,\eta') \in {\cal C}$, we
have
\begin{align*}
U_1 = \mathrm{E} \Big( \frac{\Phi}{e_{\max}^2} \Big) \le
\mathrm{E}(\Phi) &= \mathcal{O}(q^{2+2\delta/3}) \mathrm{E} \Big(
\frac{J_1^2}{J_2^2} \Big) \norm{\alpha}^2 \\
& =
\mathcal{O}(q^{2+2\delta/3}) \mathcal{O} \Big( \frac{q^2}{\rb^2 q^2}
\Big) \norm{\alpha}^2 \\
& = \mathcal{O} \Big( \frac{q^{2+\delta}}{\rb^2}
\Big) \norm{\alpha}^2 \;,
\end{align*}
\begin{align*}
U_2 = \mathrm{E} \Big( \frac{\Phi}{\tau e_{\max}^4} \Big) \le
\mathrm{E} \Big( \frac{\Phi}{\tau} \Big) & =
\mathcal{O}(q^{2+2\delta/3}) \mathrm{E} \Big( \frac{1}{\tau}
\frac{J_1^2}{J_2^2} \Big) \norm{\alpha}^2 \\ & =
\mathcal{O}(q^{2+2\delta/3}) \mathrm{E} \Big( \frac{b_2 + g(\eta + s
  \alpha)}{J_2} \frac{J_1^2}{J_2^2} \Big) \norm{\alpha}^2 \\
  & =
\mathcal{O}(q^{2+2\delta/3}) \mathcal{O}(\rb q^{1+\delta/3})
\mathrm{E} \Big( \frac{J_1^2}{J_2^3} \Big) \norm{\alpha}^2 \\
& = \mathcal{O}(\rb q^{3 + \delta})
\mathcal{O} \Big( \frac{q^2}{\rb^3 q^3} \Big) \norm{\alpha}^2 \\
& =
\mathcal{O} \Big( \frac{q^{2 + \delta}}{\rb^2} \Big) \norm{\alpha}^2
\;,
\end{align*}
\begin{align*}
U_3 = \mathrm{E} \bigg[ \big( \phi \wedge 1 \big) \frac{1}{J_1} +
  \frac{1}{J_2} \bigg] \|\alpha\|^2 & \le \mathrm{E} \bigg[
  \frac{\phi}{J_1} + \frac{1}{J_2} \bigg] \|\alpha\|^2 \\ & =
\mathrm{E} \bigg[ \mathcal{O}(q^{1+\delta/3}) \frac{1}{J_2} +
  \frac{1}{J_2} \bigg] \|\alpha\|^2 \\ & = \mathcal{O}(q^{1+\delta/3})
\mathrm{E} \Big( \frac{1}{J_2} \Big) \|\alpha\|^2 \\ & = \mathcal{O}
\Big( \frac{q^{\delta/3}}{\rb} \Big) \|\alpha\|^2 \;,
\end{align*}
and
\begin{align*}
U_4 = \mathrm{E} \Big( \frac{\Phi}{\lambda e_{\max}^3} \Big) \le
\mathrm{E} \Big( \frac{\Phi}{\lambda} \Big)
  & = \mathrm{E} \Big(
\frac{\phi^2}{J_1} + \frac{\phi}{J_2} \Big) \norm{\alpha}^2 \\
  & =
\mathrm{E} \Big( \mathcal{O} (q^{2+2\delta/3}) \frac{J_1}{J_2^2} +
\mathcal{O} (q^{1+\delta/3}) \frac{J_1}{J_2^2} \Big) \norm{\alpha}^2
\\
  & = \mathcal{O}
(q^{2+2\delta/3}) \mathcal{O} \Big( \frac{q}{\rb^2 q^2} \Big)
\norm{\alpha}^2 \\ & = \mathcal{O} \Big( \frac{q^{1+2\delta/3}}{\rb^2}
\Big) \norm{\alpha}^2 \;.
\end{align*}
Therefore, using \eqref{eq:outline2} and $(B_5)$, when $(\eta,\eta')
\in {\cal C}$, we have
\begin{align}
  \label{eq:outline3}
  & \Bigg\{ \mathrm{E} \bigg\| \frac{\df f(\eta + s \alpha)}{\df s}
  \bigg\| \Bigg\}^2 \notag \\
  = & \mathcal{O}(q) U_1 + \mathcal{O} \Big(
  \frac{p}{\rb} \Big) U_2 + \mathcal{O}(q^2) U_3 + \mathcal{O}(1) U_4
  + \mathcal{O} \Big( \frac{1}{q} \Big) \norm{\alpha}^2 \notag \\ = &
  \bigg[ \mathcal{O} \Big( \frac{q^{3+\delta}}{\rb^2} \Big) +
    \mathcal{O} \Big( \frac{q^{3+\delta}}{\rb^3} \Big) + \mathcal{O}
    \Big( \frac{q^{2+\delta/3}}{\rb} \Big) + \mathcal{O} \Big(
    \frac{q^{1+2\delta/3}}{\rb^2} \Big) \bigg] \norm{\alpha}^2 +
  \mathcal{O} \Big( \frac{1}{q} \Big) \norm{\alpha}^2 \notag \\
  = &
  \mathcal{O} \Big( \frac{q^{2+\delta}}{\rb} \vee \frac{1}{q} \Big)
  \norm{\alpha}^2 \;.
\end{align}
Our final major task is to develop bounds for $U_i$, $i=1,2,3,4$, in
the case where $(\eta,\eta') \notin {\cal C}$.  First,
\[
e_{\max} = 1 + \frac{\lambda}{r_{\max} \tau} \ge 1 + \frac{\lambda}{m
  \rb \tau} \ge \frac{1+\phi}{m} \;,
\]
and it follows that
\begin{align*}
U_1 = \mathrm{E} \Big( \frac{\Phi}{e_{\max}^2} \Big) \le \mathrm{E}
\bigg[ \frac{m^2}{(1+\phi)^2} \Big( \frac{\phi^2 \lambda}{J_1} +
  \frac{\phi \lambda}{J_2} \Big) \bigg] \norm{\alpha}^2 & \le m^2
\mathrm{E} \Big( \frac{\lambda}{J_1} + \frac{\lambda}{J_2} \Big)
\norm{\alpha}^2 \\ & \le m^2 \mathrm{E} \Big( \frac{1}{b_1} +
\frac{J_1}{b_1 J_2} \Big) \norm{\alpha}^2 \\ & = m^2 \mathcal{O}(1)
\norm{\alpha}^2 \\ & = \mathcal{O}(1) \norm{\alpha}^2 \;.
\end{align*}
Now,
\[
\frac{1}{\rb \tau e_{\max}} = \frac{1}{\rb \tau} \frac{r_{\max}
  \tau}{t_{\max}} \le \frac{m}{t_{\max}} \;,
\]
and it follows that
\begin{align*}
U_2 = \mathrm{E} \Big( \frac{\Phi}{\tau e_{\max}^4} \Big) = \rb
\mathrm{E} \Big( \frac{\Phi}{\rb \tau e_{\max}^4} \Big) & \le \rb
\mathrm{E} \bigg[ \frac{m}{t_{\max}} \Big( \frac{m}{1+\phi} \Big)^3
  \Big( \frac{\phi^2 \lambda}{J_1} + \frac{\phi \lambda}{J_2} \Big)
  \bigg] \norm{\alpha}^2 \\ & \le \rb m^4 \mathrm{E} \Big(
\frac{1}{J_1} + \frac{1}{J_2} \Big) \norm{\alpha}^2 \\ & = \rb
\mathcal{O} \Big( \frac{1}{q} \Big) \norm{\alpha}^2 \\ & = \mathcal{O}
\Big( \frac{\rb}{q} \Big) \norm{\alpha}^2 \;.
\end{align*}
Clearly,
\[
U_3 = \mathrm{E} \bigg[ \big( \phi \wedge 1 \big) \frac{1}{J_1} +
  \frac{1}{J_2} \bigg] \|\alpha\|^2 \le \mathrm{E} \Big( \frac{1}{J_1}
+ \frac{1}{J_2} \Big) \|\alpha\|^2 = \mathcal{O} \Big( \frac{1}{q}
\Big) \|\alpha\|^2 \;.
\]
Finally,
\begin{align*}
U_4 = \mathrm{E} \Big( \frac{\Phi}{\lambda e_{\max}^3} \Big) \le
\mathrm{E} \bigg[ \frac{1}{\lambda} \Big( \frac{m}{1+\phi} \Big)^3
  \Big( \frac{\phi^2 \lambda}{J_1} + \frac{\phi \lambda}{J_2} \Big)
  \bigg] \|\alpha\|^2
\le m^3 \mathrm{E} \Big( \frac{1}{J_1} +
\frac{1}{J_2} \Big) \|\alpha\|^2
= \mathcal{O} \Big( \frac{1}{q} \Big)
\|\alpha\|^2 \;.
\end{align*}
Thus, it follows from \eqref{eq:outline2}, $(B_4)$, and $(B_5)$ that
\begin{align}
  \label{eq:outline4}
  \Bigg\{ \mathrm{E} \bigg\| \frac{\df f(\eta + s \alpha)}{\df s}
  \bigg\| \Bigg\}^2 & = \mathcal{O}(q) U_1 + \mathcal{O} \Big(
  \frac{p}{\rb} \Big) U_2 + \mathcal{O}(q^2) U_3 + \mathcal{O}(1) U_4
  + \mathcal{O} \Big( \frac{1}{q} \Big) \norm{\alpha}^2 \notag \\ & =
  \bigg[ \mathcal{O}(q) + \mathcal{O}(1) + \mathcal{O}(q) +
    \mathcal{O} \Big( \frac{1}{q} \Big) \bigg] \norm{\alpha}^2 +
  \mathcal{O} \Big( \frac{1}{q} \Big) \norm{\alpha}^2 \notag \\ & =
  \mathcal{O}(q) \norm{\alpha}^2 \;.
\end{align}
Combining \eqref{eq:outline3} and \eqref{eq:outline4}, we have
\begin{equation*}
  \Bigg\{ \mathrm{E} \bigg\| \frac{\df f(\eta + s \alpha)}{\df s}
  \bigg\| \Bigg\}^2 =
  \begin{cases}
      \mathcal{O}\big( \frac{q^{2 + \delta}}{\bar{r}} \vee \frac{1}{q}
      \big) \|\alpha\|^2 \qquad & \text{if } (\eta, \eta') \in
      \mathcal{C} \\ \mathcal{O}(q) \|\alpha\|^2 \qquad &
      \text{otherwise.}
   \end{cases}
\end{equation*}
Since the right-hand side does not depend on $s$, the proposition follows.
\end{proof}

\newpage

\noindent{\LARGE \bf Appendix}
\appendix

\section{Derivations of Conditional Distributions}
\label{app:A}

Here we derive the distributions of $\eta_{00} \,|\, \lambda,\tau, Y$
and $\eta_0 \,|\, \eta_{00}, \lambda,\tau, Y$.  Recall that $t_i =
\lambda + r_i \tau$.  First, the conditional density of $(\eta_{00},
\eta_0)$ given $(\lambda,\tau, Y)$ is proportional to the following expression (where the constant of proportionality can depend on $\lambda$, $\tau$ and $Y$),
\begin{align}
& \int_{\mathbb{R}^q} \pi(\eta \,|\, \lambda,\tau, Y) \, d\eta_1
  \cdots d\eta_q \notag \\
\propto & \int_{\mathbb{R}^q} \exp \bigg\{
  -\frac{\lambda}{2} \sum_{i=1}^{q}(\eta_i - \eta_0/\sqrt{q})^2
  -\frac{\tau}{2} \sum_{i=1}^{q}\sum_{j=1}^{r_i} (\eta_i +
  x_{ij}^\intercal \eta_{00}/\sqrt{q} - y_{ij})^2 \bigg\} d\eta_1
  \cdots d\eta_q \notag \\
= & \prod_{i=1}^q \int_{\mathbb{R}} \exp
  \bigg\{ -\frac{\lambda}{2} (\eta_i - \eta_0/\sqrt{q})^2
  -\frac{\tau}{2} \sum_{j=1}^{r_i} (\eta_i + x_{ij}^\intercal
  \eta_{00}/\sqrt{q} - y_{ij})^2 \bigg\} d\eta_i \notag \\
= &
  \prod_{i=1}^{q} \exp \bigg\{ - \frac{\lambda}{2} \frac{\eta_0^2}{q}
  - \frac{\tau}{2} \sum_{j=1}^{r_i} (x_{ij}^\intercal
  \eta_{00}/\sqrt{q} - y_{ij})^2 \bigg\} \notag \\
& \hspace*{25mm} \times
  \int_{\mathbb{R}} \exp \bigg\{ - \frac{t_i}{2} \bigg[ \eta_i -
    \frac{\lambda \eta_0/\sqrt{q} + r_i \tau (\bar{y}_i -
      \bar{x}_i^\intercal \eta_{00}/\sqrt{q})}{t_i} \bigg]^2 \notag \\
& \hspace*{45mm} +
  \frac{t_i}{2} \bigg( \frac{\lambda \eta_0/\sqrt{q} + r_i \tau
    (\bar{y}_i - \bar{x}_i^\intercal \eta_{00}/\sqrt{q})}{t_i}
  \bigg)^2 \bigg\} \, d\eta_i \label{eq:a_0} \\
\propto & \prod_{i=1}^{q} \exp
  \bigg\{ - \frac{\lambda}{2} \frac{\eta_0^2}{q} - \frac{\tau}{2}
  \sum_{j=1}^{r_i} (x_{ij}^\intercal \eta_{00}/\sqrt{q} - y_{ij})^2
  \bigg\}  \notag \\
& \hspace*{40mm} \times \exp \bigg\{ \frac{t_i}{2} \bigg( \frac{\lambda
    \eta_0/\sqrt{q} + r_i \tau (\bar{y}_i - \bar{x}_i^\intercal
    \eta_{00}/\sqrt{q})}{t_i} \bigg)^2 \bigg\} \notag \\
= &
  \prod_{i=1}^{q} \exp \bigg\{ - \frac{\lambda}{2} \frac{\eta_0^2}{q}
  + \frac{\lambda^2}{2 t_i} \frac{\eta_0^2}{q} + \frac{r_i \lambda
    \tau}{\sqrt{q} t_i} (\bar{y}_i - \bar{x}_i^\intercal
  \eta_{00}/\sqrt{q}) \eta_0 \bigg\} \notag \\
& \hspace*{15mm} \times
  \exp \bigg\{ - \frac{\tau}{2} \sum_{j=1}^{r_i} (x_{ij}^\intercal
  \eta_{00}/\sqrt{q} - y_{ij})^2 + \frac{(r_i \tau)^2}{2 t_i}
  (\bar{y}_i - \bar{x}_i^\intercal \eta_{00}/\sqrt{q})^2 \bigg\} \;. \label{eq:a_1}
\end{align}
Note that the conditional density of $\eta_i$ given $(\eta_0, \eta_{00}, \lambda, \tau, Y)$ can be gleaned from \eqref{eq:a_0}.
Now recall that $z_i = t_i/(r_i \lambda \tau)$.  It follows from
\eqref{eq:a_1} that the conditional density of $\eta_0$ given
$(\eta_{00},\lambda,\tau, Y)$ is proportional to
\begin{align*}
& \prod_{i=1}^{q} \exp \bigg\{ - \frac{\lambda}{2} \Big(1 -
\frac{\lambda}{t_i} \Big) \frac{\eta_0^2}{q} + \frac{r_i \lambda
  \tau}{\sqrt{q} t_i} (\bar{y}_i - \bar{x}_i^\intercal
\eta_{00}/\sqrt{q}) \eta_0 \bigg\} \\
= & \prod_{i=1}^{q} \exp \bigg\{ -
\frac{1}{2 q z_i} \eta_0^2 + \frac{1}{\sqrt{q} z_i} (\bar{y}_i -
\bar{x}_i^\intercal \eta_{00}/\sqrt{q}) \eta_0 \bigg\} \\
= & \exp
\bigg\{ - \eta_0^2 \sum_{i=1}^{q} \frac{1}{2 q z_i} + \eta_0
\sum_{i=1}^{q} \frac{1}{\sqrt{q} z_i} (\bar{y}_i - \bar{x}_i^\intercal
\eta_{00}/\sqrt{q}) \bigg\} \;.
\end{align*}
Thus,
\[
\eta_0 \,|\, \eta_{00},\lambda,\tau, Y \sim \mbox{N} \bigg(
\frac{\sqrt{q} \sum_{i=1}^q (\bar{y}_i - \bar{x}^\intercal_i
  \eta_{00}/\sqrt{q})/z_i}{\sum_{i=1}^q 1/z_i}, \frac{q}{\sum_{i=1}^q
  1/z_i} \bigg) \;.
\]
The conditional density of $\eta_{00}$ given $(\lambda,\tau, Y)$ is
proportional to the integral of \eqref{eq:a_1} with respect to
$\eta_0$, which is given by
\begin{align}
  \label{eq:a_2}
  & \exp \bigg\{ \frac{1}{2} \bigg( \sum_{i=1}^q \frac{1}{z_i}
  \bigg)^{-1} \bigg( \sum_{i=1}^q (\bar{y}_i - \bar{x}_i^\intercal
  \eta_{00}/\sqrt{q})/z_i \bigg)^2 \bigg\} \notag \\
  & \hspace*{10mm} \times \exp
  \bigg\{ - \frac{\tau}{2} \sum_{i=1}^q \sum_{j=1}^{r_i}
  (x_{ij}^\intercal \eta_{00}/\sqrt{q} - y_{ij})^2 + \sum_{i=1}^q
  \frac{(r_i \tau)^2}{2 t_i} (\bar{y}_i - \bar{x}_i^\intercal
  \eta_{00}/\sqrt{q})^2 \bigg\} \notag \\
& = \exp \bigg\{
  \frac{\tau}{2} \bigg( \sum_{i=1}^q \sum_{j=1}^{r_i}
  \frac{\lambda}{t_i} \bigg)^{-1} \bigg( \sum_{i=1}^q \sum_{j=1}^{r_i}
  \frac{\lambda (\bar{y}_i - \bar{x}_i^\intercal
    \eta_{00}/\sqrt{q})}{t_i} \bigg)^2  \notag \\
& \hspace*{10mm} -
  \frac{\tau}{2}(X\eta_{00}/\sqrt{q} - Y)^\intercal
  (X\eta_{00}/\sqrt{q} - Y)
  +
  \frac{\tau}{2} \sum_{i=1}^q \sum_{j=1}^{r_i} \frac{r_i \tau}{t_i}
  (\bar{y}_i - \bar{x}_i^\intercal \eta_{00}/\sqrt{q})^2 \bigg\}
  \notag \\
& = \exp \bigg\{ \frac{\tau}{2} \Big( 1^\intercal (I -
  D_c) 1 \Big)^{-1} \Big( 1^\intercal (I - D_c) (\bar{Y} - \bar{X}
  \eta_{00}/\sqrt{q}) \Big)^{2}  \notag \\
& \hspace*{5mm} - \frac{\tau}{2}(X\eta_{00}/\sqrt{q} -
  Y)^\intercal (X\eta_{00}/\sqrt{q} - Y)
  +
  \frac{\tau}{2} (\bar{X} \eta_{00}/\sqrt{q} - \bar{Y})^\intercal D_c
  (\bar{X} \eta_{00}/\sqrt{q} - \bar{Y}) \bigg \} \notag \\
& = \exp
  \bigg\{ \frac{\tau}{2} (\bar{X} \eta_{00}/\sqrt{q} - \bar{Y}
  )^\intercal \bigg( \frac{(I - D_c) 1 1^\intercal (I -
    D_c)}{1^\intercal (I - D_c) 1} \bigg) (\bar{X} \eta_{00}/\sqrt{q}
  - \bar{Y} ) \notag \\
& \hspace*{5mm} -
  \frac{\tau}{2}(X\eta_{00}/\sqrt{q} - Y)^\intercal
  (X\eta_{00}/\sqrt{q} - Y) + \frac{\tau}{2} (\bar{X}
  \eta_{00}/\sqrt{q} - \bar{Y})^\intercal D_c (\bar{X}
  \eta_{00}/\sqrt{q} - \bar{Y}) \bigg \} \notag \\
& = \exp \bigg\{
  \frac{\tau}{2} (\bar{X} \eta_{00}/\sqrt{q} - \bar{Y} )^\intercal
  \bigg( D_c + \frac{(I - D_c) 1 1^\intercal (I - D_c)}{1^\intercal (I
    - D_c) 1} \bigg) (\bar{X} \eta_{00}/\sqrt{q} - \bar{Y} ) \notag
  \\
& \hspace*{65mm} - \frac{\tau}{2}(X\eta_{00}/\sqrt{q} -
  Y)^\intercal (X\eta_{00}/\sqrt{q} - Y) \bigg \} \;.
\end{align}
Now recall that
\[
M = D_c + \frac{(I - D_c) 1 1^\intercal (I - D_c)}{1^\intercal (I -
  D_c) 1} \;,
\]
so we can rewrite \eqref{eq:a_2} as
\[
  \exp \bigg\{ \frac{\tau}{2} (\bar{X} \eta_{00}/\sqrt{q} - \bar{Y}
  )^\intercal M (\bar{X} \eta_{00}/\sqrt{q} - \bar{Y} ) -
  \frac{\tau}{2}(X\eta_{00}/\sqrt{q} - Y)^\intercal
  (X\eta_{00}/\sqrt{q} - Y) \bigg \} \;.
\]
A simple complete the square argument shows that
\[
\eta_{00} \,|\, \lambda,\tau, Y \sim \mbox{N}_p \Big( v, \frac{q}{\tau}
Q \Big) \;.
\]
where $Q = (X^\intercal X - \bar{X}^\intercal M \bar{X})^{-1}$ and $v = \sqrt{q} \, Q
(X^\intercal Y - \bar{X}^\intercal M \bar{Y})$.

\section{Proofs of Supporting Lemmas}
\label{app:B}

\begin{proof}[Proof of Lemma~\ref{lem:conseq}.]
Part $(i)$: Let
\begin{equation*}
  R = \bigoplus_{i=1}^{q} \frac{1}{r_i} J_{r_i},
\end{equation*}
where, as usual, $J_{r_i}$ denotes an $r_i \times r_i$ matrix of 1s.
Note that $R$ is symmetric and idempotent.  It's easy to see that
$\bar{X} = RX$ and $\bar{Y} = RY$.  Now
\[
X^\intercal X = X^\intercal (R + I - R) X = X^\intercal R^2 X +
X^\intercal (I - R)^2 X \succcurlyeq (RX)^\intercal (RX) =
\bar{X}^\intercal \bar{X} \,.
\]

\noindent Part $(ii)$: By part $(i)$,
\[
\sum_{i=1}^q \bar{x}_i \bar{x}^\intercal_i \preccurlyeq
\frac{1}{r_{\min}} \sum_{i=1}^q \sum_{j=1}^{r_i} \bar{x}_i
\bar{x}^\intercal_i = \frac{1}{r_{\min}} \bar{X}^\intercal \bar{X}
\preccurlyeq \frac{1}{r_{\min}} X^\intercal X \,.
\]

\noindent Part $(iii)$: We have
\[
\sum_{i=1}^{q} \bar{y}_i^2 \le \frac{1}{r_{\min}} \sum_{i=1}^q
\sum_{j=1}^{r_i} \bar{y}_i^2 = \frac{1}{r_{\min}} \bar{Y}^\intercal
\bar{Y} \le \frac{1}{r_{\min}} Y^\intercal Y \,.
\]

\noindent Part $(iv)$: Define
\[a^\intercal = \Big( \frac{1}{q r_1}
1^\intercal_{r_1}, \cdots, \frac{1}{q r_q} 1^\intercal_{r_q} \Big),\]
and note that $\bar{y} = a^\intercal Y$.  Now
\[
a^\intercal a = \sum_{i=1}^q r_i \Big( \frac{1}{q r_i} \Big)^2 \le
\sum_{i=1}^q \frac{1}{q^2 r_{\min}} = \frac{1}{q r_{\min}} \;.
\]
Hence,
\[
\bar{y}^2 = (a^\intercal Y)^2 \le (a^\intercal a) (Y^\intercal Y) \le
\frac{Y^\intercal Y}{q r_{\min}} \;.
\]
\end{proof}

\begin{proof}[Proof of Lemma~\ref{lem:simple_matrix}.]

For any vector $x$,
\begin{align*}
  x^\intercal C^\intercal D x = x^\intercal D^\intercal C x & =
  (Cx)^\intercal Dx \\ & \leq \sqrt{(Cx)^\intercal Cx}
  \sqrt{(Dx)^\intercal Dx} \\ & = \sqrt{x^\intercal C^\intercal C x}
  \sqrt{x^\intercal D^\intercal D x} \\ & \leq \frac{1}{2}
  (x^\intercal C^\intercal C x + x^\intercal D^\intercal D x).
\end{align*}
Therefore, for any $x$,
\begin{align*}
  x^\intercal (C+D)^\intercal (C+D) x & = x^\intercal C^\intercal C x
  + x^\intercal D^\intercal D x + x^\intercal C^\intercal D x +
  x^\intercal D^\intercal C x \\ & \le 2 (x^\intercal C^\intercal C x
  + x^\intercal D^\intercal D x) \\ & = x^\intercal 2 ( C^\intercal C
  + D^\intercal D) x \,.
\end{align*}
\end{proof}

\begin{proof}[Proof of Lemma~\ref{lem:sr}.]
We use contradiction. Suppose that there exists an $x$ such that
\[
\lambda_{\max} \bigg\{ \Big( \frac{\df A^{1/2}}{\df x} \Big)^2
\bigg\} > \frac{\lambda_{\max} \Big\{ \Big( \frac{\df A}{\df x}
  \Big)^2 \Big\} }{4 \lambda_{\min}(A)} \;.
\]
Then either
\begin{enumerate}[label = {case (\arabic*)}:, wide=0pt, leftmargin=*] 
\item $\lambda_{\max} \Big( \frac{\df A^{1/2}}{\df x} \Big) >
  \frac{1}{2 \sqrt{\lambda_{\min}(A)}} \sqrt{\lambda_{\max} \Big\{
    \Big( \frac{\df A}{\df x} \Big)^2 \Big\}}$, or
\item $\lambda_{\min} \Big( \frac{\df A^{1/2}}{\df x} \Big) <
  - \frac{1}{2 \sqrt{\lambda_{\min}(A)}} \sqrt{\lambda_{\max} \Big\{
    \Big( \frac{\df A}{\df x} \Big)^2 \Big\}}$ \;.
\end{enumerate}
Suppose that case (1) holds.  Then there exists $x_0$ such that
\[
\lambda_{\max} \Big( \frac{\df A^{1/2}}{\df x} \Big|_{x = x_0}
\Big) > \frac{1}{2 \sqrt{\lambda_{\min}(A|_{x = x_0})}}
\sqrt{\lambda_{\max} \Big\{ \Big( \frac{\df A}{\df x} \Big|_{x = x_0}
  \Big)^2 \Big\}} \;.
\]
Let $A_0 = A(x_0)$,
\[
D_0 = \frac{\df A}{\df x} \Big|_{x = x_0} \;\;\;\; \mbox{and} \;\;\;\;
R_0 = \frac{\df A^{1/2}}{\df x} \Big|_{x = x_0} \;.
\]
Let $\lambda_0 = \lambda_{\max}(R_0)$, and let $\xi_0$ denote the
corresponding (normalized) eigenvector.  Then
\[
2 \lambda_0 \sqrt{\lambda_{\min}(A_0)} > \sqrt{\lambda_{\max}(D_0^2)}  \;.
\]
Now since $A_0^{1/2}$ and $R_0$ are both symmetric, we have
\begin{align*}
  \xi_0^\intercal D_0 \xi_0 & = \xi_0^\intercal \Big( \frac{\df A}{\df
    x} \Big|_{x = x_0} \Big) \xi_0 \\ & = \xi_0^\intercal \bigg\{
  \Big( \frac{\df A^{1/2}}{\df x} \Big|_{x = x_0} \Big)
  A_0^{1/2} + A_0^{1/2} \Big( \frac{\df
    A^{1/2}}{\df x} \Big|_{x = x_0} \Big) \bigg\} \xi_0 \\ & =
  \xi_0^\intercal \big\{ R_0 A_0^{1/2} + A_0^{1/2} R_0
  \big\} \xi_0 \\ & = 2 \xi_0^\intercal R_0 A_0^{1/2} \xi_0
  \\ & = 2 \lambda_0 \xi_0^\intercal A_0^{1/2} \xi_0 \\ & \ge
  2 \lambda_0 \sqrt{\lambda_{\min}(A_0)} \xi_0^\intercal \xi_0 \\
  & > \sqrt{\lambda_{\max}(D_0^2)} \\ &
  \ge \lambda_{\max}(D_0) \;,
\end{align*}
but this is a contradiction.  An analogous argument shows that case (2) also
leads to a contradiction.
\end{proof}

\section{Development of the Upper Bound \eqref{eq:outline2}}
\label{app:C}

In this Appendix, we develop upper bounds for each of the three terms
on the right-hand side of \eqref{eq:outline}, which together yield
\eqref{eq:outline2}.  Of course, the components of $f(\eta + s
\alpha)$ depend on $\lambda^{(\eta + s \alpha)}$ and $\tau^{(\eta + s
  \alpha)}$, so we begin with these.  According to the random mapping,
we have
\begin{equation*}
  \label{eq:lambda12}
    \begin{split}
      \lambda^{(\eta + s \alpha)} & = \frac{J_1}{b_1+ \frac{1}{2}
        \sum_{i=1}^q (\eta_i + s \alpha_i - (\eta_0 + s \alpha_0)/
        \sqrt{q})^2} \\ \tau^{(\eta + s \alpha)} & = \frac{J_2}{b_2 +
        \frac{1}{2}\sum_{i=1}^q \sum_{j=1}^{r_i} \big[ y_{ij} -
          x_{ij}^\intercal (\eta_{00} + s \alpha_{00})/\sqrt{q} -
          (\eta_i + s \alpha_i) \big]^2} \;.
    \end{split}
\end{equation*}
We will require upper bounds on $\big(\frac{\df \lambda^{(\eta +
    s\alpha)}}{\df s}\big)^2$ and $\big(\frac{\df \tau^{(\eta +
    s\alpha)}}{\df s}\big)^2$.  We have
\begin{align*}
  \frac{\df \lambda^{(\eta + s\alpha)}}{\df s} & = \frac{- J_1
    \sum_{i=1}^q \big( \eta_i + s \alpha_i - (\eta_0 + s
    \alpha_0)/\sqrt{q} \big) (\alpha_i - \alpha_0 /
    \sqrt{q})}{\big[b_1 + \frac{1}{2}\sum_{i=1}^q (\eta_i + s \alpha_i
      - (\eta_0 + s \alpha_0) / \sqrt{p})^2\big]^2} \\ & = \frac{-
    \big( \lambda^{(\eta + s\alpha)} \big)^2}{J_1} \sum_{i=1}^q \big(
  \eta_i + s \alpha_i - (\eta_0 + s \alpha_0)/\sqrt{q} \big) (\alpha_i
  - \alpha_0 / \sqrt{q}) \;.
\end{align*}
Thus, by Cauchy-Schwarz,
\begin{align}
  \label{eq:ls}
  & \bigg(\frac{\df \lambda^{(\eta + s\alpha)}}{\df s}\bigg)^2 \notag \\
  \le &
  \frac{\big( \lambda^{(\eta + s\alpha)} \big)^4}{J^2_1} \sum_{i=1}^q
  \big( \eta_i + s \alpha_i - (\eta_0 + s \alpha_0)/\sqrt{q} \big)^2
  \sum_{i=1}^q (\alpha_i - \alpha_0 / \sqrt{q})^2 \notag \\
  \le &
  \frac{4 \big( \lambda^{(\eta + s\alpha)} \big)^4}{J^2_1} \Big[ b_1 +
    \frac{1}{2} \sum_{i=1}^q \big( \eta_i + s \alpha_i - (\eta_0 + s
    \alpha_0)/\sqrt{q} \big)^2 \Big] \sum_{i=1}^q (\alpha^2_i +
  \alpha^2_0/q) \notag \\
  \le & \frac{4 \big( \lambda^{(\eta +
      s\alpha)} \big)^4}{J^2_1} \frac{J_1}{\lambda^{(\eta + s\alpha)}}
  \norm{\alpha}^2 \notag \\
  = & \frac{4 \big( \lambda^{(\eta +
      s\alpha)} \big)^3}{J_1} \norm{\alpha}^2 \;.
\end{align}
A similar argument yields
\begin{align*}
  \bigg(\frac{\df \tau^{(\eta + s\alpha)}}{\df s}\bigg)^2 & \le
  \frac{4 \big( \tau^{(\eta + s\alpha)} \big)^3}{J_2} \Bigg\{
  \sum_{i=1}^q \sum_{j=1}^{r_i} \alpha_i^2 + \frac{1}{q}
  \alpha_{00}^\intercal \bigg( \sum_{i=1}^q \sum_{j=1}^{r_i} x_{ij}
  x_{ij}^\intercal \bigg) \alpha_{00} \Bigg\} \;.
\end{align*}
Using conditions $(B_1)$, $(B_2)$ and Remark~\ref{rem:1}, we have
\begin{align}
  \label{eq:ts}
  \bigg(\frac{\df \tau^{(\eta + s\alpha)}}{\df s}\bigg)^2 & \le
  \frac{4 \big( \tau^{(\eta + s\alpha)} \big)^3}{J_2} \bigg\{ r_{\max}
  \sum_{i=1}^q \alpha_i^2 + \frac{\rb q k_2}{q} \alpha_{00}^\intercal
  \alpha_{00} \bigg\} \notag \\ & \le \frac{4 \big( \tau^{(\eta +
      s\alpha)} \big)^3}{J_2} \bigg\{ m \rb \sum_{i=1}^q \alpha_i^2 +
  \rb k_2 \alpha_{00}^\intercal \alpha_{00} \bigg\} \notag \\ & \le
  \frac{4 \rb (m + k_2) \big( \tau^{(\eta + s\alpha)} \big)^3}{J_2}
  \norm{\alpha}^2 \;.
\end{align}
It now follows from \eqref{eq:ls} and \eqref{eq:ts} that there exists
a positive constant $\ell_1$ such that
\begin{equation}
  \label{eq:ls&ts}
  \bigg(\frac{\df \lambda^{(\eta + s\alpha)}}{\df s}\bigg)^2 \le
  \frac{\ell_1 \big( \lambda^{(\eta + s\alpha)} \big)^3}{J_1}
  \norm{\alpha}^2 \;\; \mbox{and} \;\; \bigg(\frac{\df
    \tau^{(\eta + s\alpha)}}{\df s}\bigg)^2 \le \frac{\ell_1 \rb \big(
    \tau^{(\eta + s\alpha)} \big)^3}{J_2} \norm{\alpha}^2 \;.
\end{equation}
In order to simplify notation, we omit the superscript $\eta +
s\alpha$ in the remainder of the proof.  This should not cause any
confusion.  We now develop an upper bound for $\mathrm{E} \big\|
\frac{\df \tilde{\eta}_{00}}{\df s} \big\|^2$.  Recall that
$\tilde{\eta}_{00} = v + \sqrt{\frac{q}{\tau}} Q^{1/2}
N_{00}$, where $Q = (X^\intercal X - \bar{X}^\intercal M
\bar{X})^{-1}$ and $v = \sqrt{q} \; Q (X^\intercal Y -
\bar{X}^\intercal M \bar{Y})$.  Since $R X = \bar{X}$, $R Y =
\bar{Y}$, $R^2 = R$, and $R$ and $M$ commute, we have
\[
\tilde{\eta}_{00} = \sqrt{q} \; Q (X^\intercal Y - \bar{X}^\intercal M
\bar{Y}) + \sqrt{\frac{q}{\tau}} Q^{\frac{1}{2}} N_{00} = \sqrt{q} \,
Q X^\intercal (I - M R) Y + \sqrt{\frac{q}{\tau}} Q^{\frac{1}{2}}
N_{00} \;.
\]
So we have
\begin{align*}
\frac{\df \tilde{\eta}_{00}}{\df s}
& = \sqrt{q} \frac{\df}{\df s}
\Big( Q X^\intercal (I - M R) Y + \frac{1}{\tau^{\frac{1}{2}}}
Q^{\frac{1}{2}} N_{00} \Big)
\\ & = \sqrt{q} \Bigg[ \bigg( \frac{\df
    Q}{\df s} \bigg) X^\intercal (I - M R) Y - Q X^\intercal \bigg(
  \frac{\df M}{\df s} \bigg) R Y - \frac{1}{2 \tau^{\frac{3}{2}}}
  \bigg( \frac{\df \tau}{\df s} \bigg) Q^{\frac{1}{2}} N_{00}
 +
  \frac{1}{\tau^{\frac{1}{2}}} \bigg( \frac{\df Q^{\frac{1}{2}}}{\df
    s} \bigg) N_{00} \Bigg] \;.
\end{align*}
Thus,
\begin{align*}
  \Big\| \frac{\df \tilde{\eta}_{00}}{\df s} \Big\|^2 \le 4 q \Big\|
  \bigg( \frac{\df Q}{\df s} \bigg) X^\intercal (I - M R) Y \Big\|^2 +
  4 q \Big\| Q X^\intercal  \bigg( \frac{\df M}{\df s} \bigg) R Y
  \Big\|^2 \\
   + \frac{q (\tau')^2}{\tau^3} \Big\| Q^{\frac{1}{2}}
  N_{00} \Big\|^2 + \frac{4q}{\tau} \Big\| \bigg( \frac{\df
    Q^{\frac{1}{2}}}{\df s} \bigg) N_{00} \Big\|^2 \;,
\end{align*}
where $\tau' = \frac{\df \tau}{\df s}$.  Now define
\begin{align*}
  T_0 & = \mathrm{E} \bigg\{ \frac{q (\tau')^2}{\tau^3} \Big\|
  Q^{\frac{1}{2}} N_{00} \Big\|^2 \bigg\} \;, \\ T_1 & = \mathrm{E}
  \bigg\{ \frac{q}{\tau} \Big\| \bigg( \frac{\df Q^{\frac{1}{2}}}{\df
    s} \bigg) N_{00} \Big\|^2 \bigg\} \;, \\ T_2 & = \mathrm{E}
  \bigg\{ q \Big\| \bigg( \frac{\df Q}{\df s} \bigg) X^\intercal (I -
  M R) Y \Big\|^2 \bigg\} \;, \\ T_3 & = \mathrm{E} \bigg\{ q \Big\| Q
  X^\intercal \bigg( \frac{\df M}{\df s} \bigg) R Y \Big\|^2 \bigg\}
  \;.
\end{align*}
So we have
\begin{equation}
  \label{eq:eta_00}
  \mathrm{E} \Big\| \frac{\df \tilde{\eta}_{00}}{\df s} \Big\|^2 \le
  T_0 + 4T_1 + 4T_2 + 4T_3 \;.
\end{equation}
Using \eqref{eq:ls&ts} and \eqref{eq:Qbnd}, we have
\begin{align}
  \label{eq:T_0_bnd}
  T_0 = \mathrm{E} \bigg\{ \frac{q (\tau')^2}{\tau^3} \Big\|
  Q^{\frac{1}{2}} N_{00} \Big\|^2 \bigg\} & \le \ell_1 \rb q
  \mathrm{E} \Big( \frac{1}{J_2} \Big) \mathrm{E} \big(
  N_{00}^\intercal Q N_{00} \big) \norm{\alpha}^2 \notag \\ & \le
  \frac{\ell_1}{k_1} \mathrm{E} \Big( \frac{1}{J_2} \Big) \mathrm{E}
  \big( N_{00}^\intercal N_{00} \big) \norm{\alpha}^2 \notag \\ & =
  \frac{2 \ell_1 p }{k_1(N + 2a_2 -2)} \norm{\alpha}^2 \notag \\ & =
  \mathcal{O} \Big( \frac{p}{\rb q} \Big) \norm{\alpha}^2 \;.
\end{align}
We now develop some bounds that will allow us to handle $T_1$, $T_2$
\& $T_3$.  Recall that $e_i = c_i^{-1} = t_i/(r_i \tau)$, and that
$e_{\min} = t_{\min}/(r_{\min} \tau)$ and $e_{\max} =
t_{\max}/(r_{\max} \tau)$.  So $e_{\min}$ and $e_{\max}$ are actually
the largest and smallest $e_i$, respectively.  Now,
\begin{equation}
  \label{eq:dD}
  \frac{\df D_c}{\df s} = \frac{\df}{\df s} \Big( \bigoplus_{i=1}^q
  \frac{1}{e_i} I_{r_i} \Big) = - \bigoplus_{i=1}^q \frac{1}{e_i^2}
  \Big(\frac{\df e_i}{\df s}\Big) I_{r_i} \;.
\end{equation}
Recall that
\[
\Phi = \Big( \frac{\lambda^3}{\rb^2 \tau^2 J_1} + \frac{\lambda^2}{\rb
  \tau J_2} \Big) \norm{\alpha}^2 \;.
\]
Then, for each $i=1,2,\dots,q$, we have
\begin{align}
  \label{eq:eip_bnd}
  \Big(\frac{\df e_i}{\df s}\Big)^2 = \Big(\frac{\df}{\df s}
  \frac{\lambda}{r_i \tau} \Big)^2 & = \Big(\frac{\lambda'}{r_i \tau}
  - \frac{\lambda \tau'}{r_i \tau^2} \Big)^2 \notag \\ & \le
  \frac{2(\lambda')^2}{(r_i \tau)^2} + \frac{2\lambda^2
    (\tau')^2}{r_i^2 \tau^4} \notag \\ & \le \frac{2\ell_1
    \lambda^3}{r_{\min}^2 \tau^2 J_1} \norm{\alpha}^2 + \frac{2 \ell_1
    \rb \lambda^2}{r_{\min}^2 \tau J_2} \norm{\alpha}^2 \notag \\
    &
  \le 2 m^2 \ell_1 \Big( \frac{\lambda^3}{\rb^2 \tau^2 J_1}
   + \frac{\lambda^2}{\rb \tau J_2}
  \Big) \norm{\alpha}^2 \notag \\
  & = \ell_2 \Phi \;,
\end{align}
where the second inequality follows from \eqref{eq:ls&ts}, and $\ell_2
= 2 m^2 \ell_1$.  So, using \eqref{eq:dD}, we have
\begin{equation*}
  \Big( \frac{\df D_c}{\df s} \Big)^2 = \bigoplus_{i=1}^q
  \frac{1}{e_i^4} \Big(\frac{\df e_i}{\df s}\Big)^2 I_{r_i}
  \preccurlyeq \bigoplus_{i=1}^q \frac{\ell_2 \Phi}{e_{\max}^4}
  I_{r_i} = \Big( \frac{\ell_2 \Phi}{e_{\max}^4} \Big) I_N \;.
\end{equation*}
Now recall that
$M = D_c + \frac{(I-D_c)1 1^\intercal
  (I-D_c)}{1^\intercal (I-D_c) 1}$.
Define
\[
w_1 = (I-D_c) 1 \;\;\;\; \mbox{and} \;\;\;\; w_2 = \Big( \frac{\df
  D_c}{\df s} \Big) 1 \;.
\]
Then we have
\begin{align*}
  \frac{\df}{\df s} \bigg \{ \frac{(I-D_c)1 1^\intercal
    (I-D_c)}{1^\intercal (I-D_c) 1} \bigg \}
    = & \bigg \{ \frac{(I-D_c)1
    1^\intercal (I-D_c)}{\big( 1^\intercal (I-D_c) 1 \big)^2} \bigg \}
  1^\intercal \Big( \frac{\df D_c}{\df s} \Big) 1 \\
  &  -
  \bigg \{ \frac{1}{1^\intercal (I-D_c) 1} \bigg \} \bigg \{ \Big(
  \frac{\df D_c}{\df s} \Big) 1 1^\intercal (I-D_c) + (I-D_c) 1
  1^\intercal \Big( \frac{\df D_c}{\df s} \Big) \bigg \} \;,
\end{align*}
which is equal to
\[
\bigg \{ \frac{w_1 w_1^\intercal}{(1^\intercal w_1)^2} \bigg\} \big(
1^\intercal w_2 \big) - \bigg\{ \frac{1}{1^\intercal w_1} \bigg\}
\big( w_2 w_1^\intercal + w_1 w_2^\intercal \big) \;.
\]
Therefore,
\begin{equation}
  \label{eq:n1}
  \frac{\df M}{\df s} = \frac{\df D_c}{\df s} + \bigg \{ \frac{w_1
    w_1^\intercal}{(1^\intercal w_1)^2} \bigg\} \big( 1^\intercal w_2
  \big) - \bigg\{ \frac{1}{1^\intercal w_1} \bigg\} \big( w_2
  w_1^\intercal + w_1 w_2^\intercal \big) \;.
\end{equation}
Note that each of the three terms on the left-hand side of
\eqref{eq:n1} is a symmetric matrix.  Several applications of
Lemma~\ref{lem:simple_matrix} leads to
\begin{align}
  \label{eq:dMds_bound1}
  \Big( \frac{\df M}{\df s} \Big)^2 & = \bigg( \frac{\df D_c}{\df s} +
  \bigg \{ \frac{w_1 w_1^\intercal}{(1^\intercal w_1)^2} \bigg\} \big(
  1^\intercal w_2 \big) - \bigg\{ \frac{1}{1^\intercal w_1} \bigg\}
  \big( w_2 w_1^\intercal + w_1 w_2^\intercal \big) \bigg)^2 \notag
  \\ & \preccurlyeq 2 \bigg( \frac{\df D_c}{\df s} + \bigg \{
  \frac{w_1 w_1^\intercal}{(1^\intercal w_1)^2} \bigg\} \big(
  1^\intercal w_2 \big) \bigg)^2 + 2 \bigg( \bigg\{
  \frac{1}{1^\intercal w_1} \bigg\} \big( w_2 w_1^\intercal + w_1
  w_2^\intercal \big) \bigg)^2 \notag \\ & \preccurlyeq 4 \Big(
  \frac{\df D_c}{\df s} \Big)^2 + 4 \bigg( \frac{w_1
    w_1^\intercal}{(1^\intercal w_1)^2} \big( 1^\intercal w_2 \big)
  \bigg)^2
   + \frac{4}{(1^\intercal w_1)^2} \Big[ (w_2
    w_1^\intercal)^\intercal (w_2 w_1^\intercal) + (w_1
    w_2^\intercal)^\intercal (w_1 w_2^\intercal) \Big] \notag \\ &
  \preccurlyeq 4 \Big( \frac{\df D_c}{\df s} \Big)^2 + \frac{4
    (w_1^\intercal w_1)^2}{(1^\intercal w_1)^4} \big( 1^\intercal w_2
  \big)^2 I + \frac{8}{(1^\intercal w_1)^2} (w_1^\intercal w_1)
  (w_2^\intercal w_2) I \;.
\end{align}
Now
\begin{align*}
  w_1^\intercal w_1 & = 1^\intercal (I - D_c)^2 1 = \sum_{i=1}^q
  \sum_{j=1}^{r_i} \Big(\frac{\lambda}{t_i} \Big)^2 \leq \bar{r} q
  \Big(\frac{\lambda}{t_{\min}} \Big)^2 \;, \\ w_2^\intercal w_2 & =
  1^\intercal \Big(\frac{\df D_c}{\df s}\Big)^2 1 \leq 1^\intercal
  \Big( \frac{\ell_2 \Phi}{e_{\max}^4} I_{\bar{r}q} \Big) 1 = \ell_2
  \bar{r} q \Big( \frac{\Phi }{e_{\max}^4} \Big) \;, \\ (1^\intercal
  w_1)^2 & = [1^\intercal (I - D_c) 1]^2 = \bigg( \sum_{i=1}^q
  \sum_{j=1}^{r_i} \frac{\lambda}{t_i} \bigg)^2 \geq \Big(\bar{r} q
  \frac{\lambda}{t_{\max}} \Big)^2 \;, \\ (1^\intercal w_2)^2 & \leq
  (w_2^\intercal w_2) (1^\intercal 1) \leq \ell_2 (\bar{r} q)^2 \Big(
  \frac{\Phi}{e_{\max}^4} \Big) \;.
\end{align*}
So, in conjunction with \eqref{eq:dMds_bound1}, we have
\begin{align}
  \label{eq:dMds_bound2}
  \Big( \frac{\df M}{\df s} \Big)^2
  & \preccurlyeq 4 \Big( \frac{\df
    D_c}{\df s} \Big)^2 + \frac{4 (w_1^\intercal w_1)^2}{(1^\intercal
    w_1)^4} \big( 1^\intercal w_2 \big)^2 I + \frac{8}{(1^\intercal
    w_1)^2} (w_1^\intercal w_1) (w_2^\intercal w_2) I \notag \\
  &
  \preccurlyeq \Big( \frac{4 \ell_2 \Phi}{e_{\max}^4} \Big) I + \Big(
  \frac{4 \ell_2 m^4 \Phi}{e_{\max}^4} \Big) I + \Big( \frac{8 \ell_2
    m^2 \Phi}{e_{\max}^4} \Big) I \notag \\
  & \preccurlyeq 16 \ell_2
  m^4 \Big( \frac{\Phi}{e_{\max}^4} \Big) I \;.
\end{align}
Recall that $Q^{-1} = (X^\intercal X - \bar{X}^\intercal M \bar{X})$.
We have
\begin{align*}
  \Big(\frac{\df Q^{-1}}{\df s}\Big)^2 & = \Big(- \bar{X}^\intercal
  \Big(\frac{\df M}{\df s} \Big) \bar{X} \Big)^2 \\ & =
  \bar{X}^\intercal \Big(\frac{\df M}{\df s} \Big) \bar{X}
  \bar{X}^\intercal \Big(\frac{\df M}{\df s} \Big) \bar{X} \\ &
  \preccurlyeq \bar{r} q k_2 \bar{X}^\intercal \Big(\frac{\df M}{\df
    s}\Big)^2 \bar{X} \\ & \preccurlyeq 16 \ell_2 m^4 \bar{r} q k_2
  \bar{X}^\intercal \Big( \frac{\Phi}{e_{\max}^4} I \Big) \bar{X} \\ &
  \preccurlyeq 16 \ell_2 m^4 (\bar{r} q k_2)^2 \Big(
  \frac{\Phi}{e_{\max}^4} \Big) I_p \\ & = \ell_3 (\bar{r} q)^2 \Big(
  \frac{\Phi}{e_{\max}^4} \Big) I_p \;,
\end{align*}
where the first and third inequalities follow from (B2) (and
Lemma~\ref{lem:conseq}), the second inequality follows from
\eqref{eq:dMds_bound2}, and $\ell_3 = 16 \ell_2 m^4 k_2^2$.
Recall from \eqref{eq:Qbnd}
that $Q^2 \preccurlyeq (q \rb k_1)^{-2} I$.  It follows that
\begin{align}
  \label{eq:DQs_bnd}
  \Big(\frac{\df Q}{\df s}\Big)^2 & = \Big(- Q \Big( \frac{\df
    Q^{-1}}{\df s} \Big) Q \Big)^2 \notag \\ & = Q \Big( \frac{\df
    Q^{-1}}{\df s} \Big) Q^2 \Big( \frac{\df Q^{-1}}{\df s} \Big) Q
  \notag \\ & \preccurlyeq \frac{1}{(q \rb k_1)^2} Q \Big( \frac{\df
    Q^{-1}}{\df s} \Big)^2 Q \notag \\ & \preccurlyeq
  \frac{\ell_3}{k_1^2} \Big( \frac{\Phi}{e_{\max}^4} \Big) Q^2 \notag
  \\ & \preccurlyeq \frac{\ell_3}{\rb^2 q^2 k_1^4} \Big(
  \frac{\Phi}{e_{\max}^4} \Big) I_p \;,
\end{align}
where the formula for $\df Q / \df s$ in the first line can be derived by
applying the matrix version of the product rule to the equation $Q Q^{-1}
= I$.
It's clear that $X^\intercal X - \bar{X}^\intercal M \bar{X}
\preccurlyeq X^\intercal X$, and it follows that $\lambda_{\max}
(X^\intercal X - \bar{X}^\intercal M \bar{X}) \leq \lambda_{\max}
(X^\intercal X)$.  Furthermore, we have $\lambda_{\min} \big[
  (X^\intercal X - \bar{X}^\intercal M \bar{X})^{-1} \big] = \big[
  \lambda_{\max} (X^\intercal X - \bar{X}^\intercal M \bar{X})
  \big]^{-1}$ and $\lambda_{\min} \big[ (X^\intercal X)^{-1} \big] =
\big[ \lambda_{\max} (X^\intercal X) \big]^{-1}$.  It follows that
\begin{equation}
  \label{eq:co2}
  \lambda_{\min}(Q) = \lambda_{\min} \big[ (X^\intercal X -
    \bar{X}^\intercal M \bar{X})^{-1} \big] \ge \lambda_{\min} \big[
    (X^\intercal X)^{-1} \big] \ge \frac{1}{\rb q k_2} \;.
\end{equation}
Using Lemma~\ref{lem:sr}, \eqref{eq:DQs_bnd}, and \eqref{eq:co2}, we have
\begin{equation*}
  \lambda_{\max} \bigg\{ \Big( \frac{\df Q^{\frac{1}{2}}}{\df s}
  \Big)^2 \bigg\} \le \frac{\lambda_{\max} \Big\{ \Big( \frac{\df
      Q}{\df s} \Big)^2 \Big\} }{4 \lambda_{\min}(Q)} \le \frac{\rb q
    k_2}{4} \frac{\ell_3}{\rb^2 q^2 k_1^4} \Big(
  \frac{\Phi}{e_{\max}^4} \Big) \le \Big( \frac{\ell_3 k_2}{4 \rb q
    k_1^4} \Big) \Big( \frac{\Phi}{e_{\max}^4} \Big) \;.
\end{equation*}
We are now ready to attack $T_1$, $T_2$, and $T_3$.  We have
\begin{align*}
  T_1 & = \mathrm{E} \bigg\{ \frac{q}{\tau} \Big\| \bigg( \frac{\df
    Q^{\frac{1}{2}}}{\df s} \bigg) N_{00} \Big\|^2 \bigg\}  \\ &
  = q \mathrm{E} \bigg\{ \frac{1}{\tau} N_{00}^\intercal \Big(
  \frac{\df Q^{\frac{1}{2}}}{\df s} \Big)^2 N_{00} \bigg\}  \\ &
  \le q \mathrm{E} \bigg\{ \frac{1}{\tau} \Big( \frac{\ell_3 k_2}{4
    \rb q k_1^4} \Big) \Big( \frac{\Phi}{e_{\max}^4} \Big)
  N_{00}^\intercal N_{00} \bigg\}  \\ & = \mathcal{O} \Big(
  \frac{p}{\rb} \Big) \mathrm{E} \Big( \frac{\Phi}{\tau e_{\max}^4}
  \Big) \;.
\end{align*}
Recall that $T_2 = \mathrm{E} \big \{ q \big \| \big ( \frac{\df
  Q}{\df s} \big ) X^\intercal (I - M R) Y \big \|^2 \big \}$.  So
\begin{align*}
  \Big\| \Big( \frac{\df Q}{\df s} \Big) & X^\intercal (I - M R) Y
  \Big\|^2 & \\
  & = Y^\intercal (I - M R) X \Big( \frac{\df Q}{\df s}
  \Big)^2 X^\intercal (I - M R) Y \\ & \le \Big( \frac{\ell_3}{\rb^2
    q^2 k_1^4} \Big) \Big( \frac{\Phi}{e_{\max}^4} \Big) Y^\intercal
  (I - M R) X X^\intercal (I - M R) Y \\ & \le \Big(
  \frac{\ell_3}{\rb^2 q^2 k_1^4} \Big) \Big( \frac{\Phi}{e_{\max}^4}
  \Big) Y^\intercal Y \lambda_{\max} \big[ (I - M R) X X^\intercal (I
    - M R) \big] \\ & \le \Big( \frac{\ell_3}{\rb^2 q^2 k_1^4} \Big)
  \Big( \frac{\Phi}{e_{\max}^4} \Big) (\rb q \ell) \lambda_{\max}
  \big[ X^\intercal (I - M R)^2 X \big] \\
  & \le \Big( \frac{\ell_3
    \ell}{\rb q k_1^4} \Big) \Big( \frac{\Phi}{e_{\max}^4} \Big)
  \lambda_{\max} \big[2( X^\intercal X + \bar{X}^\intercal M^2
    \bar{X}) \big] \\ & \le \Big( \frac{\ell_3 \ell}{\rb q
    k_1^4} \Big) \Big( \frac{\Phi}{e_{\max}^4} \Big) (4 \rb q k_2)
  \\ & = \mathcal{O}(1)\Big( \frac{\Phi}{e_{\max}^4} \Big) \;,
\end{align*}
where the first inequality is from \eqref{eq:DQs_bnd},
the third uses $(B_3)$ and the fact that for any matrix $M$, $M M^\intercal$ has the same non-zero eigenvalues as $M^\intercal M$ \citep[see, e.g.,][Theorem 1.3.22]{horn:john:1985},
the fourth follows from Lemma 11,
and the last is from $(B_2)$ and the fact that
$\bar{X}^\intercal M^2
\bar{X} \preccurlyeq \bar{X}^\intercal \bar{X}
\preccurlyeq X^\intercal X$ (recall from \eqref{eq:Mbound} that $M \preccurlyeq I$).

Now recall that $T_3 = \mathrm{E} \big \{ q \big\| Q X^\intercal \big(
\frac{\df M}{\df s} \big) R Y \big\|^2 \big\}$.  We have
\begin{align*}
   \Big\| Q X^\intercal \Big( \frac{\df M}{\df s} \Big) R Y \Big\|^2 =
   \bar{Y}^\intercal \Big( \frac{\df M}{\df s} \Big) X Q^2 X^\intercal
   \Big( \frac{\df M}{\df s} \Big) \bar{Y}
    \le \frac{16 \ell \ell_2
     k_2 m^4}{k_1^2} \Big( \frac{\Phi}{e_{\max}^4} \Big)
    =
   \mathcal{O}(1) \Big( \frac{\Phi}{e_{\max}^4} \Big) \;,
\end{align*}
where we have used $Q^2 \preccurlyeq (q \rb k_1)^{-2} I$, $(B_2)$,
\eqref{eq:dMds_bound2}, and $(B_3)$.  Recall that
\[
U_1 = \mathrm{E} \Big( \frac{\Phi}{e_{\max}^2} \Big) \;\;\;\;
\mbox{and} \;\;\;\; U_2 = \mathrm{E} \Big( \frac{\Phi}{\tau
  e_{\max}^4} \Big) \;.
\]
Since $e_{\max} \ge 1$, $\Phi/e_{\max}^4 \le \Phi/e_{\max}^2$, so we
have $T_1 = \mathcal{O} ( \frac{p}{\rb} ) U_2$, $T_2 = \mathcal{O}(q)
U_1$, and $T_3 = \mathcal{O}(q) U_1$.  It now follows from
\eqref{eq:eta_00} and \eqref{eq:T_0_bnd} that
\begin{align}
  \label{eq:t_eta_00_bnd}
  \mathrm{E} \Big\| \frac{\df \tilde{\eta}_{00}}{\df s} \Big\|^2
  & \le
  T_0 + 4T_1 + 4T_2 + 4T_3 \notag \\
  & = \mathcal{O} \Big( \frac{p}{\rb q} \Big)
  \norm{\alpha}^2 + \mathcal{O}(q) U_1 + \mathcal{O} \Big(
  \frac{p}{\rb} \Big) U_2 \;.
\end{align}

We now develop an upper bound for $\mathrm{E} \big( \frac{\df
  \tilde{\eta}_0}{\df s} \big)^2$, which is the second term on the
right-hand side of \eqref{eq:outline}.  Recall that $\tilde{\eta}_{00}
= v + \sqrt{ \frac{q}{\tau}} Q^{1/2} N_{00}$.  Hence,
\begin{align*}
\tilde{\eta}_0 & = \sqrt{q} \frac{\sum_{i=1}^q (\bar{y}_i -
  \bar{x}_i^\intercal \tilde{\eta}_{00} / \sqrt{q}) /
  z_i}{\sum_{i=1}^q 1/z_i} + \sqrt{\frac{q}{\sum_{i=1}^q 1/z_i}} N_0
\\ & = \sqrt{q} \frac{\sum_{i=1}^q ( \bar{y}_i - \bar{x}_i^\intercal
  v/\sqrt{q} )/z_i}{\sum_{i=1}^q 1/z_i} - \frac{\sqrt{\frac{q}{\tau}}
  \sum_{i=1}^q \bar{x}_i^\intercal Q^\frac{1}{2}
  N_{00}/z_i}{\sum_{i=1}^q 1/z_i} + \sqrt{\frac{q}{\sum_{i=1}^q
    1/z_i}} N_0 \;.
\end{align*}
Now define
\begin{align*}
  T_4 & = q \mathrm{E} \bigg[ \frac{\df}{\df s} \bigg(
    \frac{\sum_{i=1}^q (\bar{y}_i - \bar{x}_i^\intercal
      v/\sqrt{q})/z_i}{\sum_{i=1}^q 1/z_i} \bigg) \bigg]^2 \;, \\ T_5
  & = \mathrm{E} \Bigg[ \frac{\df}{\df s} \Bigg(
    \frac{\sqrt{\frac{q}{\tau}} \sum_{i=1}^q \bar{x}_i^\intercal
      Q^\frac{1}{2} N_{00}/z_i}{\sum_{i=1}^q 1/z_i} \Bigg) \Bigg]^2
  \;, \\ T_6 & = q \mathrm{E} \bigg[ \frac{\df}{\df s} \bigg(
    \sqrt{\frac{1}{\sum_{i=1}^q 1/z_i}} N_0 \bigg) \bigg]^2 \;.
\end{align*}
So, we have
\begin{equation}
  \label{eq:t_eta_0_bnd}
  \mathrm{E} \Big( \frac{\df \tilde{\eta}_0}{\df s} \Big)^2 \le 3T_4 +
  3T_5 + 3T_6 \;.
\end{equation}
We have
\begin{align*}
  T_4 & = q \mathrm{E} \bigg[ \frac{\df}{\df s} \bigg(
    \frac{\sum_{i=1}^q (\bar{y}_i - \bar{x}_i^\intercal
      v/\sqrt{q})/z_i}{\sum_{i=1}^q 1/z_i} \bigg) \bigg]^2  \\
& = q
  \mathrm{E} \bigg[ \frac{\df}{\df s} \bigg( \frac{\sum_{i=1}^q
      (\bar{y}_i - \bar{x}_i^\intercal v/\sqrt{q})/e_i}{\sum_{i=1}^q
      1/e_i} \bigg) \bigg]^2 \\
& = q \mathrm{E} \bigg[
    \frac{\sum_{i=1}^q \big[ e'_i/e^2_i \big] \sum_{i=1}^q (\bar{y}_i
      - \bar{x}_i^\intercal v/\sqrt{q})/e_i}{\big[ \sum_{i=1}^q 1/e_i
        \big]^2} - \frac{\sum_{i=1}^q \big[ e'_i/e^2_i \big]
      (\bar{y}_i - \bar{x}_i^\intercal v/\sqrt{q})}{\sum_{i=1}^q
      1/e_i}
       - \frac{\sum_{i=1}^q \bar{x}_i^\intercal \big( \frac{\df
        v}{\df s} \big)/e_i}{\sqrt{q}\sum_{i=1}^q 1/e_i} \bigg]^2 \\
& \le 3T_{41} + 3T_{42} + 3T_{43} \;,
\end{align*}
where
\begin{align*}
  T_{41} & = q \mathrm{E} \bigg[ \frac{\sum_{i=1}^q \big[ e'_i/e^2_i
        \big] \sum_{i=1}^q (\bar{y}_i - \bar{x}_i^\intercal
      v/\sqrt{q})/e_i}{\big[ \sum_{i=1}^q 1/e_i \big]^2} \bigg]^2 \;,
  \\ T_{42} & = q \mathrm{E} \Bigg[ \frac{\sum_{i=1}^q \big[
        e'_i/e^2_i \big] (\bar{y}_i - \bar{x}_i^\intercal
      v/\sqrt{q})}{\sum_{i=1}^q 1/e_i} \bigg]^2 \;, \\ T_{43} & =
  \mathrm{E} \bigg[ \frac{\sum_{i=1}^q \bar{x}_i^\intercal \big(
      \frac{\df v}{\df s} \big)/e_i}{\sum_{i=1}^q 1/e_i} \bigg]^2 \;.
\end{align*}
Using \eqref{eq:eip_bnd} and the fact that $e_{\min}/e_{\max} \le m$
yields
\begin{align*}
  T_{41}
  & = q \mathrm{E} \bigg[ \frac{\sum_{i=1}^q \big[ e'_i/e^2_i
        \big] \sum_{i=1}^q (\bar{y}_i - \bar{x}_i^\intercal
      v/\sqrt{q})/e_i}{\big[ \sum_{i=1}^q 1/e_i \big]^2} \bigg]^2
  \\
  & \le q \mathrm{E} \bigg\{ \frac{e_{\min}^4}{q^4}
  \Big(\sum_{i=1}^q \frac{e'_i}{e_i^2} \Big)^2
  \Big(\sum_{i=1}^q \frac{1}{e_i^2} \Big) \Big( \sum_{i=1}^q
  (\bar{y}_i - \bar{x}_i^\intercal v/\sqrt{q})^2 \Big) \Big \} \\
  &
  \le q \mathrm{E} \bigg\{ \frac{e_{\min}^4}{q^4} \Big(
  \frac{q^2 \ell_2 \Phi }{e_{\max}^4} \Big)  \Big(
  \frac{q}{e_{\max}^2} \Big) \bigg( \sum_{i=1}^q (\bar{y}_i -
  \bar{x}_i^\intercal v/\sqrt{q})^2 \bigg) \bigg \} \\ & =
  \mathcal{O}(1) \mathrm{E} \bigg\{ \Big( \frac{\Phi}{e_{\max}^2}
  \Big) \bigg( \sum_{i=1}^q (\bar{y}_i - \bar{x}_i^\intercal
  v/\sqrt{q})^2 \bigg) \bigg \} \;.
\end{align*}
But
\begin{align}\label{eq:sum_bnd}
  \sum_{i=1}^q (\bar{y}_i - \bar{x}_i^\intercal v/\sqrt{q})^2 \le 2
  \sum_{i=1}^q \bar{y}^2_i + \frac{2}{q} v^\intercal \Big(
  \sum_{i=1}^q \bar{x}_i \bar{x}_i^\intercal \Big) v
  \le 2
  \sum_{i=1}^q \bar{y}^2_i + \frac{2 m q k_2}{q} v^\intercal v
  = \mathcal{O}(q) \;,
\end{align}
where the second inequality follows from Remark~\ref{rem:1}, and the
final equality follows from \eqref{eq:drift_5} and Remark~\ref{rem:1}.
It follows that $T_{41} = \mathcal{O}(q) U_1$.  Similarly, we have
  \begin{align*}
  T_{42} & = q \mathrm{E} \bigg[ \frac{\sum_{i=1}^q \big[ e'_i/e^2_i
        \big] (\bar{y}_i - \bar{x}_i^\intercal
      v/\sqrt{q})}{\sum_{i=1}^q 1/e_i} \bigg]^2 \\ & \le q \mathrm{E}
  \bigg\{ \frac{e_{\min}^2}{q^2} \Big(\sum_{i=1}^q
  \frac{(e'_i)^2}{e_i^4} \Big) \Big( \sum_{i=1}^q (\bar{y}_i -
  \bar{x}_i^\intercal v/\sqrt{q})^2 \Big) \bigg \} \\ & \le q
  \mathrm{E} \bigg\{ \frac{e_{\min}^2}{q^2} \Big(\frac{q \ell_2
    \Phi}{e_{\max}^4} \Big) \Big( \sum_{i=1}^q (\bar{y}_i -
  \bar{x}_i^\intercal v/\sqrt{q})^2 \Big) \bigg \} \\ & =
  \mathcal{O}(q) U_1 \;.
\end{align*}
It follows from our work on $\big \| \frac{\df \tilde{\eta}_{00}}{\df
  s} \big \|$ that $\mathrm{E} \big \| \frac{\df v}{\df s} \big \|^2
\le 2(T_2+T_3)$.  Thus,
\begin{align*}
  T_{43} & = \mathrm{E} \bigg[ \frac{\sum_{i=1}^q \bar{x}_i^\intercal
      \big( \frac{\df v}{\df s} \big)/e_i}{\sum_{i=1}^q 1/e_i}
    \bigg]^2 \\ & \le \mathrm{E} \bigg\{ \frac{e_{\min}^2}{q^2} \Big(
  \sum_{i=1}^q \frac{1}{e_i^2} \Big) \Big(\frac{\df v}{\df
    s}\Big)^\intercal \Big( \sum_{i=1}^q \bar{x}_i \bar{x}_i^\intercal
  \Big) \Big(\frac{\df v}{\df s}\Big) \bigg\} \\
  & \le \mathrm{E}
  \bigg\{ \frac{e_{\min}^2}{q^2} \Big( \frac{q}{e_{\max}^2} \Big)
  \big(m q k_2) \Big \| \frac{\df v}{\df s} \Big \|^2 \bigg\} \\
  & \le
  m^3 k_2 \mathrm{E} \Big \| \frac{\df v}{\df s} \Big \|^2 \\ & \le 2
  m^3 k_2 (T_2+T_3) \\ & = \mathcal{O}(q) U_1 \;.
\end{align*}
We conclude that
\begin{equation}
  \label{eq:T_4_bnd}
  T_4 = \mathcal{O}(q) U_1 \;.
\end{equation}
We now move on to $T_5$.  We have
\begin{align*}
  T_5 & = \mathrm{E} \Bigg[ \frac{\df}{\df s} \Bigg(
    \frac{\sqrt{\frac{q}{\tau}} \sum_{i=1}^q \bar{x}_i^\intercal
      Q^\frac{1}{2} N_{00}/z_i}{\sum_{i=1}^q 1/z_i} \Bigg) \Bigg]^2
  \\ & = \mathrm{E} \Bigg[ \frac{\sum_{i=1}^q \big[ z'_i/z^2_i \big]
      \sqrt{\frac{q}{\tau}} \sum_{i=1}^q \bar{x}_i^\intercal
      Q^\frac{1}{2} N_{00}/z_i}{\big[ \sum_{i=1}^q 1/z_i \big]^2} -
    \frac{\sqrt{\frac{q}{\tau}} \sum_{i=1}^q \big[ z'_i/z^2_i \big]
      \bar{x}_i^\intercal Q^\frac{1}{2} N_{00}}{\sum_{i=1}^q 1/z_i}
    \\ & \hspace*{65mm} + \frac{\sum_{i=1}^q \bar{x}_i^\intercal \big(
      \frac{\df}{\df s} \sqrt{\frac{q}{\tau}} Q^\frac{1}{2} N_{00}
      \big)/z_i}{\sum_{i=1}^q 1/z_i} \Bigg]^2 \\ & \le 3T_{51} +
  3T_{52} + 3T_{53} \;,
\end{align*}
where
\begin{align*}
  T_{51} & = \mathrm{E} \Bigg[ \frac{\sum_{i=1}^q \big[ z'_i/z^2_i
        \big] \sqrt{\frac{q}{\tau}} \sum_{i=1}^q \bar{x}_i^\intercal
      Q^\frac{1}{2} N_{00}/z_i}{\big[ \sum_{i=1}^q 1/z_i \big]^2}
    \Bigg]^2 \;, \\ T_{52} & = \mathrm{E} \Bigg[
    \frac{\sqrt{\frac{q}{\tau}} \sum_{i=1}^q \big[ z'_i/z^2_i \big]
      \bar{x}_i^\intercal Q^\frac{1}{2} N_{00}}{\sum_{i=1}^q 1/z_i}
    \Bigg]^2 \;, \\ T_{53} & = \mathrm{E} \Bigg[ \frac{\sum_{i=1}^q
      \bar{x}_i^\intercal \big( \frac{\df}{\df s}
      \sqrt{\frac{q}{\tau}} Q^\frac{1}{2} N_{00}
      \big)/z_i}{\sum_{i=1}^q 1/z_i} \Bigg]^2 \;.
\end{align*}
Now,
\begin{align*}
  T_{51} & = \mathrm{E} \Bigg[ \frac{\sum_{i=1}^q \big[ z'_i/z^2_i
        \big] \sqrt{\frac{q}{\tau}} \sum_{i=1}^q \bar{x}_i^\intercal
      Q^\frac{1}{2} N_{00}/z_i}{\big[ \sum_{i=1}^q 1/z_i \big]^2}
    \Bigg]^2 \\ & \le \mathrm{E} \Bigg[ \frac{z_{\min}^4}{q^4} \Big(
    \sum_{i=1}^q \frac{1}{z_i^4} \Big) \Big( \sum_{i=1}^q (z'_i)^2
    \Big) \Big( \sum_{i=1}^q \frac{1}{z_i^2} \Big) \frac{q}{\tau}
    \sum_{i=1}^q \Big( \bar{x}_i^\intercal Q^\frac{1}{2} N_{00}
    \Big)^2 \Bigg] \\ & \le \mathrm{E} \Bigg[ \frac{z_{\min}^4}{q^4}
    \Big( \frac{q^2}{z_{\max}^6} \Big) \Big( \sum_{i=1}^q (z'_i)^2
    \Big) \frac{q}{\tau} \sum_{i=1}^q \Big( \bar{x}_i^\intercal
    Q^\frac{1}{2} N_{00} \Big)^2 \Bigg] \\ & \le \frac{m^4}{q}
  \mathrm{E} \Bigg[ \Big( \sum_{i=1}^q (z'_i)^2 \Big) \Big(
    \frac{1}{z_{\max}^2} \Big) \frac{1}{\tau} \sum_{i=1}^q \Big(
    \bar{x}_i^\intercal Q^\frac{1}{2} N_{00} \Big)^2 \Bigg] \;.
\end{align*}
But we have
\begin{align}
  \label{eq:sum2_bnd}
  \sum_{i=1}^q \Big( \bar{x}_i^\intercal Q^\frac{1}{2} N_{00} \Big)^2
  = N_{00}^\intercal Q^\frac{1}{2} \Big( \sum_{i=1}^q \bar{x}_i
  \bar{x}_i^\intercal \Big) Q^\frac{1}{2} N_{00}
  \le (m q k_2)
  N_{00}^\intercal Q N_{00}
  \le \frac{m k_2}{\rb k_1} N_{00}^\intercal
  N_{00} \;,
\end{align}
so that
\begin{align}
  \label{eq:t51bnd}
  T_{51} &  \le \frac{m^4}{q} \mathrm{E} \Bigg[
    \Big( \sum_{i=1}^q (z'_i)^2 \Big) \Big( \frac{1}{z_{\max}^2} \Big)
    \frac{1}{\tau} \frac{m k_2}{\rb k_1} N_{00}^\intercal N_{00}
    \Bigg] \notag \\ & \le \frac{m^5 k_2}{q \rb k_1} \mathrm{E} \Bigg[
    \Big( \sum_{i=1}^q (z'_i)^2 \Big) \Big( \frac{1}{z_{\max}^2} \Big)
    \frac{1}{\tau} \Bigg] \mathrm{E} \big[ N_{00}^\intercal N_{00}
    \big] \notag \\ & = \frac{m^5 k_2 p}{q \rb k_1} \mathrm{E} \Bigg[
    \Big( \sum_{i=1}^q (z'_i)^2 \Big) \Big( \frac{1}{\tau z_{\max}^2}
    \Big) \Bigg] \;.
\end{align}
Recall that $z_i = \frac{t_i}{r_i \lambda \tau} =
\frac{1}{r_i \tau} + \frac{1}{\lambda}$.
Thus,
\begin{align}
  \label{eq:dz_i}
(z_i')^2 & = \Big[\frac{\df }{\df s} \Big(\frac{1}{r_i \tau} +
    \frac{1}{\lambda} \Big) \Big]^2 \notag \\ & = \Big[-
    \frac{\tau'}{r_i \tau^2} - \frac{\lambda'}{\lambda^2} \Big]^2
  \notag \\ & \leq 2 \frac{(\tau')^2}{r_i^2 \tau^4} + 2
  \frac{(\lambda')^2}{\lambda^4} \notag \\ & \leq \frac{2 m^2 \ell_1
  }{\bar{r} \tau} \frac{1}{J_2} \|\alpha\|^2 + \frac{2 m^2 \ell_1
  }{\lambda} \frac{1}{J_1} \|\alpha\|^2 \notag \\ & = 2 m^2 \ell_1
  \Big( \frac{1}{\bar{r} \tau} \frac{1}{J_2} + \frac{1}{\lambda}
  \frac{1}{J_1} \Big) \|\alpha\|^2 \;,
\end{align}
where the second inequality follows from \eqref{eq:ls&ts}.
Now
\[
\frac{1}{\tau z_{\max}} = r_{\max} \Big(\frac{\lambda}{\lambda +
  r_{\max} \tau} \Big) \le r_{\max} \Big( \frac{\lambda}{r_{\max}
  \tau} \wedge 1 \Big) \le r_{\max} \Big( \frac{\lambda}{\rb \tau}
\wedge 1 \Big) \le r_{\max} \big( \phi \wedge 1 \big) \;,
\]
where $\phi = \lambda/(\rb \tau)$.  Thus,
\begin{align*}
T_{51} & \le \frac{m^5 k_2 p}{q \rb k_1} \mathrm{E} \Bigg[ \Big(
  \sum_{i=1}^q (z'_i)^2 \Big) \Big( \frac{1}{\tau z_{\max}^2} \Big)
  \Bigg] \\ & \le \frac{2 m^7 \ell_1 k_2 p}{\rb k_1} \mathrm{E} \bigg[
  \Big( \frac{1}{\bar{r} \tau} \frac{1}{J_2} + \frac{1}{\lambda}
  \frac{1}{J_1} \Big) \Big( \frac{r_{\max}}{z_{\max}} \big( \phi
  \wedge 1 \big) \Big) \bigg] \|\alpha\|^2 \\ & \le \frac{2 m^8 \ell_1
  k_2 p}{k_1} \mathrm{E} \bigg[ \Big( \frac{1}{\bar{r} \tau}
  \frac{1}{J_2} + \frac{1}{\lambda} \frac{1}{J_1} \Big) \Big(
  \frac{1}{z_{\max}} \big( \phi \wedge 1 \big) \Big) \bigg]
\|\alpha\|^2 \\ & \le \frac{2 m^8 \ell_1 k_2 p}{k_1} \mathrm{E} \bigg[
  \Big( \frac{m \lambda}{t_{\max}} \frac{1}{J_2} + \frac{r_{\max}
    \tau}{t_{\max}} \frac{1}{J_1} \Big) \big( \phi \wedge 1 \big)
  \bigg] \|\alpha\|^2 \\ & \le \frac{2 m^9 \ell_1 k_2 p}{k_1}
\mathrm{E} \bigg[ \Big( \frac{1}{J_1} + \frac{1}{J_2} \Big) \big( \phi
  \wedge 1 \big) \bigg] \|\alpha\|^2 \\
  & = \mathcal{O}(q) \mathrm{E}
\bigg[ \big( \phi \wedge 1 \big) \frac{1}{J_1} + \frac{1}{J_2} \bigg]
\|\alpha\|^2 \;,
\end{align*}
where the last line follows from $(B_4)$.  Recall that
\[
U_3 = \mathrm{E} \bigg[ \big( \phi \wedge 1 \big) \frac{1}{J_1} +
  \frac{1}{J_2} \bigg] \|\alpha\|^2 \;,
\]
then $T_{51} = \mathcal{O}(q) U_3$.  Now
\begin{align*}
T_{52} & = \mathrm{E} \Bigg[ \frac{\sqrt{\frac{q}{\tau}} \sum_{i=1}^q
    \big[ z'_i/z^2_i \big] \bar{x}_i^\intercal Q^\frac{1}{2}
    N_{00}}{\sum_{i=1}^q 1/z_i} \Bigg]^2 \\ & \le \mathrm{E} \Bigg[
  \frac{z_{\min}^2}{\tau q} \Big( \sum_{i=1}^q \frac{1}{z_i^4} \Big)
  \Big( \sum_{i=1}^q (z'_i)^2 \Big) \sum_{i=1}^q \Big(
  \bar{x}_i^\intercal Q^\frac{1}{2} N_{00} \Big)^2 \Bigg] \\ & \le
\mathrm{E} \Bigg[ \frac{z_{\min}^2}{\tau q} \Big( \frac{q}{z_{\max}^4}
  \Big) \Big( \sum_{i=1}^q (z'_i)^2 \Big) \frac{m k_2}{\rb k_1} \Bigg]
\mathrm{E} \big[ N_{00}^\intercal N_{00} \big] \\ & \le \frac{m^3 k_2
  p}{\rb k_1} \mathrm{E} \bigg[ \Big( \frac{1}{\tau z_{\max}^2} \Big)
  \Big( \sum_{i=1}^q (z'_i)^2 \Big) \bigg]  \;,
\end{align*}
but this is $m^2 q$ times \eqref{eq:t51bnd}, so $T_{52} = \mathcal{O}(q^2)
U_3$.  Continuing, we have
\begin{align*}
  T_{53} & = \mathrm{E} \Bigg[ \frac{\sum_{i=1}^q \bar{x}_i^\intercal
      \big( \frac{\df}{\df s} \sqrt{\frac{q}{\tau}} Q^\frac{1}{2}
      N_{00} \big)/z_i}{\sum_{i=1}^q 1/z_i} \Bigg]^2 \\ & \le
  \mathrm{E} \Bigg[ \frac{z_{\min}^2}{q^2} \frac{q}{z_{\max}^2}
    \sum_{i=1}^q \Big[ \bar{x}_i^\intercal \Big( \frac{\df}{\df s}
      \sqrt{\frac{q}{\tau}} Q^\frac{1}{2} N_{00} \Big) \Big]^2 \Bigg]
  \\ & \le \frac{m^2}{q} \mathrm{E} \Bigg[ \Big( \frac{\df}{\df s}
    \sqrt{\frac{q}{\tau}} Q^\frac{1}{2} N_{00} \Big)^\intercal \Big(
    \sum_{i=1}^q \bar{x}_i \bar{x}_i^\intercal \Big) \Big(
    \frac{\df}{\df s} \sqrt{\frac{q}{\tau}} Q^\frac{1}{2} N_{00} \Big)
    \Bigg] \\ & \le m^3 k_2 \mathrm{E} \bigg \| \frac{\df}{\df s}
  \sqrt{\frac{q}{\tau}} Q^\frac{1}{2} N_{00} \bigg \|^2 \\ & \le 2 m^3
  k_2 \bigg \{ \mathrm{E} \bigg[ \frac{q (\tau')^2}{\tau^3} \|
    Q^\frac{1}{2} N_{00} \|^2 \bigg] + \mathrm{E} \bigg[
    \frac{q}{\tau} \Big \| \Big( \frac{\df Q^\frac{1}{2}}{\df s} \Big)
    N_{00} \Big \|^2 \bigg] \bigg \} \\ & = 2 m^3 k_2 (T_0 + T_1) \\ &
  = \mathcal{O} \Big( \frac{p}{\rb q} \Big) \norm{\alpha}^2 +
  \mathcal{O} \Big( \frac{p}{\rb} \Big) U_2 \;.
\end{align*}
Combining our bounds on $T_{51}$, $T_{52}$, and $T_{53}$ yields
\begin{align}
  \label{eq:T_5_bnd}
  T_5 & \le 3T_{51} + 3T_{52} + 3T_{53} \notag \\
  & = \mathcal{O}(q^2) U_3 + \mathcal{O} \Big( \frac{p}{\rb q}
  \Big) \norm{\alpha}^2 + \mathcal{O} \Big( \frac{p}{\rb} \Big) U_2
  \;.
\end{align}
We now move on to $T_6$.  We have
\begin{align}
  \label{eq:T_6_bnd}
  T_6 & = q \mathrm{E} \bigg[ \frac{\df}{\df s} \bigg(
    \sqrt{\frac{1}{\sum_{i=1}^q 1/z_i}} N_0 \bigg) \bigg]^2 \notag
  \\ & = \frac{q}{4} \mathrm{E} \Bigg[ \Big( \sum_{i=1}^q
    \frac{1}{z_i} \Big)^{-3/2} \Big( \sum_{i=1}^q \frac{z_i'}{z_i^2}
    \Big) N_0 \Bigg]^2 \notag \\ & \le \frac{q}{4} \mathrm{E} \Bigg[
    \Big( \sum_{i=1}^q \frac{1}{z_i} \Big)^{-3} \Big( \sum_{i=1}^q
    (z_i')^2 \Big) \Big( \sum_{i=1}^q \frac{1}{z_i^4} \Big) \Bigg]
  \mathrm{E} \big[ N_0^2 \big] \notag \\ & \le \frac{q}{4} \mathrm{E}
  \bigg[ \Big( \frac{z_{\min}^3}{q^3} \Big) \Big( \frac{q}{z_{\max}^4}
    \Big) \big(2 q m^2 \ell_1 \big) \Big( \frac{1}{\bar{r} \tau}
    \frac{1}{J_2} + \frac{1}{\lambda} \frac{1}{J_1} \Big) \bigg]
  \|\alpha\|^2 \notag \\ & \le \frac{m^5 \ell_1}{2} \mathrm{E} \bigg[
    \Big( \frac{1}{z_{\max}} \Big) \Big( \frac{1}{\bar{r} \tau}
    \frac{1}{J_2} + \frac{1}{\lambda} \frac{1}{J_1} \Big) \bigg]
  \|\alpha\|^2 \notag \\ & \le \frac{m^5 \ell_1}{2} \mathrm{E} \bigg[
    \Big( \frac{m \lambda}{t_{\max}} \frac{1}{J_2} + \frac{r_{\max}
      \tau}{t_{\max}} \frac{1}{J_1} \Big) \bigg] \|\alpha\|^2 \notag
  \\ & \le \frac{m^5 \ell_1}{2} \mathrm{E} \Big( \frac{m}{J_2} +
  \frac{1}{J_1} \Big) \|\alpha\|^2 \notag \\ & = \mathcal{O}
  \Big( \frac{1}{q} \Big) \|\alpha\|^2 \;,
\end{align}
where the second inequality follows from \eqref{eq:dz_i}, and the last
line holds because $\mathrm{E}(J_1^{-1}) = (q/2+a_1-1)^{-1}$ and
$\mathrm{E}(J_2^{-1}) = (q\rb/2+a_2-1)^{-1}$.  Combining
\eqref{eq:t_eta_0_bnd}, \eqref{eq:T_4_bnd}, \eqref{eq:T_5_bnd}, and
\eqref{eq:T_6_bnd}, we have
\begin{align}
  \label{eq:t_eta_0_bnd2}
  \mathrm{E} \Big( \frac{\df \tilde{\eta}_0}{\df s} \Big)^2 & \le 3T_4
  + 3T_5 + 3T_6 \notag \\ & = \mathcal{O}(q) U_1 + \mathcal{O}(q^2)
  U_3 + \mathcal{O} \Big( \frac{p}{\rb q} \Big) \norm{\alpha}^2 +
  \mathcal{O} \Big( \frac{p}{\rb} \Big) U_2 + \mathcal{O} \Big(
  \frac{1}{q} \Big) \|\alpha\|^2
\end{align}

We now go to work on $\sum_{i=1}^q \mathrm{E} \big( \frac{\df
  \tilde{\eta}_i}{\df s} \big)^2$, which is the third term on the
right-hand side of \eqref{eq:outline}.  First,
\begin{align*}
  \tilde{\eta}_i & = \frac{\lambda}{t_i} \tilde{\eta}_0/\sqrt{q} +
  \frac{r_i \tau}{t_i} (\bar{y}_i - \bar{x}_i^\intercal
  \tilde{\eta}_{00}/\sqrt{q}) + \sqrt{\frac{1}{t_i}} N_i \\ & = \Big(
  1 - \frac{1}{e_i} \Big) \tilde{\eta}_0/\sqrt{q} + \frac{1}{e_i}
  (\bar{y}_i - \bar{x}_i^\intercal \tilde{\eta}_{00}/\sqrt{q}) +
  \sqrt{\frac{1}{t_i}} N_i \;.
\end{align*}
Now using \eqref{eq:eip_bnd},
\begin{align*}
  & \Big( \frac{\df \tilde{\eta}_i }{\df s} \Big)^2 \\
= & \bigg\{
  \frac{e_i'}{e_i^2} \tilde{\eta}_0/\sqrt{q} + \Big( 1 - \frac{1}{e_i}
  \Big) \frac{\df \tilde{\eta}_0 }{\df s}/\sqrt{q} -
  \frac{e_i'}{e_i^2} (\bar{y}_i - \bar{x}_i^\intercal
  \tilde{\eta}_{00}/\sqrt{q}) - \frac{1}{e_i} \bar{x}_i^\intercal
  \frac{\df \tilde{\eta}_{00}}{\df s}/\sqrt{q} - \frac{1}{2}
  t_i^{-\frac{3}{2}} t_i' N_i \bigg\}^2
\\ \le &
  \frac{5}{q} \frac{(e_i')^2 }{e_i^4} \tilde{\eta}_0^2 + \frac{5}{q}
  \Big( 1 - \frac{1}{e_i} \Big)^2 \Big( \frac{\df \tilde{\eta}_0 }{\df
    s} \Big)^2 + 5 \frac{(e_i')^2 }{e_i^4} (\bar{y}_i -
  \bar{x}_i^\intercal \tilde{\eta}_{00}/\sqrt{q})^2 + \frac{5}{q}
  \Big(\frac{1}{e_i} \Big)^2 \Big( \bar{x}_i^\intercal \frac{\df
    \tilde{\eta}_{00}}{\df s} \Big)^2
    + \frac{5}{4} t_i^{-3} (t_i')^2
  N_i^2
\\ \le & \frac{5}{q} \frac{\ell_2
    \Phi}{e_{\max}^4} \tilde{\eta}_0^2 + \frac{5}{q} \Big( \frac{\df
    \tilde{\eta}_0 }{\df s} \Big)^2 + 5 \frac{\ell_2 \Phi}{e_{\max}^4}
  (\bar{y}_i - \bar{x}_i^\intercal \tilde{\eta}_{00}/\sqrt{q})^2 +
  \frac{5}{q} \Big( \bar{x}_i^\intercal \frac{\df
    \tilde{\eta}_{00}}{\df s} \Big)^2 + \frac{5}{4} t_i^{-3} (t_i')^2
  N_i^2 \;.
\end{align*}
We now bound $\sum_{i=1}^q \mathrm{E} \big( \frac{\df
  \tilde{\eta}_i}{\df s} \big)^2$. First,
\begin{align*}
\frac{5 \ell_2 \Phi}{e_{\max}^4} \tilde{\eta}_0^2 & = \frac{5 \ell_2
  \Phi}{e_{\max}^4} \Big( \sqrt{q} \frac{\sum_{i=1}^q (\bar{y}_i -
  \bar{x}_i^\intercal \tilde{\eta}_{00} / \sqrt{q}) /
  z_i}{\sum_{i=1}^q 1/z_i} + \sqrt{\frac{q}{\sum_{i=1}^q 1/z_i}} N_0
\Big)^2 \\ & \le \frac{10 \ell_2 \Phi}{e_{\max}^4} \bigg( \frac{q
  \big[ \sum_{i=1}^q (\bar{y}_i - \bar{x}_i^\intercal
    \tilde{\eta}_{00} / \sqrt{q})/z_i \big]^2}{\big[ \sum_{i=1}^q
    1/z_i \big]^2} + \frac{q}{\sum_{i=1}^q 1/z_i} N_0^2 \bigg) \\ &
\le \frac{10 \ell_2 \Phi}{e_{\max}^4} \bigg( \frac{q \sum_{i=1}^q
  (\bar{y}_i - \bar{x}_i^\intercal \tilde{\eta}_{00}/\sqrt{q})^2
  \sum_{i=1}^q (1/z_i^2)}{\big[ q/z_{\min} \big]^2} + z_{\min} N_0^2
\bigg) \\ & \le \frac{10 \ell_2 m^2 \Phi}{e_{\max}^4} \sum_{i=1}^q
(\bar{y}_i - \bar{x}_i^\intercal \tilde{\eta}_{00}/\sqrt{q})^2 +
\frac{10 \ell_2 \Phi}{e_{\max}^4} z_{\min} N_0^2 \;,
\end{align*}
but
\begin{align*}
\sum_{i=1}^q (\bar{y}_i - \bar{x}_i^\intercal
\tilde{\eta}_{00}/\sqrt{q})^2 & = \sum_{i=1}^q \bigg\{ \bar{y}_i -
\bar{x}_i^\intercal \Big( v + \sqrt{\frac{q}{\tau}} Q^{\frac{1}{2}}
N_{00} \Big)/\sqrt{q} \bigg\}^2 \\ & \le 2 \sum_{i=1}^q \Big(
\bar{y}_i - \bar{x}_i^\intercal v/\sqrt{q} \Big)^2 + \frac{2}{q}
\sum_{i=1}^q \Big( \sqrt{\frac{q}{\tau}} \bar{x}_i^\intercal
Q^{\frac{1}{2}} N_{00} \Big)^2 \\ & = \mathcal{O}(q) + \mathcal{O}
\Big( \frac{1}{\rb} \Big) \frac{N_{00}^\intercal N_{00}}{\tau} \;,
\end{align*}
where the last line follows from \eqref{eq:sum_bnd} and
\eqref{eq:sum2_bnd}.  So finally,
\begin{align*}
  \frac{5 \ell_2 \Phi}{e_{\max}^4} \tilde{\eta}_0^2 & = \mathcal{O}(q)
  \frac{\Phi}{e_{\max}^4} + \mathcal{O} \Big( \frac{1}{\rb} \Big)
  \frac{\Phi}{\tau e_{\max}^4} N_{00}^\intercal N_{00} +
  \mathcal{O}(1) \frac{\Phi}{\lambda e_{\max}^3} N_0^2 \;,
\end{align*}
and we have used the fact that $z_{\min}/e_{\max} \le m/\lambda$.
Now,
\[
\frac{5 \ell_2 \Phi}{e_{\max}^4} \sum_{i=1}^q (\bar{y}_i -
\bar{x}_i^\intercal \tilde{\eta}_{00}/\sqrt{q})^2 = \mathcal{O}(q)
\frac{\Phi}{e_{\max}^4} + \mathcal{O} \Big( \frac{1}{\rb} \Big)
\frac{\Phi}{\tau e_{\max}^4} N_{00}^\intercal N_{00} \;,
\]
and
\begin{align*}
\frac{5}{q} \sum_{i=1}^q \Big( \bar{x}_i^\intercal \frac{\df
  \tilde{\eta}_{00}}{\df s} \Big)^2 = \frac{5}{q} \Big( \frac{\df
  \tilde{\eta}_{00}}{\df s} \Big)^\intercal \Big( \sum_{i=1}^q
\bar{x}_i \bar{x}_i^\intercal \Big) \Big( \frac{\df
  \tilde{\eta}_{00}}{\df s} \Big) = \mathcal{O}(1) \Big \| \frac{\df
  \tilde{\eta}_{00}}{\df s} \Big \|^2 \;.
\end{align*}
Finally, we have
\begin{align*}
\frac{5}{4} \sum_{i=1}^q t_i^{-3} (t_i')^2 N_i^2 & = \frac{5}{4}
\sum_{i=1}^q \Big(\frac{1}{r_i \tau + \lambda}\Big)^3 \big( \lambda' +
r_i \tau' \big)^2 N_i^2 \\ & \le \frac{5}{2} \Big(\frac{1}{r_{\min}
  \tau + \lambda}\Big)^3 \big( (\lambda')^2 + r_{\max}^2 (\tau')^2
\big) \sum_{i=1}^q N_i^2 \\ & \le \frac{5 m^3}{2(\rb \tau +
  \lambda)^3} \bigg( \frac{\ell_1 \lambda^3}{J_1} \norm{\alpha}^2 +
\frac{m^2l_1 \rb^3 \tau^3}{J_2} \norm{\alpha}^2 \bigg) \sum_{i=1}^q
N_i^2 \\ & \le \frac{5 m^5 \ell_1}{2} \bigg( \frac{\lambda^3}{(\rb
  \tau + \lambda)^3} \frac{1}{J_1} + \frac{(\rb \tau)^3}{(\rb \tau +
  \lambda)^3} \frac{1}{J_2} \bigg) \norm{\alpha}^2 \sum_{i=1}^q N_i^2
\\ & \le \frac{5 m^5 \ell_1}{2} \bigg( \big[ \phi \wedge 1 \big]
\frac{1}{J_1} + \frac{1}{J_2} \bigg) \norm{\alpha}^2 \sum_{i=1}^q
N_i^2 \;,
\end{align*}
where the second inequality follows from \eqref{eq:ls&ts} and the
fourth follows from the fact that $\phi = \lambda/(\rb \tau)$.
Combining all of these bounds yields
\begin{align*}
\sum_{i=1}^q \Big( \frac{\df \tilde{\eta}_i}{\df s} \Big)^2 =
\mathcal{O}(q) \frac{\Phi}{e_{\max}^4} + \mathcal{O} \Big(
\frac{1}{\rb} \Big) \frac{\Phi}{\tau e_{\max}^4} N_{00}^\intercal
N_{00} + \mathcal{O}(1) \frac{\Phi}{\lambda e_{\max}^3} N_0^2 +
\mathcal{O}(1) \Big( \frac{\df \tilde{\eta}_0 }{\df s} \Big)^2
\\  + \mathcal{O}(1) \Big \| \frac{\df
  \tilde{\eta}_{00}}{\df s} \Big \|^2 + \mathcal{O}(1) \bigg( \big[
  \phi \wedge 1 \big] \frac{1}{J_1} + \frac{1}{J_2} \bigg)
\norm{\alpha}^2 \sum_{i=1}^q N_i^2 \;.
\end{align*}
Recall that
\[
U_4 = \mathrm{E} \Big( \frac{\Phi}{\lambda e_{\max}^3} \Big) \;,
\]
and that
\[
U_1 = \mathrm{E} \Big( \frac{\Phi}{e_{\max}^2} \Big) \;, \;\;\; U_2 =
\mathrm{E} \Big( \frac{\Phi}{\tau e_{\max}^4} \Big) \;, \;\;
\mbox{and} \;\;\; U_3 = \mathrm{E} \bigg[ \big( \phi \wedge 1 \big)
  \frac{1}{J_1} + \frac{1}{J_2} \bigg] \|\alpha\|^2 \;.
\]
So, finally, we have
\begin{align*}
& \sum_{i=1}^q  \mathrm{E} \Big( \frac{\df \tilde{\eta}_i}{\df s}
  \Big)^2 \\
& = \mathcal{O}(q) U_1 + \mathcal{O} \Big( \frac{p}{\rb} \Big)
  U_2 + \mathcal{O}(q) U_3 + \mathcal{O}(1) U_4 + \mathcal{O}(1)
  \mathrm{E} \Big( \frac{\df \tilde{\eta}_0 }{\df s} \Big)^2 +
  \mathcal{O}(1) \mathrm{E} \Big \| \frac{\df \tilde{\eta}_{00}}{\df
    s} \Big \|^2 \;.
\end{align*}
Combining this with \eqref{eq:outline}, \eqref{eq:t_eta_00_bnd}, and
\eqref{eq:t_eta_0_bnd2} we have
\begin{align*}
  & \Bigg\{ \mathrm{E} \bigg\| \frac{\df f(\eta + s \alpha)}{\df s}
  \bigg\| \Bigg\}^2 \notag \\
  \leq & \mathrm{E} \Big\| \frac{\df
    \tilde{\eta}_{00}}{\df s} \Big\|^2 + \mathrm{E} \Big( \frac{\df
    \tilde{\eta}_0}{\df s} \Big)^2 + \sum_{i=1}^q \mathrm{E} \Big(
  \frac{\df \tilde{\eta}_i}{\df s} \Big)^2 \notag \\
  = &
  \mathcal{O}(q) U_1 + \mathcal{O} \Big( \frac{p}{\rb}
  \Big) U_2 + \mathcal{O}(q) U_3 + \mathcal{O}(1) U_4 + \mathcal{O}(1)
  \mathrm{E} \Big( \frac{\df \tilde{\eta}_0 }{\df s} \Big)^2 +
  \mathcal{O}(1) \mathrm{E} \Big \| \frac{\df \tilde{\eta}_{00}}{\df
    s} \Big \|^2 \notag \\ = & \mathcal{O}(q) U_1 + \mathcal{O} \Big(
  \frac{p}{\rb} \Big) U_2 + \mathcal{O}(q^2) U_3 + \mathcal{O}(1) U_4
  + \mathcal{O} \Big( \frac{p}{q \rb} \Big) \norm{\alpha}^2 +
  \mathcal{O} \Big( \frac{1}{q} \Big) \norm{\alpha}^2 \notag \\ = &
  \mathcal{O}(q) U_1 + \mathcal{O} \Big( \frac{p}{\rb} \Big) U_2 +
  \mathcal{O}(q^2) U_3 + \mathcal{O}(1) U_4 + \mathcal{O} \Big(
  \frac{1}{q} \Big) \norm{\alpha}^2 \;,
\end{align*}
where the last line follows from $(B_4)$ and $(B_5)$.  This is
\eqref{eq:outline2}.



\bibliographystyle{ims}
\bibliography{refs}


\end{document}